\documentclass[10pt,a4paper]{article}
\pdfoutput=1 

\usepackage{lineno,hyperref}
\modulolinenumbers[5]

\usepackage[english]{babel}
\usepackage{graphicx}
\usepackage{marginnote}

\usepackage{tikz}
\usetikzlibrary{patterns}
\usepackage{bbm}
\usetikzlibrary{arrows}
\usepackage{mathtools}
 \usepackage{multirow}
\usepackage{amssymb}
\usepackage{amsfonts}
\usepackage{amsthm}
\usepackage{subfig}
\usepackage[scr]{rsfso} 
\usepackage{algorithm}     
\usepackage{algpseudocode} 

\newcommand{\email}[1]{\hspace*{\stretch{1}}\emph{\texttt{#1}}}

\makeatletter
\def\blfootnote{\xdef\@thefnmark{$\star$}\@footnotetext}
\makeatother
\newenvironment{Authors}%
  {\begin{center}\begin{bfseries}}%
  {\end{bfseries}\end{center}}
\newenvironment{Addresses}%
  {\begin{flushleft}\begin{itshape}}%
  {\end{itshape}\end{flushleft}}

 \usepackage{fancyhdr}  
  
\fancypagestyle{plain}{
	\fancyhead{}
	\fancyhead[C]{\hfill submitted to SIAM Journal on Scientific Computing (SISC), June 2019}
}
  
\newtheorem{theorem}{Theorem}[section]
\newtheorem{theorem1}{Theorem}[section]
\newtheorem{theorem2}{Theorem}[section]
\newtheorem{proposition}[theorem]{Proposition}
\newtheorem{lemma}[theorem1]{Lemma}
\newtheorem{remark}[theorem2]{Remark}
\newtheorem{corollary}[theorem]{Corollary}

  \newcommand{\vertiii}[1]{{\left\vert\kern-0.25ex\left\vert\kern-0.25ex\left\vert #1 
    \right\vert\kern-0.25ex\right\vert\kern-0.25ex\right\vert}}

\begin{document}

\thispagestyle{plain}

\title{A registration method for model order reduction: data compression and geometry reduction}
 \date{}
 
 \maketitle
\vspace{-50pt} 
 
\begin{Authors}
Tommaso Taddei$^{1}$
\end{Authors}

\begin{Addresses}
$^1$
IMB, UMR 5251, Univ. Bordeaux;  33400, Talence, France.
Inria Bordeaux Sud-Ouest, Team MEMPHIS;  33400, Talence, France, \email{tommaso.taddei@inria.fr} \\
\end{Addresses}

\begin{abstract}
 We propose a general --- i.e., independent of the underlying equation --- registration method for parameterized model order reduction. 
{
Given the spatial domain $\Omega \subset \mathbb{R}^d$ and the manifold
$\mathcal{M}_{\rm u}= \{ u_{\mu} : \mu \in \mathcal{P}  \}$ associated with the parameter domain $\mathcal{P} \subset \mathbb{R}^P$ and the parametric field
$\mu \mapsto u_{\mu} \in   L^2(\Omega)$,
the algorithm takes as input a  set of snapshots $\{ u^k \}_{k=1}^{n_{\rm train}} \subset \mathcal{M}_{\rm u}$ and returns a parameter-dependent bijective mapping 
$\boldsymbol{\Phi}: \Omega \times \mathcal{P} \to \mathbb{R}^d$:
the mapping is designed to make the mapped manifold $\{ u_{\mu} \circ \boldsymbol{\Phi}_{\mu}: \, \mu \in \mathcal{P} \}$ more suited for linear compression methods. }
We apply the registration procedure, in combination with a linear compression method, to devise low-dimensional representations of solution manifolds with slowly-decaying Kolmogorov $N$-widths; we also consider the application to problems in parameterized geometries. We present a  theoretical result to show the  mathematical rigor of  the registration procedure. We further present numerical results for several two-dimensional problems, to empirically demonstrate the effectivity of our proposal.
\end{abstract}


\section{Introduction}
\label{sec:intro}
 
 \subsection{Background}
Numerical simulations based on mathematical models represent a valuable tool to study complex phenomena of scientific interest and/or industrial value. For several applications --- including design and optimization, uncertainty quantification, active control --- it is important to approximate the solution (and associated quantities of interest) to the model over a range of parameters: parameters might correspond to material properties, geometric features, operating conditions.
To alleviate the computational burden associated with the evaluation of the model for many values of the parameters, parameterized  {m}odel {o}rder {r}eduction (pMOR) techniques aim to generate a {r}educed-{o}rder {m}odel (ROM) that approximates the original system over a prescribed parameter range. In this work, we shall 
develop a \emph{general} --- i.e., independent of the underlying equation --- registration procedure for pMOR applications; we shall here focus on systems modelled by stationary 
partial differential  equations (PDEs).
{In computer vision and pattern recognition, registration refers to the process of finding a spatial transformation that aligns two datasets; in this paper, registration refers to the process of finding a parametric transformation that improves the linear compressibility of a given parametric manifold.
}

We denote by $\mu$  the set of parameters associated with the model in the prescribed parameter domain $\mathcal{P} \subset \mathbb{R}^P$; we denote by $\Omega \subset \mathbb{R}^d$  ($d=2,3$) the spatial domain, which is assumed to be Lipschitz; we also introduce the Hilbert space $\mathcal{X}$  defined over $\Omega$, endowed with the inner product $(\cdot, \cdot)$   and the induced norm  $\|  \cdot \|:= \sqrt{(\cdot, \cdot)}$.
Then, we introduce the problem:
\begin{equation}
\label{eq:standard_PMOR_setting}
{\rm find} \; \mathbf{y}_{\mu}:= \mathfrak{s}_{\mu} (z_{\mu}) \;\;
{\rm s.t. } \; \;
\mathcal{G}_{\mu}(z_{\mu}, v) = 0 \; \;  \forall \, v \in   \mathcal{X},
\end{equation}
where $\mathbf{y}_{\mu}$ is a vector of $D$ quantities of interest,
$\mathfrak{s}_{\mu}: \mathcal{X} \to \mathbb{R}^D$ is a continuous functional, and 
$\mathcal{G}_{\mu}(z_{\mu}, \cdot) = 0$ admits a unique solution for all $\mu \in \mathcal{P}$. 
We denote by $\mathcal{M}:= \{ z_{\mu}: \, \mu \in \mathcal{P}  \} \subset \mathcal{X}$ the solution manifold associated with the parametric problem, 
 and we denote by $\widehat{z}_{\mu}$ the approximation to $z_{\mu}$ provided by a given ROM, for $\mu \in \mathcal{P}$.

A pMOR technique for \eqref{eq:standard_PMOR_setting}
relies on three building blocks:
(i) a data compression strategy for the construction of a low-dimensional approximation space of the solution manifold $\mathcal{M}$;
(ii) a reduced-order statement  for the rapid and reliable prediction of the solution and associated quantities of interest for any value of the parameter; and
 (iii) an \emph{a posteriori} error estimation procedure for certification.
 pMOR strategies rely on an \emph{offline/online} computational decomposition: during the offline stage, which is computationally expensive and performed once, a ROM for \eqref{eq:standard_PMOR_setting}  is generated by exploiting several high-fidelity (finite element, finite volume,...) solutions to the mathematical model for (properly-chosen) parameter values; during the online stage, which is inexpensive and performed for any new parameter  value  $\mu$, the solution to the ROM is computed to rapidly obtain predictions of $z_{\mu}$ and associated quantities of interest.
In the past few decades, many authors have developed pMOR strategies for a broad class of problems: we refer to the surveys \cite{benner2015survey,hesthaven2016certified,quarteroni2015reduced}
for thorough introductions to pMOR. 

In \eqref{eq:standard_PMOR_setting}, we assume that the domain $\Omega$ is fixed (parameter-independent); however, for several applications, the problem of interest is of the form:
\begin{equation}
\label{eq:standard_PMOR_setting_param_domain}
{\rm find} \; \mathbf{y}_{\mu} := \mathfrak{s}_{\mu} (z_{\mu}) \;\;
{\rm s.t. } \; \;
\mathcal{G}_{\mu}(z_{\mu}, v) = 0 \; \;  \forall \, v \in   \mathcal{X}_{\mu},
\end{equation}
where $\mathcal{X}_{\mu}$ is a suitable Hilbert space defined over the parameter-dependent domain $\Omega_{\mu} \subset \mathbb{R}^d$. 
{Given $\mu \in \mathcal{P}$, we remark that the solution field $z_{\mu}$ is defined over the parameter-dependent domain $\Omega_{\mu}$: since pMOR procedures rely on the definition of a low-dimensional approximation  defined over a parameter-independent domain, 
they cannot be directly applied to \eqref{eq:standard_PMOR_setting_param_domain}.
 We should thus recast  the problem in a fixed domain.}

Parametric mappings are used in pMOR to recast problem \eqref{eq:standard_PMOR_setting} or \eqref{eq:standard_PMOR_setting_param_domain} as
\begin{equation}
\label{eq:mapped_PMOR_setting}
{\rm find} \; \mathbf{y}_{\mu} = \mathfrak{s}_{\mu,\Phi} (\tilde{z}_{\mu}) \;\;
{\rm s.t. } \; \;
\mathcal{G}_{\mu,\Phi}(\tilde{z}_{\mu}, v) = 0 \; \;  \forall \, v \in   \mathcal{X},
\end{equation}
where $\mathfrak{s}_{\mu,\Phi} (w ) :=
\mathfrak{s}_{\mu} (w\circ \boldsymbol{\Phi}_{\mu}^{-1} ) $,
and $\mathcal{G}_{\mu,\Phi}(w, v) :=
\mathcal{G}_{\mu}(w\circ \boldsymbol{\Phi}_{\mu}^{-1}, v\circ \boldsymbol{\Phi}_{\mu}^{-1})$ for all $w,v \in \mathcal{X}$. 
We denote by 
 $\widetilde{\mathcal{M}}:=  \{ \tilde{z}_{\mu} = z_{\mu} \circ \boldsymbol{\Phi}_{\mu}: \, \mu \in \mathcal{P}   \}$ the mapped solution manifold.
 For 
\eqref{eq:standard_PMOR_setting}, we require that 
$\boldsymbol{\Phi}_{\mu}$ is a bijection from $\Omega$ into itself, for all $\mu \in \mathcal{P}$; furthermore, we require that 
$v \circ \boldsymbol{\Phi}_{\mu}^{-1} \in \mathcal{X}$ for all 
$v \in \mathcal{X}$ ---{if $\mathcal{X}$ is either $L^2(\Omega)$ or $H^1(\Omega)$, it suffices to require that $\boldsymbol{\Phi}$ and its inverse are Lipschitz continuous for all $\mu \in \mathcal{P}$.}
For  \eqref{eq:standard_PMOR_setting_param_domain}, 
we require that 
$\boldsymbol{\Phi}_{\mu}$ is a bijection from a reference domain ${\Omega}$ into $\Omega_{\mu}$ and that 
$\mathcal{G}_{\mu,\Phi}$ is well-posed in
$\mathcal{X} := \{ v \circ \boldsymbol{\Phi}_{\mu}: \, v \in \mathcal{X}_{\mu} \}$.
{As discussed in section 
\ref{sec:data_compression_intro}, the use of parameterized mappings for problems of the form
\eqref{eq:standard_PMOR_setting} 
 is motivated by approximation considerations and is not strictly necessary; on the other hand, the use of mappings for problems of the form \eqref{eq:standard_PMOR_setting_param_domain} 
serves to restate the mathematical problem in a parameter-independent spatial domain, and is thus essential for the application of pMOR procedures.
}

Problem-dependent mappings are broadly used in applied mathematics to compress 
information and ultimately simplify the solution to PDEs. 
$r$-adaptivity  (\cite{zahr2018optimization}) is used in combination with discontinuous Galerkin  (DG) discretization methods to approximate discontinuous solutions to hyperbolic PDEs; similarly, problem-dependent scaling factors --- in effect, mappings --- are used to determine self-similar (approximate) solutions to PDEs, including the boundary layer equations (see, e.g., \cite[Chapter 16.4]{panton2013incompressible}).  
Furthermore, problem-dependent mappings are employed 
to devise high-fidelity solvers for problems in deforming domains (\cite{de2007mesh,persson2009discontinuous}). 
We describe below the two classes of pMOR tasks considered in this work.

\subsection{Data compression of problems with slowly-decaying Kolmogorov widths}
\label{sec:data_compression_intro}
Most approaches to data compression for
\eqref{eq:standard_PMOR_setting}
rely on linear approximation spaces $\mathcal{Z}_N= {\rm span}\{ \zeta_n \}_{n=1}^N \subset \mathcal{X}$ to approximate the manifold $\mathcal{M}$: we refer to this class of methods as to \emph{linear approximation (or compression) methods}; given $\mathcal{Z}_N$, we denote by $Z_N: \mathbb{R}^N \to \mathcal{Z}_N$ the linear parameter-independent operator such that 
$Z_N \boldsymbol{\alpha}:= \sum_{n=1}^N \alpha_n \,  \zeta_n$, for all
$\boldsymbol{\alpha} \in \mathbb{R}^N$. Two well-known linear approximation methods are the Proper Orthogonal Decomposition (POD,
\cite{berkooz1993proper,sirovich1987turbulence,volkwein2011model}) and the weak-Greedy algorithm (\cite[section 7.2.2]{rozza2007reduced}).
Success of linear methods relies on the availability of low-dimensional accurate approximation spaces for the solution manifold $\mathcal{M}$. The Kolmogorov $N$-width (\cite{pinkus2012n}) provides a rigorous measure of the linear reducibility  of $\mathcal{M}$: given $N>0$, the Kolmogorov $N$-width $d_N(\mathcal{M})$ is given by
\begin{equation}
\label{eq:kolmogorov_nwidth}
d_N(\mathcal{M}):=  \inf_{\mathcal{Z}_N \subset \mathcal{X}, \; \; {\rm dim} (\mathcal{Z}_N) = N} \; \; 
\sup_{w \in \mathcal{M}} \, \|  w  - \Pi_{\mathcal{Z}_N} w \|,
\end{equation}
where the infimum is taken over all $N$-dimensional spaces, and 
$ \Pi_{\mathcal{Z}_N}: \mathcal{X} \to \mathcal{Z}_N$ denotes the orthogonal projection operator onto $\mathcal{Z}_N \subset \mathcal{X}$. 
In certain engineering-relevant cases, it is possible to demonstrate that 
the decay of $d_N(\mathcal{M})$ with $N$ is extremely rapid:
we refer to  \cite[section 8.1]{rozza2007reduced},
\cite[Theorem 3.3]{babuska2011optimal}  and
\cite{cohen2015approximation,cohen2015kolmogorov} for further details.
Recalling the (quasi-)optimality properties of POD and Greedy algorithms (see \cite{volkwein2011model} and
\cite{cohen2015approximation}, respectively), rapid-decaying Kolmogorov widths justify  the use of linear methods.

However, linear approximation methods are fundamentally ill-suited for several classes of relevant engineering problems.
As observed in \cite[Example 3.5]{haasdonk2013convergence} and 
\cite[Example 2.5]{taddei2015reduced}, solution fields with parameter-dependent jump discontinuities --- or alternatively boundary/internal layers --- cannot be adequately approximated through a low-dimensional linear expansion; similarly, linear methods are   inappropriate to deal with problems with  discontinuous parameter-dependent coefficients. Motivated by these considerations, several authors have proposed \emph{nonlinear approximation/compression methods} to deal with these problems: we here distinguish between \emph{Eulerian} and \emph{Lagrangian} approaches.

Eulerian approaches consider approximations of the form $\widehat{z}_{\mu}:= Z_{N,\mu} (\widehat{\boldsymbol{\alpha}}_{\mu} )$, where
$Z_{N,\mu}: \mathbb{R}^N \to \mathcal{X}$ is a suitably-chosen operator which might depend on the parameter $\mu$ and might also be nonlinear in 
$\boldsymbol{\alpha}$. Eulerian techniques might rely on 
Grassmannian learning \cite{amsallem2008interpolation,zimmermann2018geometric},
convolutional auto-encoders \cite{kashima2016nonlinear,lee2018model},
transported/transformed snapshot methods
\cite{cagniart2019model,nair2018transported,reiss2018shifted,welper2017interpolation},
 displacement interpolation \cite{rim2018model}, to determine the operator $Z_{N,\mu}$. Alternatively, they might exploit adaptive local-in-parameter and/or local-in-space enrichment  strategies
\cite{ballani2018component,bergmann2018zonal,carlberg2015adaptive,
efendiev2016online,eftang2013approximation,etter2018online,ohlberger2015error,
peherstorfer2018model}. Note that, since $Z_{N,\mu}$ is nonlinear and/or   depends on parameter, it might be difficult to develop rapid and reliable procedures for the online computations of the coefficients $\widehat{\boldsymbol{\alpha}}_{\mu}$.

Lagrangian approaches propose to exploit a linear method 
$\widetilde{Z}_N: \mathbb{R}^N \to \widetilde{\mathcal{Z}}_N \subset \mathcal{X}$ 
to approximate the mapped solution  $\widetilde{z}_{\mu}:= z_{\mu} \circ \boldsymbol{\Phi}_{\mu}$ where $\boldsymbol{\Phi}_{\mu} : \Omega \to \Omega$ is a suitably-chosen bijection from $\Omega$ into itself: the mapping $\boldsymbol{\Phi}_{\mu}$ should be chosen to make the mapped solution manifold $\widetilde{\mathcal{M}}$ more amenable for linear approximations. Examples of Lagrangian approaches have been proposed in\footnote{
The method of freezing proposed in \cite{ohlberger2013nonlinear} 
aims at decomposing the solution into a group component and  a shape component. The approach reduces to a Lagrangian approach if the group action is induced by a mapping of the underlying spatial domain.
} 
\cite{iollo2014advection,mojgani2017arbitrary,ohlberger2013nonlinear,taddei2015reduced}:
in \cite{ohlberger2013nonlinear,taddei2015reduced} the construction of the map is performed separately from the construction of the solution coefficients, while in 
\cite{mojgani2017arbitrary} the authors propose to build the mapping $\boldsymbol{\Phi}_{\mu}$ and the solution coefficients $\widehat{\boldsymbol{\alpha}}_{\mu}$ simultaneously. 
Simultaneous learning of mapping and solution coefficients leads to a  nonlinear and non-convex optimization problem even for linear PDEs.
Note that
{any  Lagrangian method is equivalent to an Eulerian method with 
$Z_{N,\mu} (\boldsymbol{\alpha}) := \sum_{n=1}^N \alpha_n \, \tilde{\zeta}_n\circ \boldsymbol{\Phi}_{\mu}^{-1}$, while the converse is not true.}
On the other hand,  the application of pMOR techniques to the mapped problem \eqref{eq:mapped_PMOR_setting} is completely standard: this is in contrast with Eulerian approaches, which might require more involved strategies for the computation of the solution coefficients $\widehat{\boldsymbol{\alpha}}_{\mu}$
(see \cite{lee2018model}).

\subsection{Reduction of problems in parameterized domains}
Given the family of parameterized domains $\{ \Omega_{\mu}: \mu \in \mathcal{P} \}$, we shall here identify a 
reference domain $\Omega$ and a bijective mapping $\boldsymbol{\Phi}$ such that
$\Omega_{\mu}= \boldsymbol{\Phi}_{\mu}({\Omega
})$ for all $\mu \in \mathcal{P}$; the mapping should be computable for new values of $\mu \in \mathcal{P}$, with limited computational and memory resources. 
Several authors have developed geometry reduction techniques based on automatic piecewise-affine maps (\cite{rozza2007reduced}), Radial Basis Functions (RBFs, \cite{manzoni2012model}), transfinite maps
(\cite{lovgren2006reduced,iapichino2012reduced}) and solid extension (\cite{manzoni2017efficient}). While the approach in \cite{rozza2007reduced} is restricted to a specific class of parametric deformations, the other approaches
--- which exploit ideas originally developed in the framework of mesh deformation  and surface interpolation/approximation -- 
are broadly applicable.
Despite the many recent contributions to the field, development of rapid and reliable geometry reduction techniques for large deformations is still a challenging task.

For applications to biological systems, parametric descriptions of the boundary $\partial \Omega_{\mu}$ for all $\mu$ are unavailable. For this reason,
in order to find the mapping $\boldsymbol{\Phi}$, 
 it is important to implement strategies to systematically determine parametric descriptions of $\partial \Omega_{\mu}$.
This problem, which is shared by many of the geometry reduction approaches mentioned above, is not addressed  in the present paper.

\subsection{Contributions and outline of the paper}

We propose a registration procedure that takes as input a set of snapshots $\{u^k = u_{\mu^k}   \}_{k=1}^{n_{\rm train}} \subset L^2(\Omega)$ with $\mu^1,\ldots, \mu^{n_{\rm train}} \in \mathcal{P}$, and returns the  parametric mapping  $\boldsymbol{\Phi}: \Omega
\times \mathcal{P} \to \mathbb{R}^d$.
The key features of the approach are (i) a nonlinear non-convex  optimization statement that aims at reducing the difference (in a suitable metric) between a properly-chosen reference field $\bar{u}= u_{\bar{\mu}}$ and the mapped field $\tilde{u}_{\mu^k}= u_{\mu^k} \circ \boldsymbol{\Phi}_{\mu^k}$ for $\mu^1,\ldots,\mu^{n_{\rm train}}$,  and 
(ii) a generalization procedure based on kernel  regression  that aims at determining a mapping for all $\mu \in \mathcal{P}$ based on the available data for $\{ \mu^k \}_{k=1}^{n_{\rm train}}$.

For nonlinear data compression, we apply the registration procedure to the solution itself ($u_{\mu}=z_{\mu}$); for geometry reduction, we 
{apply the registration procedure  to determine a mapping from $\Omega$ to $\Omega_{\mu}$, for all $\mu \in \mathcal{P}$.
} To demonstrate the generality of the approach for  data compression, we consider the application to a boundary layer problem and  to an advection-reaction problem with a parameter-dependent sharp gradient region; on the other hand, we consider the application to two diffusion problems to demonstrate the effectivity of our proposal to geometry reduction.

We observe that the present paper is related to a number of prior works. First, as the Eulerian approach in 
\cite{lee2018model}, our registration algorithm is independent of the underlying equation: for this reason, we believe that the approach can be applied to a broad class of problems in science and engineering. Second, 
our optimization statement for the construction of 
$\boldsymbol{\Phi}$
is related to the recent proposal by Zahr and Persson for $r$-adaptivity in the DG framework (\cite{zahr2018optimization}).
Third, the optimization statement is tightly linked to optimal transport  (\cite{villani2008optimal}), which has been employed in \cite{iollo2014advection} to devise a nonlinear approximation method: in section \ref{remark:optimal_transport}, we discuss similarities and differences between the mappings obtained using our method and optimal transport maps.
{
Fourth, the optimization statement shares also similarities with the method of slices considered in \cite{rowley2000reconstruction,mowlavi2018model}. In particular, the reference field $\bar{u}$ plays the role of the \emph{template} in the above mentioned papers. The key difference is that the method of slices --- similarly to 
\cite{ohlberger2013nonlinear}  --- 
   is associated with problems with symmetries (and more in general with  problems that are invariant under a given group action) and  exploits specific features of the physical system; on the other hand, our approach is exclusively driven by approximation considerations and does not require any particular structure to the elements of the parametric manifold.}

The outline of the paper is as follows. 
{ In section \ref{sec:theory_stuff}, we present some necessary mathematical background that will guide us in the definition of the methodology. }  Then, in sections \ref{sec:data_compression} and \ref{sec:geometry_reduction}, we discuss the development of the registration procedure for data compression and geometry reduction; for each task, we investigate performance  through the vehicle of two model problems. Finally, in section \ref{sec:conclusions}, we provide a short summary and we discuss several potential next steps.

\subsection{Notation}

By way of preliminaries, we introduce notation used throughout the paper. Given the parametric mapping $\boldsymbol{\Phi}$, we denote by $\mathbf{X}$  (resp. $\mathbf{x}$) a generic point in the reference (resp. physical) configuration;  we use notation
$\widehat{\nabla}= [\widehat{\partial}_1,\ldots, \widehat{\partial}_d]^T$ and
$ {\nabla}= [ {\partial}_1,\ldots,  {\partial}_d]^T$
to refer to reference and physical gradients and we recall that
 $\nabla = \left(\widehat{\nabla} \boldsymbol{\Phi}_{\mu} \right)^{-T} \, \widehat{\nabla}$;
 we denote by $\mathfrak{J}_{\mu}= {\rm det} \left( \widehat{\nabla} \, \boldsymbol{\Phi}_{\mu} \right)$ the Jacobian determinant; 
  finally, given the field $g \in L^2$, we define  the corresponding mapped field
 $\tilde{g}_{\mu}:=g \circ \boldsymbol{\Phi}_{\mu}$.

We   denote by $\mathbf{e}_1,\ldots,\mathbf{e}_N$ the canonical basis in $\mathbb{R}^N$ and by $\|  \cdot \|_2$ the Euclidean norm. 
Given the open set $A \subset \mathbb{R}^d$, we define 
  the corresponding indicator function $\mathbbm{1}_A: \mathbb{R}^d \to \{0,1\}$, the   distance  function ${\rm dist}(\mathbf{x}, A) = \inf_{\mathbf{y} \in A} \, \| \mathbf{x} - \mathbf{y} \|_2$, 
the closure   $\overline{A}:= A \cup \partial A$, and given the function 
  $w: \Omega  \to \mathbb{R}^q$, $q \geq 1$ and $A \subset \Omega$, we define 
   the image of $w$ in $A$  as $w(A) = \{ w(\mathbf{x}) : \, \mathbf{x} \in A \} \subset \mathbb{R}^q$. 
Given $\delta>0$, we  denote by $ A_{\delta}$ the $\delta$-neighboorhood of $A$,
$A_{\delta} := \{  \mathbf{x} \in \mathbb{R}^d:  {\rm dist}(\mathbf{x}, A) < \delta \}$.
Furthermore,  we say that the differentiable function
$\mathbf{H}: A \to \mathbb{R}^d$ is a diffeomorphism  if it is a bijection and its inverse  is differentiable;
 given the open set $B \subset \mathbb{R}^d$, we say that  $A$ is diffeomorphic to $B$ (and we use notation $A \simeq B$) if there exists a diffeomorphism $\mathbf{H}$  such that $\mathbf{H}(A)=B$, and we say that $A$ is compactly embedded in $B$ (and we use notation $A \Subset B$) if $\overline{A}$  is contained in $B$.
Finally, we introduce the Hausdorff distance   ${\rm dist}_{\rm H} (A, B) := \max \left\{
\sup_{\mathbf{x} \in A} {\rm dist}(\mathbf{x}, B), \, 
\sup_{\mathbf{y} \in B} {\rm dist}(\mathbf{y}, A)
\right\}$.
  
{Numerical results rely on  a high-fidelity continuous or discontinuous
Galerkin    finite element (FE) discretization.
We denote by $\{ \texttt{D}^k  \}_{k=1}^{n_{\rm el}}$ the elements of the mesh; we denote by $N_{\rm hf}$ the number of degrees of freedom; and 
we denote by $\{  \mathbf{x}_q^{\rm qd}  \}_{q=1}^{N_{\rm q}}$ the quadrature points.
}

 In this work, we resort to POD to generate low-dimensional linear approximation spaces; we use the method of snapshots (\cite{sirovich1987turbulence}) to compute POD eigenvalues and eigenvectors. Given the snapshot set $\{ u^k \}_{k=1}^{n_{\rm train}} \subset \mathcal{M}_{\rm u}$ and the inner product $(\cdot, \cdot)$, we define the Gramian matrix
$\mathbf{C} \in \mathbb{R}^{n_{\rm train}, n_{\rm train}}$,
$\mathbf{C}_{k,k'}= (u^k, u^{k'})$, 
and we define the POD eigenpairs    
$\{(\lambda_n, \zeta_n)  \}_{n=1}^{n_{\rm train}}$
as
$$
\mathbf{C} \boldsymbol{\zeta}_n = \lambda_n  \, \boldsymbol{\zeta}_n,
\quad
\zeta_n:=   \sum_{k=1}^{n_{\rm train}} \, \left(   \boldsymbol{\zeta}_n  \right)_k \, u^k,
\quad
n=1,\ldots,n_{\rm train},
$$
with $\lambda_1 \geq \lambda_2 \geq \ldots \lambda_{n_{\rm train}}=0$. In our implementation, we orthonormalize the modes, that is
$\| \zeta_n \|= 1$ for $n=1,\ldots,n_{\rm train}$.
To stress dependence of the POD space on the choice of the inner product, we use notation $L^2$-POD if $(\cdot, \cdot ) = (\cdot, \cdot)_{L^2(\Omega)}$, $H^1$-POD if 
$(\cdot, \cdot ) = (\cdot, \cdot)_{H^1(\Omega)}$, and
 $\|  \cdot \|_2$-POD if $(\cdot, \cdot)$ is the Euclidean inner product. Finally, we shall choose the size $N$ of the POD space based on the criterion
\begin{equation}
\label{eq:POD_cardinality_selection}
N := \min \left\{
N': \, \sum_{n=1}^{N'} \lambda_n \geq  \left(1 - tol_{\rm pod} \right) 
\sum_{i=1}^{n_{\rm train}} \lambda_i
\right\},
\end{equation} 
where $tol_{\rm pod} > 0$ is a given tolerance.

\section{Affine mappings: theoretical preliminaries}
\label{sec:theory_stuff}
 
Given the diffeomorphic open sets $U$ and $V$ in $\mathbb{R}^d$, and the target tolerance $\texttt{tol} \geq 0$, we are interested in determining $\boldsymbol{\Phi}: U \to \mathbb{R}^d$ such that
\begin{equation}
\label{eq:bijective_map}
{\rm dist}_{\rm H} (\boldsymbol{\Phi}(U), V) \leq \texttt{tol}, \quad
 \boldsymbol{\Phi}: U \to \boldsymbol{\Phi}(U) \; {\rm is \; bijective}.
\end{equation} 
In data compression (cf. section \ref{sec:data_compression}), we have $U=V=\Omega$, while in geometry reduction (cf.  section  \ref{sec:geometry_reduction}) $U=\Omega$ and $V=\Omega_{\mu}$. In this paper, we shall seek mappings of the form
\begin{equation} 
\label{eq:prescribed_mapping_form}
  \boldsymbol{\Psi}_{\mathbf{a}}^{\rm hf} (\mathbf{X}) :=
\mathbf{X} + \sum_{m=1}^{M_{\rm hf}} \, a_m \boldsymbol{\varphi}_m^{\rm hf}(\mathbf{X}),
\end{equation}
where $\{ \boldsymbol{\varphi}_m^{\rm hf} \}_{m=1}^{M_{\rm hf}} \subset {\rm Lip} (U, \mathbb{R}^d )$.
In   view of the discussion, we further define the Jacobian
   $\mathfrak{J}_{\mathbf{a}}^{\rm hf}
:= {\rm det} \left( \widehat{\nabla} \boldsymbol{\Psi}_{\mathbf{a}}^{\rm hf} \right)$.

In order to devise an effective computational methodology to determine problem-dependent parameterized mappings, we shall  discuss two issues: (i) the choice of the search space ${\rm span}\{ \boldsymbol{\varphi}_m^{\rm hf} \}_{m=1}^{M_{\rm hf}} $, and (ii) the derivation of an actionable condition for the coefficients $\mathbf{a}= \left[a_1, \ldots, a_{M_{\rm hf}}\right]^T$
to enforce \eqref{eq:bijective_map}. 

\subsection{Main result}
\label{sec:theory}

In Proposition \ref{th:mapping_general},  we state the main result of the paper.
Proof of Proposition   \ref{th:mapping_general}  is provided in section \ref{sec:proof_technical}. 

\begin{proposition}
\label{th:mapping_general}
Let $U, V \subset \mathbb{R}^d$ satisfy
\begin{equation}
\label{eq:hyp_domain}
U = \boldsymbol{\Upsilon}(\widehat{\Omega}), \; \;
V = \boldsymbol{\Lambda}(\widehat{\Omega}), \qquad
 \widehat{\Omega} = \{\mathbf{x} \in \mathbb{R}^d: \, f(\mathbf{x}) < 0   \}, \quad
{\rm where} \, f: \mathbb{R}^d \to \mathbb{R}
\; {\rm is \; convex,}
\end{equation}
{and $\boldsymbol{\Upsilon}$ (resp. $\boldsymbol{\Lambda}$) is a diffemorphism  from $\widehat{\Omega}$ to $U$ (resp. $V$) that has a $C^1$ extension to $\mathbb{R}^d$.}
Given $\delta>0$, let $\boldsymbol{\Phi}:  U_{\delta} \to \mathbb{R}^d$ be a vector-valued function that satisfies 
\begin{enumerate}
\item[(i)]
$\boldsymbol{\Phi} \in C^1 \left(  U_{\delta}; \mathbb{R}^d \right)$;
\item[(ii)]
$\mathfrak{J} (\mathbf{X}) = {\rm det} \left( \widehat{\nabla} \boldsymbol{\Phi}(\mathbf{X}) \right)   \geq \epsilon >0$  for all $\mathbf{X} \in U$ and a given  $\epsilon>0$;
\item[(iii)]
${\rm dist} \left(\boldsymbol{\Phi}(\mathbf{X}), \,  \partial V   \right) = 0$  for all $\mathbf{X} \in \partial U$ (i.e., $\boldsymbol{\Phi}(\partial  U)  \subseteq \partial V $).
\end{enumerate} 
Then, $\boldsymbol{\Phi}$ is a bijection that maps  $U$ into $V$.
\end{proposition}

 We observe that the (unit) ball and the (unit) hyper-cube satisfy  
\eqref{eq:hyp_domain}: we have indeed that
$\mathcal{B}_1(\mathbf{0})= \{\mathbf{x} \in \mathbb{R}^d: \, f(\mathbf{x})=\| \mathbf{x}\|_2^2 - 1 < 0   \}$, and
$(0,1)^d = \{\mathbf{x} \in \mathbb{R}^d: \, f(\mathbf{x})=2 \| \mathbf{x}  - [1/2,1/2] \|_{\infty} - 1 < 0   \}$. 
We further remark  that 
 $\partial \widehat{\Omega} = \{ \mathbf{x}\in \mathbb{R}^d: f(\mathbf{x})=0  \}$: this implies that 
any $U$ satisfying \eqref{eq:hyp_domain}    is simply connected with connected boundary. Note that, 
{recalling standard extension theorems in analysis,    any simply connected 
 two-dimensional smooth domain}     satisfies  \eqref{eq:hyp_domain}; the latter result is in general false for $d=3$.

Corollary \ref{th:3D_simply_mapping} illustrates an extension of Proposition \ref{th:mapping_general} to a broader class of domains.
The result is stated for $U=V=\Omega$. The extension to the more general case
$U \neq V$ is here omitted.

\begin{corollary}
\label{th:3D_simply_mapping}
Let $\Omega = \Omega^{\rm out} \setminus \bigcup_{i=1}^Q \Omega_i^{\rm in}$
Let   $\Omega^{\rm out}, \Omega_1^{\rm in},\ldots, \Omega_Q^{\rm in} \subset \mathbb{R}^d$ satisfy the following assumptions.
\begin{enumerate}
\item[(i)]
$\Omega_1^{\rm in},\ldots, \Omega_Q^{\rm in}, \Omega^{\rm out}$ satisfy  \eqref{eq:hyp_domain}.
\item[(ii)]
$\Omega_1^{\rm in},\ldots, \Omega_Q^{\rm in}$ are pairwise disjoint, and $ \Omega_1^{\rm in},\ldots, \Omega_Q^{\rm in} \Subset  
\Omega^{\rm out}$. 
\end{enumerate} 
Let  
 $\boldsymbol{\Phi}:  {\Omega}_{\delta}^{\rm out} \to  \mathbb{R}^d$ be a function such that
\begin{enumerate} 
\item[(i)]
$\boldsymbol{\Phi} \in C^1 \left(  {\Omega}_{\delta}^{\rm out}; \mathbb{R}^d \right)$;
\item[(ii)]
$ \mathfrak{J}(\mathbf{X})   \geq \epsilon >0$  for all $\mathbf{X} \in {\Omega}^{\rm out}$ and a given  $\epsilon>0$;
\item[(iii)]
$\boldsymbol{\Phi}(\partial  {\Omega}^{\rm out})  \subseteq \partial \Omega^{\rm out} $,
$\boldsymbol{\Phi}(\partial  {\Omega}_i^{\rm in})  \subseteq \partial \Omega_i^{\rm in} $ for $i=1,\ldots,Q$.
\end{enumerate} 
Then, $\boldsymbol{\Phi}$ is a bijection that maps $\Omega$ into itself.
\end{corollary}

\begin{proof}
We first observe that, by construction, we have $\partial \Omega = \partial \Omega^{\rm out} \cup  \bigcup_{i=1}^Q  \partial \Omega_i^{\rm in}$.
Exploiting Proposition  \ref{th:mapping_general}, we find that $\boldsymbol{\Phi}$ is a bijection that maps
$\Omega^{\rm out}, \Omega_1^{\rm in},\ldots,\Omega_Q^{\rm in}$ into themselves.
Then, we observe that
$$
\boldsymbol{\Phi}(\Omega) =   \boldsymbol{\Phi}\left(
 \Omega^{\rm out} \setminus \bigcup_{i=1}^Q \Omega_i^{\rm in}
\right)
=
\boldsymbol{\Phi} \left(  \Omega^{\rm out} \right) 
\setminus \bigcup_{i=1}^Q
\boldsymbol{\Phi} \left(  \Omega_i^{\rm in} \right) 
=
\Omega^{\rm out} \setminus \bigcup_{i=1}^Q \Omega_i^{\rm in}
= \Omega,
$$
{where the second identity follows from the fact that $\boldsymbol{\Phi}$ is a bijection from $\Omega^{\rm out}$ into itself and 
$\boldsymbol{\Phi}(\Omega_i^{\rm in})= \Omega_{i}^{\rm in}$ for $i=1,\ldots, Q$.
}
Thesis follows.
\end{proof}

\subsection{Implications for $U=V= (0,1)^d$}
\label{sec:choice_family}
Next result, which is a straightforward consequence of Proposition \ref{th:mapping_general}, provides indications for the choice of $\boldsymbol{\varphi}_1^{\rm hf},\ldots,
\boldsymbol{\varphi}_{M_{\rm hf}}^{\rm hf}$ and of the coefficients $\mathbf{a}$ for $U=V=(0,1)^d$.
 
\begin{proposition}
\label{th:application_unit_square}
Let  $d=2$ or $d=3$ and let
$\boldsymbol{\varphi}_1^{\rm hf},\ldots,
\boldsymbol{\varphi}_{M_{\rm hf}}^{\rm hf}$ in \eqref{eq:prescribed_mapping_form}  be of class 
$C^1(\mathbb{R}^d; \, \mathbb{R}^d)$ and 
 satisfy
 \begin{subequations}
 \label{eq:family_mappings}
\begin{equation}
\label{eq:boundary_conditions_mapping}
\boldsymbol{\varphi}_m^{\rm hf}(\mathbf{X}) \cdot \mathbf{e}_i \, = \, 0,
\quad
{\rm on} \; \{ \mathbf{X}: \; X_i=0, {\rm or} \, X_i=1\},
\quad
\left\{
\begin{array}{l}
m=1,\ldots, M_{\rm hf};
\\ 
i=1,\ldots, d. 
\\
\end{array}
\right.
\end{equation}
 \end{subequations}
Then, for any $\bar{\mathbf{a}} \in \mathbb{R}^{M_{\rm hf}}$,
$\boldsymbol{\Phi}  = \boldsymbol{\Psi}_{ \bar{\mathbf{a}}}^{\rm hf}$ is bijective from the unit square (or cube) $\Omega=(0,1)^d$ into itself if   
\begin{equation}
\label{eq:condition_equivalent}
\min_{\mathbf{X} \in \overline{\Omega}} \,  
\,{\rm det} \left(
\widehat{\nabla} \boldsymbol{\Phi}(\mathbf{X}) \right) \, > \, 0.
\end{equation}

Furthermore,
for any $\epsilon \in (0,1)$, there exists a ball
$B =\mathcal{B}_{r_{\epsilon}}(\mathbf{0})$
 of radius $r_{\epsilon}>0$ centered in $\mathbf{0}$ such that
 $\inf_{\mathbf{X} \in  {\Omega} , \mathbf{a} \in B}  
 \mathfrak{J}_{\mathbf{a}}^{\rm hf}(\mathbf{X})
 \geq \epsilon $. 
\end{proposition}

\begin{proof}
We first consider the case $d=2$.
Recalling Proposition \ref{th:mapping_general}, we only need to verify that $\boldsymbol{\Phi}(\partial \Omega) \subset \partial \Omega$. We here verify that 
$\boldsymbol{\Phi}(\mathcal{E}_{\rm top}) \subset \mathcal{E}_{\rm top}$ where
$\mathcal{E}_{\rm top}:=\{\mathbf{X} = (t, 1): \, t \in (0,1)  \}$ is the top edge of $(0,1)^2$:  proofs for the other edges are analogous.

Exploiting \eqref{eq:boundary_conditions_mapping}, we have that 
$$
\boldsymbol{\Phi}(\mathbf{X}= (t,1)) = \left[
\begin{array}{c}
\varphi(t) \\
1 \\
\end{array}
\right],
\quad
{\rm where} \;
\varphi(t) := \mathbf{e}_1 \cdot \boldsymbol{\Phi}(\mathbf{X}=(t,1)).
$$
By contradiction, we assume that $\varphi([0,1])$ is not contained in $[0,1]$.
Since $\varphi(0)=0$ and $\varphi(1)=1$,  
$\varphi$ has a local minimum or maximum in $(0,1)$; this implies that there exists $\bar{t} \in (0,1)$ such that
$\varphi'(\bar{t})=0$. As a result, we find
$$
\mathbf{e}_1^T \, \widehat{\nabla} \boldsymbol{\Phi} \big|_{\mathbf{X}= (\bar{t},1)} 
=
\left[
\varphi'(\bar{t}) , \; \; 
0 
\right]
= 0,
$$
which implies that $\mathfrak{J}(\mathbf{X}=(\bar{t},1)) = 0$. Contradiction.

{To extend the result to $d=3$, we show that 
$\boldsymbol{\Phi}(\mathcal{E}_{\rm top}) \subset \mathcal{E}_{\rm top}$ where
$\mathcal{E}_{\rm top}:=\{\mathbf{X} = (t,s, 1): \, t,s \in (0,1)  \}$ is the top face of $(0,1)^3$. Towards this end, 
we define $\widetilde{\boldsymbol{\Phi}}(t,s)=[\Phi_1(t,s,1),\Phi_2(t,s,1) ]^T$:  clearly, $\widetilde{\boldsymbol{\Phi}}$ satisfies the hypotheses of Proposition \ref{th:application_unit_square} for $d=2$; as a result, $\widetilde{\boldsymbol{\Phi}}((0,1)^2) = (0,1)^2$ and thus
$ {\boldsymbol{\Phi}}(\mathcal{E}_{\rm top}) =  \mathcal{E}_{\rm top}$. Thesis follows.
}

Proof of the latter statement is a direct consequence of the fact that the determinant of a matrix-valued function is continuous and that
$ \mathfrak{J}_{\mathbf{a}}^{\rm hf}\equiv 1$ for $\mathbf{a}=\mathbf{0}$. We omit the details.
\end{proof}

Condition \eqref{eq:boundary_conditions_mapping} imposes that each edge of the square should be mapped in itself and that each corner is mapped in itself: Figure \ref{fig:geometric_interpretation} provides the  geometric interpretation.
In our implementation, for $d=2$, we define $\boldsymbol{\varphi}_1^{\rm hf},\ldots,\boldsymbol{\varphi}_{M_{\rm hf}}^{\rm hf}$ as 
\begin{equation}
\label{eq:tensorized_polynomials}
\left\{
\begin{array}{l}
\boldsymbol{\varphi}_{m=
i + (i'-1)\bar{M}}^{\rm hf}(\mathbf{X})
=
\ell_i(X_1) \ell_{i'}(X_2) \, X_1 (1 - X_1) \, \mathbf{e}_1 \\[2mm]
\boldsymbol{\varphi}_{m= \bar{M}^2 + 
i + (i'-1)\bar{M}}^{\rm hf}(\mathbf{X})
=
\ell_i(X_1) \ell_{i'}(X_2) \, X_2 (1 - X_2) \, \mathbf{e}_2 \\
\end{array}
\right.
\quad
i,i'=1,\ldots, \bar{M},
\end{equation}
where $\{  \ell_i \}_{i=1}^{\bar{M}}$ are the first $\bar{M}$ Legendre polynomials and $M_{\rm hf} = 2 \bar{M}^2$;
note, however, that  other choices 
satisfying \eqref{eq:boundary_conditions_mapping}
(e.g., Fourier expansions) might also be considered.

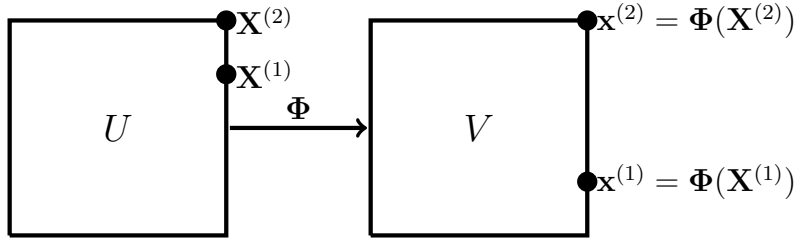
\begin{figure}[h!]
\centering

\begin{tikzpicture}[scale=0.95]
\linethickness{0.3 mm}
\draw[ultra thick]  (0,0)--(3,0)--(3,3)--(0,3)--(0,0);
\draw[ultra thick]  (5,0)--(8,0)--(8,3)--(5,3)--(5,0);

\coordinate [label={center:  {\Large {$U$}}}] (E) at (1.5, 1.5) ;

\coordinate [label={center:  {\Large {$V$}}}] (E) at (6.5, 1.5) ;

\draw[->,ultra thick]  (3.05,1.5)--(4.95,1.5);

\coordinate [label={above:  {\large {$\boldsymbol{\Phi}$}}}] (E) at (4, 1.5) ;

\coordinate [label={right:  {\large {$\mathbf{X}^{(1)}$}}}] (E) at (3, 2.25) ;
\fill (E) circle[radius=4pt];

\coordinate [label={right:  {\large {$\mathbf{x}^{(1)}= \boldsymbol{\Phi}(\mathbf{X}^{(1)})$}}}] (E) at (8, 0.75) ;
\fill (E) circle[radius=4pt];

\coordinate [label={right:  {\large {$\mathbf{X}^{(2)}$}}}] (E) at (3, 3) ;
\fill (E) circle[radius=4pt];

\coordinate [label={right:  {\large {$\mathbf{x}^{(2)}= \boldsymbol{\Phi}(\mathbf{X}^{(2)})$}}}] (E) at (8, 3) ;
\fill (E) circle[radius=4pt];

\end{tikzpicture}

\caption{  geometric interpretation of
condition \eqref{eq:boundary_conditions_mapping} in Proposition  \ref{th:application_unit_square}. 
}
\label{fig:geometric_interpretation}
\end{figure}

{Condition \eqref{eq:condition_equivalent} for the coefficients $\mathbf{a} \in \mathbb{R}^{M_{\rm hf}}$
is difficult to impose computationally; however, we might replace \eqref{eq:condition_equivalent} with the approximation
\begin{equation}
\label{eq:condition_equivalent_weak}
\int_{\Omega} \, 
{\rm exp} \left( \frac{\epsilon  - \mathfrak{J}_{\mathbf{a}}^{\rm hf}(\mathbf{X})}{C_{\rm exp}} \right) \,  + \, 
{\rm exp} \left( \frac{\mathfrak{J}_{\mathbf{a}}^{\rm hf}(\mathbf{X}) - 1/\epsilon }{C_{\rm exp}} \right)
\, dX \leq \delta,
\end{equation}
where $\epsilon \in (0,1)$.
Provided that ${\rm exp} \left(  \frac{\epsilon}{C_{\rm exp}} \right) \gg 1$ and that
$\|  \widehat{\nabla} \mathfrak{J}_{\mathbf{a}}^{\rm hf}  \|_{L^{\infty}(\Omega)}$ is moderate, 
this constraint enforces that   the mapping is invertible for all $\mathbf{X} \in \Omega$:
more formally, for all $\epsilon>0$, there exist $C_{\rm exp}, \delta,C>0$ such that if
$\mathbf{a}$ satisfies \eqref{eq:condition_equivalent_weak} and $\|  \widehat{\nabla} \mathfrak{J}_{\mathbf{a}}^{\rm hf}  \|_{L^{\infty}(\Omega)}<C$, then $\boldsymbol{\Psi}_{\mathbf{a}}^{\rm hf}$ is globally invertible.
 To show this statement for $d=2$, suppose that 
$\mathfrak{J}_{\mathbf{a}}^{\rm hf}(\mathbf{X}^{\star})=0$
for some $\mathbf{X}^{\star} \in \Omega$; then, there exists a ball $B= \mathcal{B}_r(\mathbf{X}^{\star})$ of radius $r \geq 
\left(  \|  \widehat{\nabla} \mathfrak{J}_{\mathbf{a}}^{\rm hf}  \|_{L^{\infty}(\Omega)}    \right)^{-1} \epsilon/2$ such that 
$\mathfrak{J}_{\mathbf{a}}^{\rm hf}(\mathbf{X} ) \leq \epsilon/2$ for all $\mathbf{X} \in B$. As a result, we find\footnote{The factor $\frac{1}{4}$ follows from  the identity 
$|\mathcal{B}_r(\mathbf{X}) \cap \Omega| \geq \frac{\pi r^2}{4}$, for all $r \leq \frac{1}{2}$ and $\mathbf{X} \in \Omega$.
} that
$$
\begin{array}{l}
\displaystyle{
\int_{\Omega} \, 
{\rm exp} \left( \frac{\epsilon  - \mathfrak{J}_{\mathbf{a}}^{\rm hf}(\mathbf{X})}{C_{\rm exp}} \right) \,  + \, 
{\rm exp} \left( \frac{\mathfrak{J}_{\mathbf{a}}^{\rm hf}(\mathbf{X}) - 1/\epsilon }{C_{\rm exp}} \right)
\, dX \, 
\geq  \,
\frac{\pi r^2}{4} \, {\rm exp} \left(\frac{\epsilon}{2 C_{\rm exp}}  \right)
}
\\[3mm]
\displaystyle{
\geq
\frac{\pi}{16} \,
\left(  \|  \widehat{\nabla} \mathfrak{J}_{\mathbf{a}}^{\rm hf}  \|_{L^{\infty}(\Omega)}    \right)^{-2} \, \epsilon^2 \, 
 {\rm exp} \left(\frac{\epsilon}{2 C_{\rm exp}}  \right) =: \delta_0,
}
\end{array}
$$
which implies that $\mathbf{a}$ does not satisfies \eqref{eq:condition_equivalent_weak} for $\delta< \delta_0$.

 We can  interpret $\epsilon \in (0,1)$ as the maximum allowed pointwise contraction induced by the mapping $\boldsymbol{\Psi}^{\rm hf}$ and by its inverse. 
We further remark that  the constant 
$\delta$ should satisfy
\begin{equation}
\label{eq:delta_condition}
\delta \geq  |\Omega| \; \left(
{\rm exp} \left( \frac{\epsilon  - 1}{C_{\rm exp}} \right) \,  + \, 
{\rm exp} \left( \frac{1 - 1/\epsilon }{C_{\rm exp}} \right) \right),
\end{equation}
so that $\mathbf{a}= \mathbf{0}$ is admissible.
In all our numerical examples, we choose 
\begin{equation}
\label{eq:choice_Cdelta}
\epsilon=0.1, \quad 
C_{\rm exp} = 0.025 \epsilon, \quad
\delta = |\Omega|.
\end{equation}
}
 
\subsection{Extension to a more general class of domains}
\label{sec:generalization_theory}
We can exploit the results proved in sections \ref{sec:theory} and \ref{sec:choice_family} to obtain the following result. The proof is straightforward and is omitted.

\begin{proposition}
\label{th:extension_sad}
Let  
$\boldsymbol{\varphi}_1^{\rm hf},\ldots,
\boldsymbol{\varphi}_{M_{\rm hf}}^{\rm hf}$ in \eqref{eq:prescribed_mapping_form}  be of class 
$C^1(\mathbb{R}^d; \, \mathbb{R}^d)$ and 
 satisfy \eqref{eq:boundary_conditions_mapping}. 
Let $U,V \subset \Omega_{\rm box}=(0,1)^d$ satisfy \eqref{eq:hyp_domain}.

Then, for any $\bar{\mathbf{a}} \in \mathbb{R}^{M_{\rm hf}}$,
$\boldsymbol{\Phi}  = \boldsymbol{\Psi}_{ \bar{\mathbf{a}}}^{\rm hf}$ is bijective from $U$ into $V$ if  
\begin{equation}
\label{eq:condition_equivalent_general}
\left\{
\begin{array}{l}
\displaystyle{
\min_{\mathbf{X} \in \overline{\Omega}_{\rm box}} \,  
\,{\rm det} \left(
\widehat{\nabla} \boldsymbol{\Phi}(\mathbf{X}) \right) \, > \, 0,
} \\[3mm]
\displaystyle{
{\rm dist} \left( \boldsymbol{\Phi}(\mathbf{X}), \partial V   \right) = 0
\quad
\forall \, \mathbf{X} \in \partial U.
} \\ 
\end{array}
\right.
\end{equation}
\end{proposition}

We remark that  condition \eqref{eq:condition_equivalent_general}$_2$ is not practical and should in practice 
be replaced by 
$$
{\rm dist} \left( \boldsymbol{\Phi}(\mathbf{X}), \partial V   \right)  \leq \texttt{tol}
\quad
\forall \, \mathbf{X} \in \partial U,
$$
for some tolerance $\texttt{tol}>0$. We are currently working on the extension of Proposition \ref{th:extension_sad} to this more general case: as discussed in section 
 \ref{sec:geometry_reduction},  this extension
would rigorously justify the approach for geometry reduction proposed in this paper.

\section{Data compression}
\label{sec:data_compression}
 
 In section \ref{sec:registration}, we present the registration procedure for $U=V=\Omega$, for data compression. We shall here assume that $\boldsymbol{\Phi}_{\mu}(\Omega)= \Omega$ for all $\mu \in \mathcal{P}$ and $\Omega=(0,1)^2$: in section  \ref{sec:theory_stuff}, we provide  an actionable way to enforce this condition for this choice of $\Omega$.
In section \ref{sec:NM_width}, we introduce the notion of Kolmogorov $N-M$ width, which is a generalization of \eqref{eq:kolmogorov_nwidth}.
Then, in section \ref{sec:offline_online_pMOR}, we 
discuss the integration of the proposed compression method within the standard pMOR offline/online paradigm.
 Finally, in section \ref{sec:numerics_data_compression}, we present  numerical investigations for two two-dimensional model problems.

\subsection{Registration procedure}
\label{sec:registration}

Given the snapshot $\{ u^k  \}_{k=1}^{n_{\rm train}} \subset \mathcal{M}_{\rm u}$,
we first introduce the parameterized function $\boldsymbol{\Psi}^{\rm hf}$ in \eqref{eq:prescribed_mapping_form}-\eqref{eq:tensorized_polynomials} and a reference field $\bar{u} \in \mathcal{M}_{\rm u}$ associated with the parameter $\bar{\mu} \in \mathcal{P}$. Then,
(i)
for $k=1,\ldots, n_{\rm train}$, we choose $\mathbf{a}_{\rm hf}^k$ associated with $u^k$ and $\bar{u}$ by solving an optimization problem (cf. section 
\ref{sec:optimization}); 
(ii)
given the dataset $\{ (\mu^k, \mathbf{a}_{\rm hf}^k)  \}_{k=1}^{n_{\rm train}}$, we use a multi-target regression procedure to generate the mapping $\boldsymbol{\Phi}$ of the form
\begin{equation}
\label{eq:registration_mapping_form}
 \boldsymbol{\Phi}_{\mu} (\mathbf{X}) :=
\mathbf{X} + \sum_{m=1}^{M} \, 
\left( \widehat{\mathbf{a}}_{\mu} \right)_m   \, \boldsymbol{\varphi}_m (\mathbf{X}),
\quad
{\rm where} \;
\left\{
\begin{array}{ll}
\widehat{\mathbf{a}}: \mathcal{P} \to \mathbb{R}^M, &
 \\
\boldsymbol{\varphi}_m \in {\rm span} \{  \boldsymbol{\varphi}_m^{\rm hf} \}_{m=1}^{M_{\rm hf}}
&
M \leq M_{\rm hf},
\\
\end{array}
\right.
\end{equation}
 for all $\mu \in \mathcal{P}$ (cf. section \ref{sec:generalization}).

\subsubsection{Optimization statement}
\label{sec:optimization}

We propose to choose $\mathbf{a}_{\rm hf}^k= \mathbf{a}_{\rm hf}(\mu^k)$ as a solution to 
\begin{subequations} 
\label{eq:optimization_statement}
\begin{equation}
\label{eq:optimization_statement_a}
\begin{array}{l}
\displaystyle{
\min_{\mathbf{a} \in \mathbb{R}^{M^{\rm hf}}} \,
\mathfrak{f}^k \left(\mathbf{a}  \right) \,  + \, \xi \big|   
\boldsymbol{\Psi}_{\mathbf{a}}^{\rm hf}  \big|_{H^2(\Omega)}^2, } \\[3mm]
\displaystyle{
{\rm s.t.} \;
\int_{\Omega} \, 
{\rm exp} \left( \frac{\epsilon  - \mathfrak{J}_{\mathbf{a}}^{\rm hf}(\mathbf{X})}{C_{\rm exp}} \right) \,  + \, 
{\rm exp} \left( \frac{\mathfrak{J}_{\mathbf{a}}^{\rm hf}(\mathbf{X}) - 1/\epsilon }{C_{\rm exp}} \right)
\, dX \leq \delta,} \\
\end{array}
\end{equation}
where  $\mathbf{f}^k(\mathbf{a}) =  \mathfrak{f} \left(\mathbf{a} , \mu^k, \bar{\mu}  \right)  $  --- which is here referred to as \emph{proximity measure} --- is given by
\begin{equation}
\label{eq:L2obj}
\mathfrak{f}\left(
\mathbf{a} , \mu, \bar{\mu} \right)
:= \; \int_{\Omega } \; \big\|   u_{\mu} \circ \boldsymbol{\Psi}_{\mathbf{a}}^{\rm hf} -  {u}_{\bar{\mu}}   \big\|_2^2 \, dX,
\end{equation} 
\end{subequations}
while  the $H^2$ seminorm is given by
$|  \mathbf{v}   |_{H^2(\Omega)}^2 
:= \sum_{i,j,k=1}^d \, \int_{\Omega} \, 
\left(   \widehat{\partial}_{i,j}^2 v_k  \,\right)^2 \, dX$ for all $\mathbf{v} \in 
H^2(\Omega; \mathbb{R}^d) $.  The constraint in \eqref{eq:optimization_statement_a}, which was introduced in 
\eqref{eq:condition_equivalent_weak}, weakly enforces that 
$\mathfrak{J}_{\mathbf{a}}^{\rm hf} \in [\epsilon, 1/\epsilon]$ and thus  that $\boldsymbol{\Psi}^{\rm hf}$ is a bijection from $\Omega$ into itself for all admissible solutions to \eqref{eq:optimization_statement_a}.

Since $| v  |_{H^2(\Omega)}=0$ for all linear polynomials, 
the penalty term measures deviations from linear maps: {due to the condition
$\boldsymbol{\Psi}_{\mathbf{a}}^{\rm hf}(\Omega) = \Omega$, we might further interpret the penalty as a measure of the deviations from the identity map.}
{The penalty in \eqref{eq:optimization_statement_a}  
can be further interpreted as a Tikhonov regularization, and 
has   the effect to control the gradient of the Jacobian $\mathfrak{J}_{\mathbf{a}}^{\rm hf}$
---
recalling the discussion in section \ref{sec:theory_stuff}, the latter is important to enforce  bijectivity. }
In addition,  we observe that
 $\mathfrak{f}^k(\mathbf{a})=0$ if and only if $u^k$ and $\bar{u}$ coincide in the mapped configuration; in particular, we have 
$\mathfrak{f}^k(\mathbf{0})=0$ if $\mu^k=\bar{\mu}$. The latter  implies  that $\mathbf{a}= \mathbf{0}$ is a solution to \eqref{eq:optimization_statement_a} for $\mu^k=\bar{\mu}$, provided that $\delta$ satisfies \eqref{eq:choice_Cdelta}.

The hyper-parameter $\xi$ balances accuracy
 --- measured by $\mathfrak{f}$ --- and smoothness of the mapping. 
We refer to the   results of section 
 \ref{sec:numerics_data_compression} for a numerical investigation of the sensivity of the solution to the choice of $\xi$: in our experience, 
the choice  of     $\xi$  is problem-dependent;
however,  the same value of $\xi$ can be considered for all training points.
{Furthermore, since in our code 
$\boldsymbol{\Psi}^{\rm hf}$ is a polynomial, computation of the penalty function is straightforward: we precompute and store the matrix $\mathbf{A}^{\rm reg} \in \mathbb{R}^{M_{\rm hf}, M_{\rm hf}}$
such that
$A_{m,m'}^{\rm reg} = ((\boldsymbol{\varphi}_{m'}^{\rm hf},      
\boldsymbol{\varphi}_{m}^{\rm hf} ))_{H^2(\Omega)}$ for $m,m'=1,\ldots, M_{\rm hf}$ 
---
$  ((\cdot,       \cdot ))_{H^2(\Omega)}$ is the bilinear form associated with 
$| \cdot |_{H^2(\Omega)}$
---
and then we compute 
$|  \boldsymbol{\Psi}_{\mathbf{a}}^{\rm hf}   |_{H^2(\Omega)}$ as
$\mathbf{a}^T \, \mathbf{A}^{\rm reg} \, \mathbf{a}$.}

\subsubsection{Implementation}

We resort to the Matlab routine \texttt{fmincon} 
\cite{MATLAB:2018b} to solve \eqref{eq:optimization_statement}: the routine relies on an interior point  method (\cite{byrd1999interior}) to find local minima of \eqref{eq:optimization_statement}. In our implementation, we first reorder the snapshots such that
\begin{subequations}
\label{eq:computational_tricks}
\begin{equation}
\mu^{(k+1)} := {\rm arg} \min_{\mu \in \Xi_{\rm train} \setminus  \{ \mu^{(i)}  \}_{i=1}^k   } \, \| \mu^{(k)}  - \mu \|_2, 
\quad
\mu^{(1)} = 
{\rm arg} \min_{\mu \in \Xi_{\rm train}   } \, \| \mu   - \bar{\mu} \|_2, 
\end{equation}
where $\Xi_{\rm train} = \{  \mu^k \}_{k=1}^{n_{\rm train}}$. Then, we choose the initial conditions $\mathbf{a}_0^{(k)}$ for the $k$-th optimization problem as
\begin{equation}
\mathbf{a}_0^{(1)} = \mathbf{0},
\quad
\mathbf{a}_0^{(k)} = \mathbf{a}_{\rm hf}^{({\rm ne}_k)},
\quad
{\rm ne}_k:=
{\rm arg} \min_{i \in \{1,\ldots, k-1\}}
\|  \mu^{(k)}  - \mu^{(i)}  \|_2,
\end{equation}
for  $k=2,\ldots,n_{\rm train}$.
Since the problem is non-convex and thus the solution is not guaranteed to be unique and to depend continuously on $\mu$,   $\mathbf{a}_{\rm hf}^{(k)}$ might be far from $\mathbf{a}_{\rm hf}^{({\rm ne}_k)}$ even if $\| \mu^{(k)} - \mu^{({\rm ne}_k)}  \|_2$ is small: this makes the generalization step (cf. section \ref{sec:generalization}) extremely challenging. For this reason, we propose to add the box constraints:
\begin{equation}
\label{eq:desperate_constraints}
-C_{\infty} \, \| \mu^{(k)} - \mu^{({\rm ne}_k)}  \|_2 \, \leq 
\left( \mathbf{a}  \right)_m  \, - \,
\left( \mathbf{a}_{\rm hf}^{({\rm ne}_k)}  \right)_m  \leq
C_{\infty} \, \| \mu^{(k)} - \mu^{({\rm ne}_k)}  \|_2,
\end{equation}
for $m=1,\ldots,M_{\rm hf}$.
In all our experiments, we set $C_{\infty} = 10$: for this choice of $C_{\infty}$, we have empirically found that this set of box constraints is not active at the local minima, for none of the  cases considered.
{Nevertheless, we envision that for more challenging problems constraint \eqref{eq:desperate_constraints} might be useful to improve the robustness of the approach.
}
\end{subequations}

In our numerical experiments, \texttt{fmincon} converges to local minima in $10^2-10^3$ iterations; the computational cost on a commodity laptop is $\mathcal{O}( 1-5{[\rm s]})$ for  all tests run. 
{Given the snapshot set $\{ u^k \}_k$, 
the cost per iteration is dominated by the computation of $u^k$ and $\nabla u^k$ in the mapped quadrature points $\{  \boldsymbol{\Psi}_{\mathbf{a}}^{\rm hf}(\mathbf{x}_q^{\rm qd}) \}_{q=1}^{N_{\rm q}}$ and by subsequent computation of the derivative of $\partial_{a_m} \mathfrak{f}^k$ for $m=1,\ldots,M_{\rm hf}$: for structured grids, this cost scales with $\mathcal{O} \left(N_{\rm q} \left( \log (N_{\rm hf}) + M_{\rm hf} \right)  \right)$.
} 

\subsubsection{Connection with optimal transport}
\label{remark:optimal_transport}
 
Let us assume that $u,\bar{u}$ are probability {densities} over $\Omega$, that is 
$u,\bar{u} \geq 0$ and $\int_{\Omega} u \, dx= \int_{\Omega} \bar{u} \, dx = 1$. Then, $\boldsymbol{\Phi}^{\rm opt}$ is an optimal transport map if it is a global minimizer of
\begin{equation}
\label{eq:optimal_transport}
\min_{\boldsymbol{\Phi}: \Omega \to \Omega} \; \int_{\Omega} \, 
\bar{u} (\mathbf{X}) \, \|
\mathbf{X}  \, -  \, \boldsymbol{\Phi}(\mathbf{X})  \|_2^2 \, dX \, \quad
{\rm s.t.} \; \;
\bar{u}(\mathbf{X}) = u( \boldsymbol{\Phi}(\mathbf{X}) ) \, \mathfrak{J}(\mathbf{X}) \; \; \forall \, \mathbf{X} \in \Omega,
\end{equation}
where $ \mathfrak{J} = {\rm det} \left(\widehat{\nabla} \boldsymbol{\Phi}  \right)$. A barrier method to solve \eqref{eq:optimal_transport} reads as 
\begin{equation}
\label{eq:optimal_transport_barrier}
\min_{\boldsymbol{\Phi}: \Omega \to \Omega} \; 
\int_{\Omega} \, \left(   \bar{u}(\mathbf{X})  \, - \, u( \boldsymbol{\Phi}(\mathbf{X}) ) \, \mathfrak{J}(\mathbf{X}) 
\right)^2 \, dX \, + \, \frac{1}{\lambda} \, 
\int_{\Omega} \, \bar{u}(\mathbf{X}) \, \|  \mathbf{X}  -  \boldsymbol{\Phi}(\mathbf{X})  \|_2^2 \, dX,
\end{equation}
where $\lambda \gg 1$. 

The first addend in \eqref{eq:optimal_transport_barrier} is closely linked to  $\mathfrak{f}$ in \eqref{eq:L2obj}, while the second term can be interpreted as a measure of the deviation of $\boldsymbol{\Phi}$ from the identity map, exactly as $|\boldsymbol{\Phi}  |_{H^2(\Omega)}$. Note, however, that there are important differences between \eqref{eq:optimization_statement} and \eqref{eq:optimal_transport_barrier}.
First,
  in \eqref{eq:optimal_transport_barrier}, the functions $u, \bar{u}$ should be probability {densities}, while in \eqref{eq:optimization_statement} $u, \bar{u}$  are  arbitrary real-valued functions in a  Hilbert spaces $\mathcal{U}$ contained in $L^2(\Omega)$. 
Second,  if $u, \bar{u} \geq 0$, solutions to \eqref{eq:optimization_statement} do not in general conserve mass. 
Third, if $u, \bar{u}$ are compactly supported in $\Omega$, solutions to \eqref{eq:optimal_transport} are not guaranteed to be locally invertible in $\Omega$.

\subsubsection{Generalization}
\label{sec:generalization}

Given the dataset of pairs $\{ (\mu^k, \mathbf{a}_{\rm hf}^k )\}_{k=1}^{n_{\rm train}}$, we resort to a multi-variate multi-target regression procedure to compute the mapping 
$\boldsymbol{\Phi}$
of the form \eqref{eq:registration_mapping_form}.
First, we   resort to $\| \cdot  \|_2$-POD  to determine a low-dimensional  approximation of $\{  \mathbf{a}_{\rm hf}^k \}_k$:
\begin{subequations}
\label{eq:generalization}
\begin{equation}
\label{eq:genPOD}
\mathbf{a}_{\rm hf}^k \approx \mathbf{U}_{\boldsymbol{\Phi}} \, \mathbf{a}^k,
\quad
\mathbf{U}_{\boldsymbol{\Phi}} \in \mathbb{R}^{M_{\rm hf},M},
\; \;
\mathbf{U}_{\boldsymbol{\Phi}}^T \mathbf{U}_{\boldsymbol{\Phi}}  = \mathbbm{1}_{M},
\; \;
M < M_{\rm hf}.
\end{equation}
Then, we build the regressors $\widehat{a}_{m}: \mathcal{P} \to \mathbb{R}$  based on the datasets
$\{ (\mu^k,   a_{\rm hf}^{m,k} ) $ 
$\}_{k=1}^{n_{\rm train}}$,  
$a_{\rm hf}^{m,k} := \left(
 \mathbf{U}_{\boldsymbol{\Phi}}^T \mathbf{a}_{\rm hf}^k  \right)_{m}$,
for $m=1,\ldots,M$;
 finally, we return the mapping
\begin{equation}
\label{eq:gen_final}
\boldsymbol{\Phi}_{\mu}(\mathbf{X}) \, = \, \mathbf{X} \, + \, \sum_{m=1}^{M} \, \left( \widehat{\mathbf{a}}_{\mu} \right)_m \,  
\boldsymbol{\varphi}_m (\mathbf{X}),
 \;
\left\{
\begin{array}{l}
\displaystyle{
 \widehat{\mathbf{a}}:=
\left[
\widehat{a}_{1}, \ldots, \widehat{a}_{M}  \right]^T
} \\[3mm]
\displaystyle{\boldsymbol{\varphi}_m (\mathbf{X}) \, = \sum_{m'=1}^{M_{\rm hf}} \, \left( \mathbf{U}_{\boldsymbol{\Phi}} \right)_{m',m} \,
\boldsymbol{\varphi}_{m'}^{\rm hf}(\mathbf{X})
} \\ 
\end{array}
\right. 
\end{equation}
\end{subequations}

Some comments are in order. First, application of POD  in \eqref{eq:genPOD} leads to a (potentially substantial) reduction of the  size of the mapping expansion and thus ultimately to a reduction of online costs;  application of POD  also reduces  the importance  of the choice of $M_{\rm hf}$ in \eqref{eq:family_mappings} since it provides an automatic way of choosing the size of the expansion at the end of the offline stage. Note that, since POD is linear, the resulting expansion satisfies boundary conditions in
\eqref{eq:boundary_conditions_mapping}: 
applying Proposition \ref{th:application_unit_square}, we
{conclude that there exists a ball $B=\mathcal{B}_{r_{\epsilon}}(\mathbf{0})$ in $\mathbb{R}^M$ for which $\boldsymbol{\Phi}_{\mu}$ is bijective from $\Omega$ into itself if 
$\widehat{\mathbf{a}}_{\mu} \in B$.
}
Second,  any linear or nonlinear strategy for multivariate regression can be employed to construct  $\widehat{\mathbf{a}}$ based on
$\{ (\mu^k, \mathbf{a}_{\rm hf}^k ) \}_{k=1}^{n_{\rm train}}$:
in this work, we resort to a kernel-based Ridge regression procedure  based on inverse multiquadric RBFs (\cite{wendland2004scattered}). 
Third, we remark that  the mapping $\boldsymbol{\Phi}_{\mu}$ is not guaranteed to be bijective for all $\mu \in \mathcal{P}$, particularly for small-to-moderate values of $n_{\rm train}$: this represents a major issue of the methodology that will be addressed in a subsequent work. We remark, nevertheless, that our approach is able to provide accurate and stable results for $n_{\rm train}= \mathcal{O}(10^2)$ for all test cases considered in this paper.
Finally, given the eigenvalues $\{ \lambda_i  \}_{i=1}^{n_{\rm train}}$ of the POD Kernel matrix
$\mathbf{C}_{k,k'}:= \mathbf{a}^k \cdot \mathbf{a}^{k'}$, we choose $M$ in \eqref{eq:genPOD} 
based on the criterion \eqref{eq:POD_cardinality_selection}, 
for some problem-dependent tolerance $tol_{\rm POD}>0$ that will be specified in the numerical sections.

\subsubsection{Review of the computational procedure}

Algorithm \ref{alg:registration} summarizes the pieces of the general approach proposed in this paper.

\begin{algorithm}[H]                      
\caption{Registration algorithm}     
\label{alg:registration}                           
 \small
\emph{Inputs:}  $\{ (\mu^k,  u^k = u_{\mu^k}) \}_{k=1}^{n_{\rm train}} \subset  \mathcal{P} \times \mathcal{M}_{\rm u}$ snapshot set,
$\bar{u}  \in \mathcal{M}_{\rm u}$  reference field.

\emph{Output:} 
parametric mapping
$\boldsymbol{\Phi}: \Omega \times \mathcal{P} \to \mathbb{R}^d$.

 \normalsize 

\begin{algorithmic}[1]
\State
Definition of the parametric function 
$\boldsymbol{\Psi}^{\rm hf}$ 
(cf. \eqref{eq:prescribed_mapping_form}-\eqref{eq:tensorized_polynomials}).
\vspace{3pt}

\State
Computation of $\mathbf{a}_{\rm hf}^k$ based on the pair of fields $(u^k, \bar{u})$ (cf. section \ref{sec:optimization}).

\hfill
$\rightarrow \{ ( \mu^k, \mathbf{a}_{\rm hf}^k) \}_{k=1}^{n_{\rm train}}$

\State
Generalization: $\{ (\mu^k, \mathbf{a}_{\rm hf}^k) \}_{k=1}^{n_{\rm train}} \to 
\widehat{\mathbf{a}}: \mathcal{P} \to \mathbb{R}^M, \,
\{ \boldsymbol{\varphi}_m \}_{m=1}^M$
(cf. section \ref{sec:generalization}).
\vspace{3pt}

\State
Return   $\boldsymbol{\Phi}$ (cf. \eqref{eq:generalization}).
\vspace{3pt}
\end{algorithmic}

\end{algorithm}

\subsection{Two-level approximations and $N$-$M$ Kolmogorov widths}
\label{sec:NM_width}

We can interpret Lagrangian approximations as two-level approximations of parametric fields, $u_{\mu}  \approx \widehat{u}_{\mu} \circ  \boldsymbol{\Phi}_{\mu}^{-1} $: the inner layer corresponds to the mapping process, while the outer layer is associated with the linear approximation. This observation highlights the connection between Lagrangian approaches and deep networks. 
As for deep vs shallow networks (\cite{poggio2017and}), the use of Lagrangian-based pMOR methods for a given class of problems should be supported by evidence of their superior approximation power compared to linear approaches. 
While performance of linear methods can be theoretically measured through the Kolmogorov $N$-width $d_N(\mathcal{M}_{\rm u})$ in \eqref{eq:kolmogorov_nwidth}, we here introduce the notion of Kolmogorov $N-M$ width  (see also \cite{rim2019nonlinear})
to assess performance of Lagrangian approximations.

{In view of the definition of the $N-M$ width, given the $M$-dimensional space $\mathcal{Y}_M:=  {\rm span} \{  \boldsymbol{\varphi}_m \}_{m=1}^M \subset {\rm Lip} (\Omega; \mathbb{R}^d )$, and the tolerance $\epsilon \in (0,1)$, we define $\mathcal{Y}_M^{\rm bis, \epsilon}$ such that
\begin{subequations}
\label{eq:kolmogorov_NM_width}
\begin{equation}
\label{eq:kolmogorov_NM_width_a}
\mathcal{Y}_M^{\rm bis, \epsilon}:=  \{
\boldsymbol{\Phi} = \mathbf{X} + \boldsymbol{\varphi}(\mathbf{X}): \; \boldsymbol{\Phi}(\Omega)=\Omega, \quad
{\rm det} \left( \widehat{\nabla} \boldsymbol{\Phi} \right) \geq \epsilon \; {\rm a.e.} \}.
\end{equation}
Then, given the manifold $\mathcal{M}_{\rm u}$ in the Banach space $\mathcal{U}$, we define the $N-M$ width as
\begin{equation}
d_{N,M,\epsilon}\left(\mathcal{M}_{\rm u}; \, \mathcal{U}  \right) \, = \, 
\inf_{
\substack{\mathcal{Y}_M \subset {\rm Lip}(\Omega; \mathbb{R}^d),     \; {\rm dim}(\mathcal{Y}_M) = M \\ 
\mathcal{Z}_N \subset \mathcal{U}, \; {\rm dim}(\mathcal{Z}_N) = N \\
}}
\; 
\sup_{w \in \mathcal{M}_{\rm u}} \, 
\inf_{
\substack{\boldsymbol{\varphi} \in \mathcal{Y}_M^{\rm bis, \epsilon} \\ 
z \in \mathcal{Z}_N \\
}}
\;    \| w \circ \boldsymbol{\varphi} - z  \|_{\mathcal{U}}
\end{equation}
where the second argument stresses the dependence on the norm of the  space $\mathcal{U}$.
\end{subequations} 
}
Note that, for any $N, M \geq 1$, 
$d_{N,M,\epsilon}(\mathcal{M}_{\rm u};  \, \mathcal{U} ) \leq  d_N \left(  {\mathcal{M}}_{\rm u}; \, \mathcal{U}   \right)$, 
$d_{N,0,\epsilon}(\mathcal{M}_{\rm u}) =  d_N \left(  {\mathcal{M}}_{\rm u} ; \mathcal{U}    \right)$, and 
$d_{N,M,\epsilon=1}(\mathcal{M}_{\rm u}; \mathcal{U}  ) =  d_N \left(  {\mathcal{M}}_{\rm u}; \mathcal{U}     \right)$.

We can provide estimates of the $N-M$ width for representative solution manifolds:
we refer to section \ref{sec:NM_widths_examples} for the proofs.
First, we consider  the solution  $u_{\mu}$ to the one-dimensional problem
\begin{subequations}
\label{eq:NM_boundary_layer}
\begin{equation}
-\partial_{x x} \, u_{\mu} \, + \, \mu^2 \, u_{\mu} \, = \, 0,
\; \; {\rm in} \; \; \Omega_{\rm 1D} = (0,1),
\quad u_{\mu}(0)=1, \; \; \partial_x \, u_{\mu}(1) = 0,
\end{equation}
where $\mathcal{P}=[\mu_{\rm min}, \mu_{\rm max}=\epsilon^{-2} \mu_{\rm min}]$ and $e^{\mu_{\rm min}} \gg 1$; for this problem, we find \eqref{eq:NM_boundary_layer}
\begin{equation}
d_{N=1, M=1, \epsilon}\left(\mathcal{M}_{\rm u}, \; L^2(\Omega_{\rm 1D})  \right)
\, 
\leq     
\frac{1}{\sqrt{
1 + \epsilon}}
\, {\rm exp} \left( 
- \frac{\mu_{\rm min}}{1+\epsilon}
 \right) \, + \,
 \frac{e^{-  \mu_{\rm min}}}{1 + e^{-2 \mu_{\rm min}}}
\end{equation}
\end{subequations}
Second, we consider 
  $u_{\mu}= {\rm sign} \left( x  - \mu \right)$, $x \in \Omega_{\rm 1D}$, $\mu \in \mathcal{P}=[1/3, 2/3]$: in this case, we find  
\begin{equation}
\label{eq:NM_shock_wave}
\begin{array}{l}
\displaystyle{
d_{N}\left(\mathcal{M}_{\rm u}, \; L^2(\Omega_{\rm 1D})  \right) \, = \,
\mathcal{O} \left( \frac{1}{\sqrt{N}} \right),
}
\\[3mm]
\displaystyle{
d_{N, M, \epsilon}\left(\mathcal{M}_{\rm u}, \; L^2(\Omega_{\rm 1D})  \right)
\, = \, 0,
\; \; \;
\forall \, N,M \geq 1, \; \epsilon\leq \frac{2}{3}.
}
\\ 
\end{array}
\end{equation}
Note that the manifold in \eqref{eq:NM_boundary_layer} presents a boundary layer  at $x=0$, while the manifold  in \eqref{eq:NM_shock_wave} is associated with an advection-reaction problem with time interpreted as parameter: therefore, these estimates suggest that Lagrangian methods might be effective for problems with boundary layers and/or travelling waves. 
As third and last example,  we consider  the parametric field 
\begin{subequations}
\label{eq:NM_2D}
\begin{equation}
u: \mathcal{P} \to L^2((0,1)^2),
\quad
u_{\mu}(\mathbf{x}) \, = \, 
\left\{
\begin{array}{ll}
0 & {\rm if} \; x_2 < f_{\mu}(x_1) \\[3mm] 
1 & {\rm if} \; x_2 > f_{\mu}(x_1) \\
\end{array}
\right.
\end{equation}
where $f_{\mu} \in {\rm Lip}([0,1])$ and 
$f_{\mu}([0,1]) \subset [\delta, 1 - \delta]$, for some $\delta \in (0,1/2)$ and for all $\mu \in \mathcal{P}$. In this case, given $0 < \epsilon < 2 \delta$, we find
\begin{equation}
d_{N,M, \epsilon}
\left( \mathcal{M}_{\rm u}, L^1(\Omega) \right)
\leq  \, C_{\delta} \, \frac{
d_M\left(\mathcal{M}_{\rm f}, {\rm Lip}([0,1])\right)
}{\sqrt{N}}.
\end{equation}
where  
$\mathcal{M}_{\rm f} := \{ f_{\mu}: \mu \in \mathcal{P}  \}$,
 $\| v \|_{{\rm Lip}([0,1])} := \| v  \|_{L^{\infty}([0,1])} + 
\| v' \|_{L^{\infty}([0,1])}$, 
and the multiplicative constant $C_{\delta}$ depends on $\delta$.
The latter estimate suggests a multiplicative effect between $N$- and $M$- convergence: we further investigate this aspect in the numerical results.
\end{subequations}

\subsection{Integration within the offline/online paradigm}
\label{sec:offline_online_pMOR}

Algorithm  \ref{alg:abstract_lagrange} summarizes the key steps of pMOR  procedures for \eqref{eq:standard_PMOR_setting} based on Lagrangian nonlinear data compression. 
First, we generate a set of snapshots $u^1,\ldots, u^{n_{\rm train}} \in \mathcal{M}_{\rm u}$ associated with the parameters $\mu^1,\ldots, \mu^{n_{\rm train}} \in \mathcal{P}$:
the field $u$ might be the solution to the PDE or a set of model coefficients. Then, we use the snapshot set $\{ u^k \}_k$ to generate the problem-dependent mapping $\boldsymbol{\Phi}$, and  we recast the problem in the form \eqref{eq:mapped_PMOR_setting}. Then, we use standard pMOR techniques to generate a ROM for \eqref{eq:mapped_PMOR_setting}. During the online stage, given a new parameter value $\mu^{\star} \in \mathcal{P}$, we query the ROM to estimate the output of interest and the associated prediction error. 

\begin{algorithm}[H]                      
\caption{pMOR with Lagrangian nonlinear data compression. Offline/online decomposition.}     
\label{alg:abstract_lagrange}                           

\textbf{Offline stage}

\begin{algorithmic}[1]
\State
Build the dataset $\{(\mu^k, u^k = u_{\mu^k} )  \}_{k=1}^{n_{\rm train}}$.
\vspace{3pt}

\State
Compute the mapping $\{ \boldsymbol{\Phi}_{\mu}: \mu \in \mathcal{P} \}$ based on
 $\{(\mu^k, u^k)  \}_{k=1}^{n_{\rm train}}$
 (cf. section \ref{sec:registration}).
\vspace{3pt}

\State
Generate the ROM for \eqref{eq:mapped_PMOR_setting}.
\end{algorithmic}

\textbf{Online stage}

\begin{algorithmic}[1]
\State
Given $\mu^{\star}$, query the ROM to estimate 
$\hat{\mathbf{y}}_{\mu^{\star}}$ and the error
$\|  \hat{\mathbf{y}}_{\mu^{\star}}  -   {\mathbf{y}}_{\mu^{\star}}    \|_2$.
\vspace{3pt}
\end{algorithmic}
\end{algorithm}

\subsubsection{{A POD-Galerkin ROM for advection-diffusion-reaction problems}}

We discuss the application of our approach to the advection-diffusion-reaction problem:
\begin{equation}
\label{eq:ADR_general}
\left\{
\begin{array}{ll}
\displaystyle{
 - \, 
\nabla \cdot \left(  \mathbf{K}_{\mu} \, \nabla z_{\mu}  -  \mathbf{c}_{\mu} \, z_{\mu}  \right) \, + \,
\sigma_{\mu} \, z_{\mu} \,  = \, f_{\mu}
}
& {\rm in} \,  \Omega,
\\[3mm]
z_{\mu}= z_{\rm D, \mu}
& {\rm on} \, \Gamma_{\rm D} \subset  \partial \Omega \\[3mm]
\left( \mathbf{K}_{\mu} \, \nabla z_{\mu}  - \mathbf{c}_{\mu} \, z_{\mu} \right) \,\cdot \,  \mathbf{n} \, = \, g_{\mu}
& {\rm on} \, \Gamma_{\rm N} =   \partial \Omega  \setminus \Gamma_{\rm D}\\
\end{array}
\right.
\end{equation}
The FE discretization of \eqref{eq:ADR_general}
(we omit the superscript $(\cdot)^{\rm hf}$ to shorten notation) reads as:
find  $z_{\mu} \in \mathcal{X}$ such that
\begin{equation}
\label{eq:ADR_FEM_general}
\begin{array}{rl}
\displaystyle{\mathcal{G}_{\mu}(z_{\mu}, v) \,= }
 &
\displaystyle{ \sum_{k=1}^{n_{\rm el}} \,
\int_{\texttt{D}^k} \,   \boldsymbol{\Upsilon}_{\mu}^{\rm el} \cdot A^{\rm el}(z_{\mu}, v) \,  - f_{\mu} \, v \, dx } 
 \\[3mm]
 &
 \displaystyle{  + \,
\int_{\partial \texttt{D}^k} \, 
 \boldsymbol{\Upsilon}_{\mu}^{\rm ed} \cdot A^{\rm ed}(z_{\mu}, v) \,  - f_{\mu}^{\rm ed} \, v \, dx \, = \, 0,  
 \quad
 \forall \, v \in \mathcal{X}^{\rm hf}},
 \\ 
\end{array}
\end{equation}
where $ \boldsymbol{\Upsilon}_{\mu}^{\rm el}, \boldsymbol{\Upsilon}_{\mu}^{\rm ed}, f_{\mu}, f_{\mu}^{\rm ed}$ are parameter-dependent coefficients and $A^{\rm el}, A^{\rm ed}$ are parameter - independent local bilinear operators.

Given the parametric mapping $\boldsymbol{\Phi}: \Omega \times \mathcal{P} \to \Omega$, 
provided that the data and the mapping are sufficiently smooth,
we can prove that the mapped field
$\tilde{z}_{\mu} = z_{\mu} \circ \boldsymbol{\Phi}_{\mu}$ solves a problem of the form \eqref{eq:ADR_general} with coefficients:
\begin{equation}
\label{eq:mapped_coefficients_general}
\left\{
\begin{array}{lll}
\displaystyle{
\mathbf{K}_{\mu}^{\star}:= \mathfrak{J}_{\mu} \, \widehat{\nabla} \boldsymbol{\Phi}_{\mu}^{-1} \, \widetilde{\mathbf{K}}_{\mu} \,
\widehat{\nabla} \boldsymbol{\Phi}_{\mu}^{-T},}
&
\displaystyle{
\mathbf{c}_{\mu}^{\star} :=
\mathfrak{J}_{\mu} \, \widehat{\nabla} \boldsymbol{\Phi}_{\mu}^{-1}  \widetilde{\mathbf{c}}_{\mu},}
&
\displaystyle{
\sigma_{\mu}^{\star} := \mathfrak{J}_{\mu} \, \widetilde{\sigma}_{\mu},
}
\\[3mm]
\displaystyle{
f_{\mu}^{\star} := \mathfrak{J}_{\mu} \, \tilde{f}_{\mu}^{\star},
}
&
\displaystyle{
g_{\mu}^{\star}:=  \mathfrak{J}_{\mu} \, \| \widehat{\nabla} \boldsymbol{\Phi}_{\mu}^{-T} \widehat{\mathbf{n}}  \|_2 \,
\tilde{g}_{\mu},
}
&
\displaystyle{
z_{\rm D, \mu}^{\star}
=
\tilde{z}_{\rm D, \mu},
}
\\
\end{array}
\right.
\end{equation}
where the symbol $\widetilde{(\cdot)}$ indicates the composition with $\boldsymbol{\Phi}_{\mu}$, and 
$\widehat{\mathbf{n}}$ denotes the normal in the reference configuration --- which coincides with $\mathbf{n}$ on $\partial \Omega$, provided that $\boldsymbol{\Phi}_{\mu}(\Omega)=\Omega$.
Exploiting \eqref{eq:ADR_FEM_general}, we find that the mapping process  will simply lead to different expressions for $ \boldsymbol{\Upsilon}_{\mu}^{\rm el}, \boldsymbol{\Upsilon}_{\mu}^{\rm ed}, f_{\mu}, f_{\mu}^{\rm ed}$, which can be efficiently computed. 

Problem \eqref{eq:ADR_FEM_general} with coefficients in \eqref{eq:mapped_coefficients_general} is not expected to be parametrically-affine (see, e.g., \cite{quarteroni2015reduced}); 
To devise an online-efficient ROM to approximate $\tilde{z}_{\mu}$, we should thus introduce the parametrically-affine approximations 
$\boldsymbol{\Upsilon}_{\mu}^{\rm el,eim}$,
$f_{\mu}^{\rm eim}$,
$\boldsymbol{\Upsilon}_{\mu}^{\rm ed,eim}$ , 
 $f_{\mu}^{\rm ed,eim}$
of 
$\boldsymbol{\Upsilon}_{\mu}^{\rm el}$, $f_{\mu}$, 
$\boldsymbol{\Upsilon}_{\mu}^{\rm ed}$, 
$f_{\mu}^{\rm ed}$,
\begin{equation}
\label{eq:eim_ADR}
\begin{array}{ll}
\displaystyle{
\boldsymbol{\Upsilon}_{\mu}^{\rm el, eim}(\mathbf{x}):=
\sum_{q=1}^{Q_{\rm a, el}} \, \Theta_{\mu, q}^{\rm el, a} \, \boldsymbol{\Upsilon}_q^{\rm el}(\mathbf{x}),
}
&
\displaystyle{
\boldsymbol{\Upsilon}_{\mu}^{\rm ed, eim}(\mathbf{x}):=
\sum_{q=1}^{Q_{\rm a, ed}} \, \Theta_{\mu, q}^{\rm ed, a} \, \boldsymbol{\Upsilon}_q^{\rm ed}(\mathbf{x}),
}
\\[3mm]
\displaystyle{
f_{\mu}^{\rm  eim}(\mathbf{x}):=
\sum_{q=1}^{Q_{\rm f, el}} \, \Theta_{\mu, q}^{\rm el, f} \, f_q^{\rm el}(\mathbf{x}),
}
&
\displaystyle{
f_{\mu}^{\rm ed, eim}(\mathbf{x}):=
\sum_{q=1}^{Q_{\rm f, ed}} \, \Theta_{\mu, q}^{\rm ed, f} \, f_q^{\rm ed}(\mathbf{x}).
}
\\
\end{array}
\end{equation}
We here resort to the 
empirical interpolation method (EIM, \cite{barrault2004empirical}) 
and to  one of  its extensions to vector-valued fields (cf. \cite[Appendix B]{taddei2018offline}): we refer to section \ref{sec:eim} for further details concerning the implementation of EIM.
Then we substitute these approximations in \eqref{eq:ADR_FEM_general}  to obtain a parametrically-affine surrogate of $\mathcal{G}_{\mu}$:
$$
\begin{array}{rl}
\displaystyle{
\mathcal{G}_{\mu}^{\rm eim}(z_{\mu}^{\rm eim}, v) \,=
}
&
\displaystyle{
 \sum_{k=1}^{n_{\rm el}} \,
\int_{\texttt{D}^k} \, 
 \boldsymbol{\Upsilon}_{\mu}^{\rm el,eim} \cdot A^{\rm el}(z_{\mu}^{\rm eim}, v) \,  - f_{\mu}^{\rm eim} \, v \, dx
}
\\[3mm]
&
\displaystyle{
\, + \,
\int_{\partial \texttt{D}^k} \, 
 \boldsymbol{\Upsilon}_{\mu}^{\rm ed,eim} \cdot A^{\rm ed}(z_{\mu}^{\rm eim}, v) \,  - f_{\mu}^{\rm ed,eim} \, v \, dx
\, = \, 0,
\quad
\forall \, v \in \mathcal{X}.
}
\end{array}
$$
Then, given the reduced space $\mathcal{Z}_N := {\rm span}\{  \zeta_n \}_{n=1}^N \subset \mathcal{X}$, we define the online-efficient Galerkin ROM: 
\begin{equation}
\label{eq:galerkin_ROM_ADR}
{\rm find} \, \widehat{z}_{\mu} \in \mathcal{Z}_N: \; 
\mathcal{G}_{\mu}^{\rm eim}\left(\widehat{z}_{\mu}, v \right) \, = \, 0
\quad
\forall \, v \in \mathcal{Z}_N.
\end{equation}

We  resort to POD to build the reduced space $\mathcal{Z}_N$.  
Note  that the weak or strong Greedy algorithms could also be used to build the space $\mathcal{Z}_N$; similarly, we might also rely on Petrov-Galerkin (minimum residual) projection to estimate the solution, and we might also consider several other hyper-reduction techniques to achieve online efficiency. Since the focus of this paper is to assess the performance of the mapping procedure, we do not further discuss these aspects in the remainder.  

We conclude this section with three remarks.
 
 \begin{remark}
\label{remark:a_posteriori_error}
\textbf{Error estimation.}
As discussed in the introduction, in pMOR it is key to estimate the prediction error. For second-order elliptic problems, if we consider the norm\footnote{Similar estimates can be obtained for other norms.}
$\| \cdot \|:= \sqrt{ \int_{\Omega} \, \| \nabla \cdot \|_2^2 \, dx}$, it is straightforward to verify that 
$$
\mathfrak{c}_{\mu} \|  \widetilde{z}_{\mu} - \widehat{z}_{\mu}  \| \leq
\|   {z}_{\mu} - \widehat{z}_{\mu} \circ \boldsymbol{\Phi}_{\mu}^{-1} \|
\leq
\mathfrak{C}_{\mu}   \|  \widetilde{z}_{\mu} - \widehat{z}_{\mu}  \|,
$$
where $\widehat{z}_{\mu}$ denotes a generic approximate solution to \eqref{eq:mapped_PMOR_setting}, 
$\mathfrak{c}_{\mu} = \min_{\mathbf{X} \in \overline{\Omega}} \, \lambda_{\rm min}(\mathbf{K}_{\mu})$,
$\mathfrak{C}_{\mu} = \max_{\mathbf{X} \in \overline{\Omega}} \, \lambda_{\rm max}(\mathbf{K}_{\mu})$, and 
$\mathbf{K}_{\mu}=  \, \mathfrak{J}_{\mu} \, \widehat{\nabla} \boldsymbol{\Phi}_{\mu}^{-1} \, \widehat{\nabla} \boldsymbol{\Phi}_{\mu}^{-T}$.
Provided that $\mathfrak{c}_{\mu},\mathfrak{C}_{\mu}= \mathcal{O}(1)$, traditional residual-based error estimates (see, e.g., \cite[Chapter 3]{quarteroni2015reduced})
 in the mapped configuration can be used to sharply  bound the prediction error 
$\|   {z}_{\mu} - \widehat{z}_{\mu} \circ \boldsymbol{\Phi}_{\mu}^{-1} \|$.
\end{remark}
 
\begin{remark}
\textbf{Online computation of $\tilde{z}_{\mu}$.}
The online solution to the mapped problem involves 
(i) the evaluation of the mapping $\boldsymbol{\Phi}_{\mu}$ and of its gradient  in the $(Q_{\rm a,el} + Q_{\rm a,ed} + Q_{\rm f,el} + Q_{\rm f,ed})$ interpolation points, and
(ii) the assembling of the reduced system and its solution.
The second step scales  with $\mathcal{O}( (Q_{\rm a,el} + Q_{\rm a,ed}) N^2 + N^3 )$, while the first step depends on the supervised learning algorithm used to estimate the $M$ mapping coefficients $\widehat{\mathbf{a}}_{\mu}$: for RBF approximations, plain implementations require 
$\mathcal{O}(n_{\rm train} M)$ floating point operations but several acceleration techniques are now available to dramatically reduce the costs
(see \cite{roussos2005rapid,wendland2002fast}).
{It is difficult to establish a connection between the size of the EIM expansions in the physical and reference configuration: if some of the coefficients are non-affine, the mapping might have the effect of simplifying hyper-reduction  (cf. section \ref{sec:model_pb_diff}); if most or all coefficients are affine in parameter,
the mapping process likely leads to much longer expansions (see, e.g., the example 
in section  \ref{sec:hyperbolic}).
}
\end{remark}

\begin{remark}
\textbf{Pointwise estimation of $z_{\mu}$.}
Computation of $\hat{z}_{\mu} \circ \boldsymbol{\Phi}_{\mu}^{-1}(\mathbf{x})$ for a given $\mathbf{x} \in \Omega$ involves the  evaluation  of $\boldsymbol{\Phi}_{\mu}^{-1}(\mathbf{x})$, which requires the solution to the nonlinear problem
$$
\min_{\mathbf{X} \in \overline{\Omega} } \| \boldsymbol{\Phi}_{\mu}(\mathbf{X}) - \mathbf{x}  \|_2.
$$
In this work, we do not address the issue of how to efficiently evaluate $\hat{z}_{\mu} \circ \boldsymbol{\Phi}_{\mu}^{-1}$.
\end{remark}


\subsection{Numerical results}
\label{sec:numerics_data_compression}

\subsubsection{Approximation of boundary layers}
\label{sec:boundary_layers}

\subsubsection*{Problem statement}

We consider the diffusion-reaction problem:
\begin{equation}
\label{eq:md_pb_BL}
\left\{
\begin{array}{ll}
-\Delta z_{\mu} \,  + \, \mu^2 \, z_{\mu} \, = \, 0 &
{\rm in} \, \Omega=(0,1)^2 \\[2mm]
z_{\mu} \,
= \, 1 &
{\rm on} \, \Gamma_{\rm D}:=
\{  
\mathbf{x} \in \partial \Omega : \; 
x_1=0 \, {\rm or} \, x_2=0 \}, \\[2mm]
\partial_n \, z_{\mu} \,
= \, 0 &
{\rm on} \, \Gamma_{\rm N}:=
\partial \Omega \setminus \Gamma_{\rm D},
 \\
\end{array}
\right.
\end{equation}
where $\mu \in \mathcal{P}=[\mu_{\rm min}, \mu_{\rm max}]$, $\mu_{\rm min}=20, \mu_{\rm max}=200$.
For large values of $\mu$, the solution exhibits a boundary layer at $\Gamma_{\rm D}$.
By exploiting a standard argument, we can derive the variational formulation for the lifted field
$\mathring{z}_{\mu}:= z_{\mu} - 1 \in \mathcal{X}:= H_{0,\Gamma_{\rm D}}^1(\Omega)$:
\begin{equation}
\label{eq:md_pb_BL_weak}
\mathcal{G}_{\mu}(\mathring{z}_{\mu}, v    )
:=
\int_{\Omega} \, \nabla \, \mathring{z}_{\mu} \cdot \nabla  v \, dx \, + \mu^2 \, \int_{\Omega} \left(  \mathring{z}_{\mu} + 1 \right) \, v \, dx \, = \, 0
\quad
\forall \, v \in \mathcal{X}.
\end{equation}
Given the parametric mapping $\boldsymbol{\Phi}_{\mu}: \Omega \to \Omega$ such that
 $\boldsymbol{\Phi}_{\mu}(\Gamma_{\rm D})= \Gamma_{\rm D}$, 
 we find that $\widetilde{\mathring{z}}_{\mu}
= \mathring{z}_{\mu} \circ \boldsymbol{\Phi}_{\mu} \in \mathcal{X}$ satisfies:
\begin{equation}
\label{eq:md_pb_BL_weak_mapped}
\mathcal{G}_{\mu,\Phi}(\widetilde{\mathring{z}}_{\mu}, v    )
:=
\int_{\Omega} \,
\mathbf{K}_{\mu}^{\star} \, 
 \widehat{\nabla} \,
\widetilde{\mathring{z}}_{\mu}
 \cdot \widehat{\nabla} v \, dX \, +   \, \int_{\Omega} 
\sigma_{\mu}^{\star} \,  \left(  \widetilde{\mathring{z}}_{\mu} + 1 \right) \, v \, dX \, = \, 0
\quad
\forall \, v \in \mathcal{X},
\end{equation}
where $\mathbf{K}_{\mu}^{\star} =  \, \mathfrak{J}_{\mu} \, \widehat{\nabla} \boldsymbol{\Phi}_{\mu}^{-1} \, \widehat{\nabla} \boldsymbol{\Phi}_{\mu}^{-T}$ and
$\sigma_{\mu}^{\star} = \mu^2 \mathfrak{J}_{\mu}$. We here resort to a continuous P3 FE discretization with $N_{\rm hf}= 11236$ degrees of freedom on a structured triangular mesh. The grid is refined close to the boundary $\Gamma_{\rm D}$, to accurately capture the boundary layer.

\subsubsection*{Construction of the mapping}

We apply the registration procedure presented in section \ref{sec:registration} to the solution itself ($u_{\mu}= z_{\mu}$).
We set $\bar{\mu}= \sqrt{ \mu_{\rm min} \, \mu_{\rm max}  }$,
 $\xi=10^{-10}$, and we consider a polynomial expansion with   $\overline{M}=8$ in \eqref{eq:tensorized_polynomials} ($M_{\rm hf}=128$).
{We refer to section \ref{sec:NM_widths_examples} for an heuristic  motivation of the choice of $\bar{\mu}$.} 
  To build the regressor $\widehat{\mathbf{a}}:\mathcal{P} \to \mathbb{R}^M$, we consider $n_{\rm train}=70$ log-equispaced parameters in $\mathcal{P}$, 
 and we set $tol_{\rm pod}=10^{-4}$ in \eqref{eq:POD_cardinality_selection}:
 for this choice of the parameters, the procedure   returns an affine expansion with $M= 5$ terms.

{Figure  \ref{fig:BL_eigen} shows the 
behavior  of $\|  \cdot \|_2$-POD eigenvalues associated with 
$\{  \mathbf{a}_{\mu^i}^{\rm hf} \}_{i=1}^{n_{\rm train}}$, and the 
$L^2$-POD eigenvalues associated with $\{  \mathbf{K}_{\mu^i}^{\star}\}_{i=1}^{n_{\rm train}}$.
 }
 We remark that the decay of the POD eigenvalues cannot be directly related to the behavior of the Kolmogorov $N$-width; nevertheless, 
POD eigenvalues   provide an heuristic measure of the linear complexity of parametric manifolds and are thus shown in this paper to investigate  performance of the registration algorithm.
 
 \begin{figure}[h!]
\centering
 \subfloat[] {\includegraphics[width=0.4\textwidth]
 {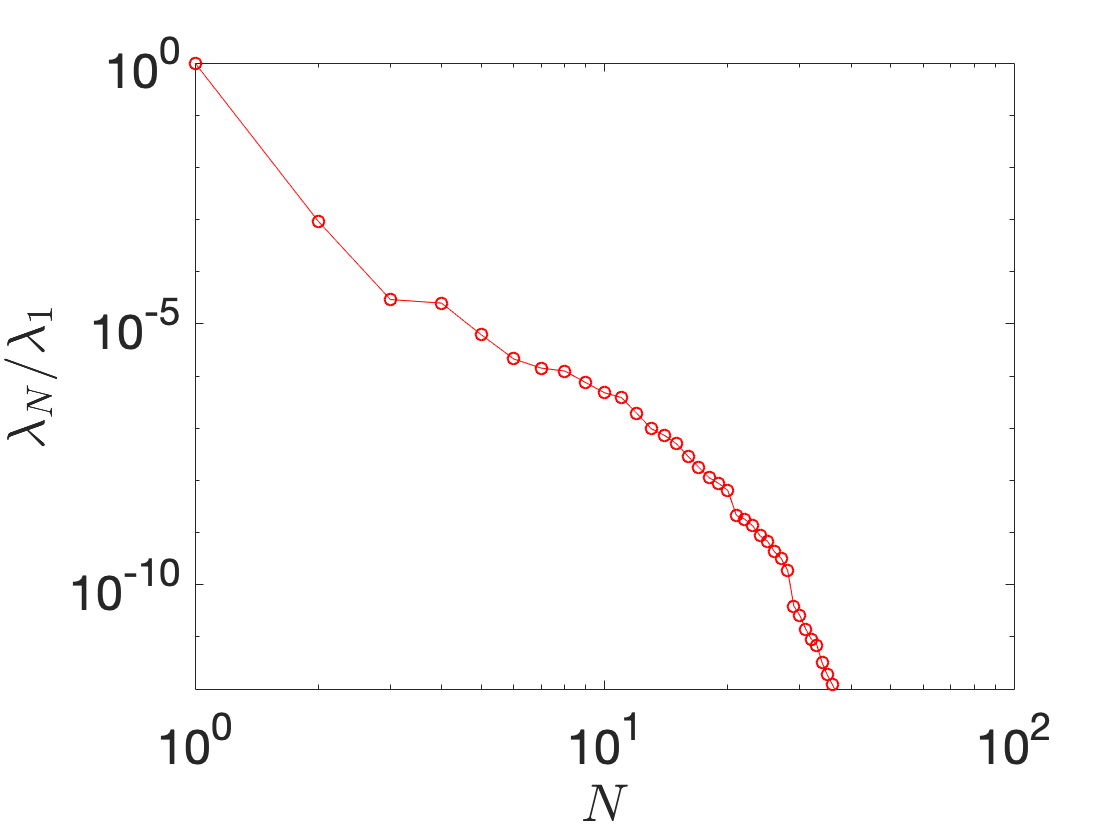}}  
     ~ ~
\subfloat[] {\includegraphics[width=0.4\textwidth]
 {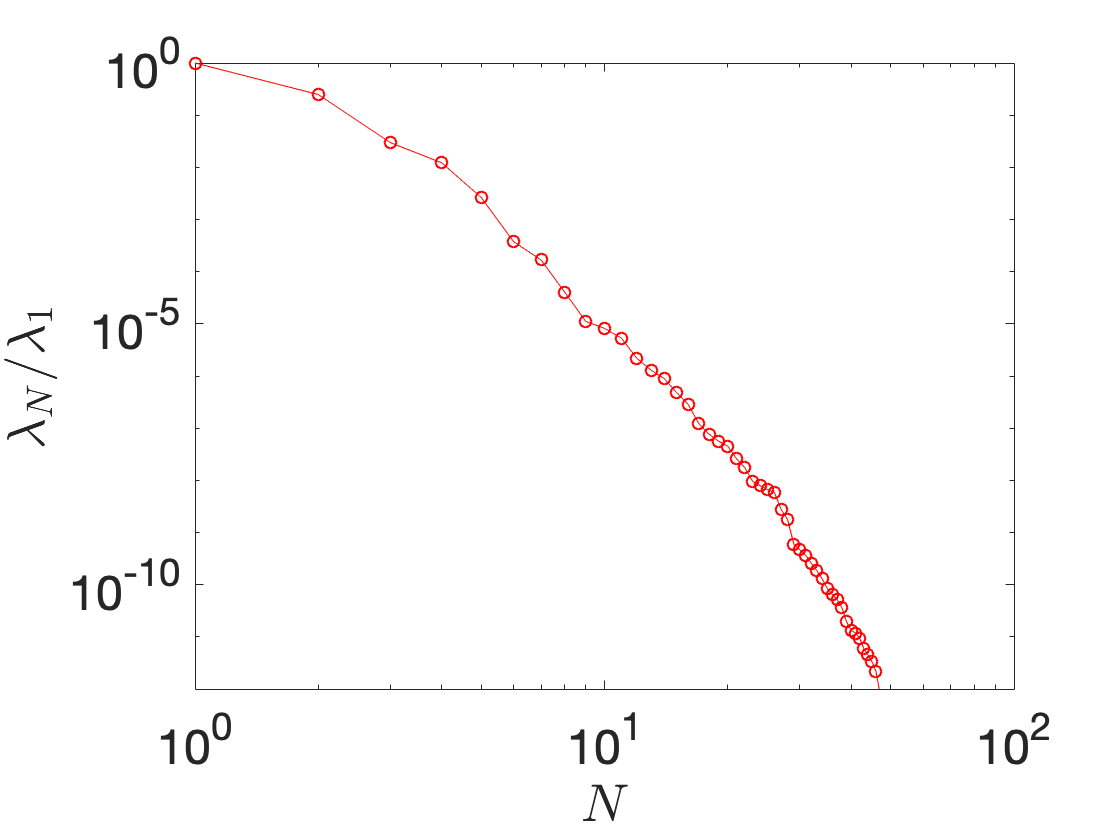}}  
 
\caption{reaction-diffusion problem with boundary layer.  Behavior  of POD eigenvalues. 
(a): 
$\|  \cdot \|_2$-POD for 
$\{  \mathbf{a}_{\mu^i}^{\rm hf} \}_{i=1}^{n_{\rm train}}$.
(b)
$L^2$-POD for $\{  \mathbf{K}_{\mu^i}^{\star} \}_{i=1}^{n_{\rm train}}$.
}
 \label{fig:BL_eigen}
\end{figure} 
 
\subsubsection*{Results}
Figure \ref{fig:BL_vis}(a) shows the behavior of the solution field $z_{\mu}$ for $\mu=20$; as anticipated above, the solution exhibits a boundary layer  close to $\Gamma_{\rm D}$. 
In Figure \ref{fig:BL_vis}(b) and (c), we show slices of the solution field 
$t \mapsto z_{\mu}(t,t)$, $t \mapsto z_{\mu}(t,0.5)$,
for three parameter values along the straight lines depicted in  
Figure \ref{fig:BL_vis}(a): as expected, the size of the boundary layer strongly depends on the value of $\mu$.
In Figure \ref{fig:BL_vis_mapped}, we reproduce the results of Figure 
\ref{fig:BL_vis} for the mapped field: we observe that the mapping procedure significantly reduces the sensitivity of the solution to the value of $\mu$.

\begin{figure}[h!]
\centering
 \subfloat[$z_{\mu}$, $\mu=20$] {\includegraphics[width=0.3\textwidth]
 {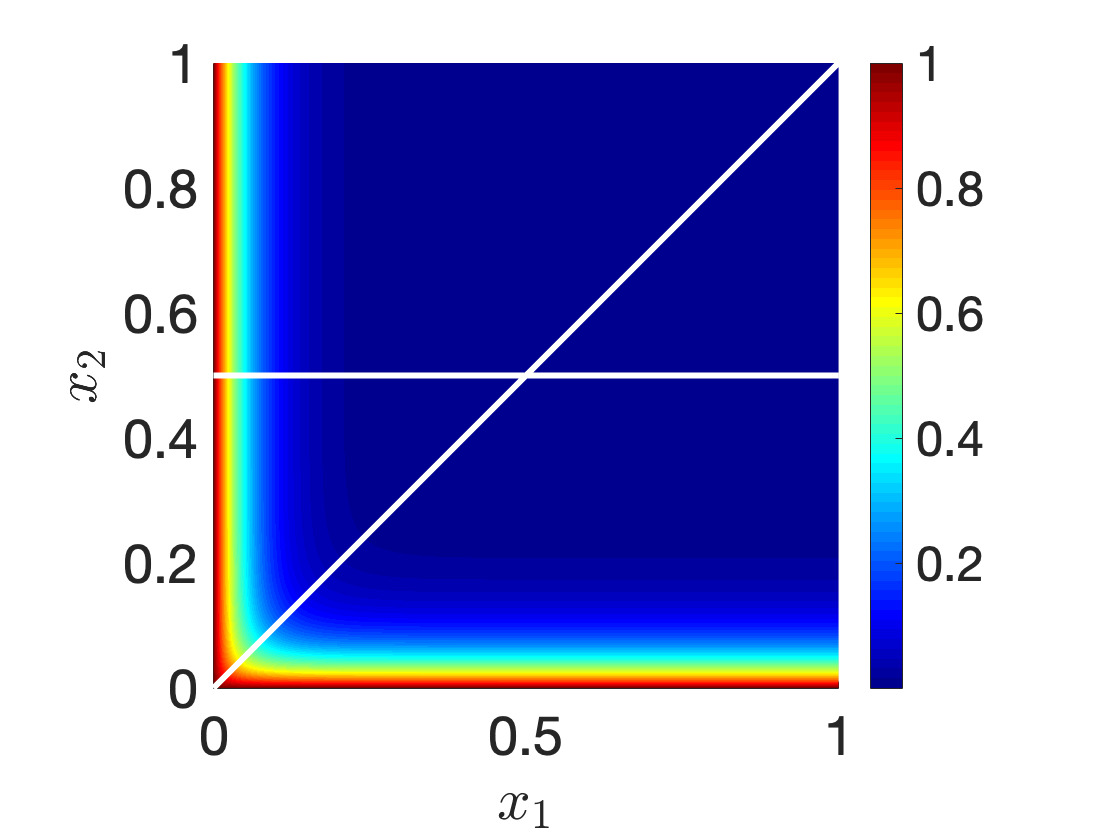}}  
    ~ ~
\subfloat[] {\includegraphics[width=0.3\textwidth]
 {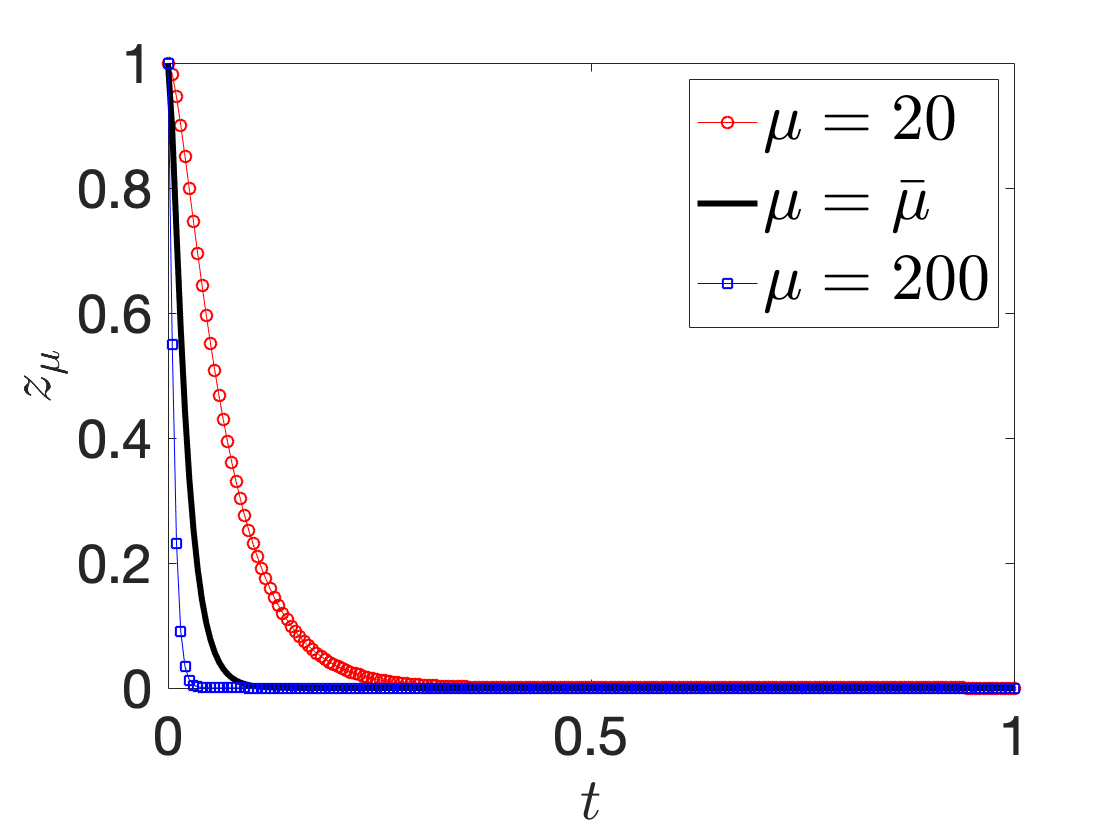}}  
     ~ ~
\subfloat[] {\includegraphics[width=0.3\textwidth]
 {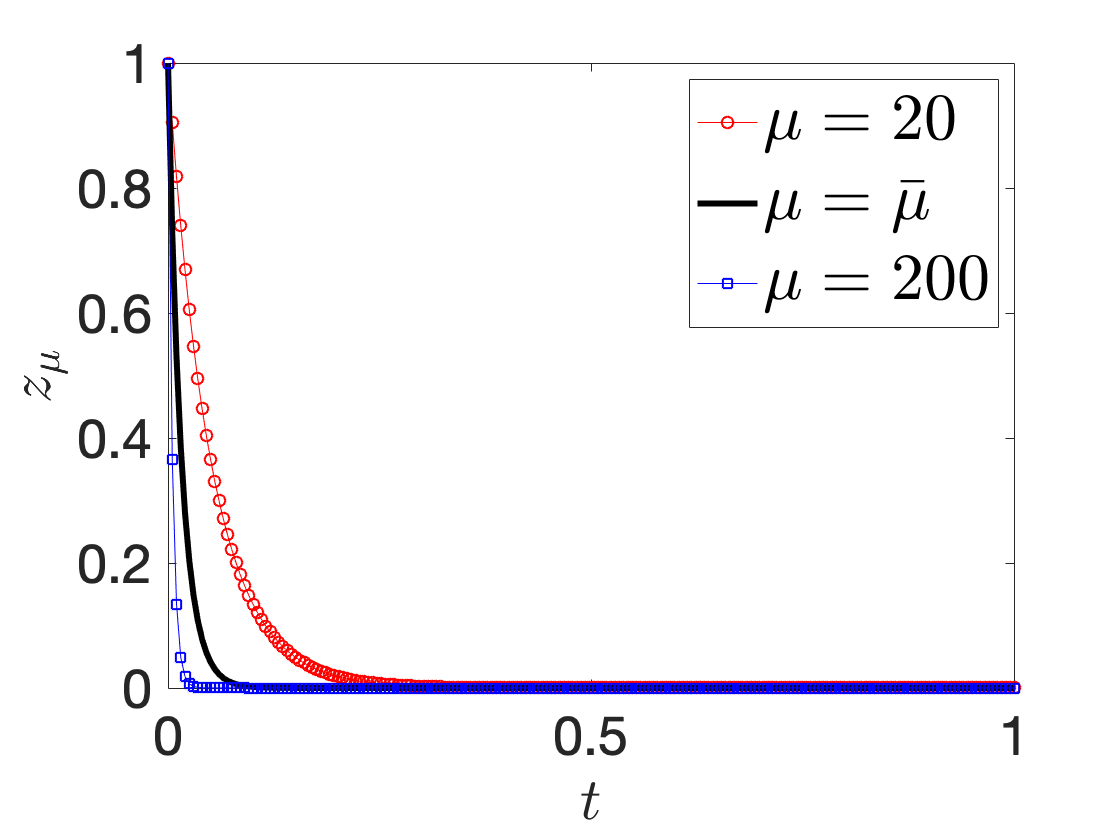}}  
 
\caption{reaction-diffusion problem with boundary layer.
 (a): behavior of $z_{\mu=20}$.
(b) and (c): behavior of  $z_{\mu}(\mathbf{x} = (t,t))$ 
and $z_{\mu}(\mathbf{x} = (t,0.5))$ 
for three values of $\mu \in \mathcal{P}$, $t \in [0,1]$.
}
 \label{fig:BL_vis}
\end{figure} 

\begin{figure}[h!]
\centering
 \subfloat[$\tilde{z}_{\mu}$, $\mu=20$] {\includegraphics[width=0.3\textwidth]
 {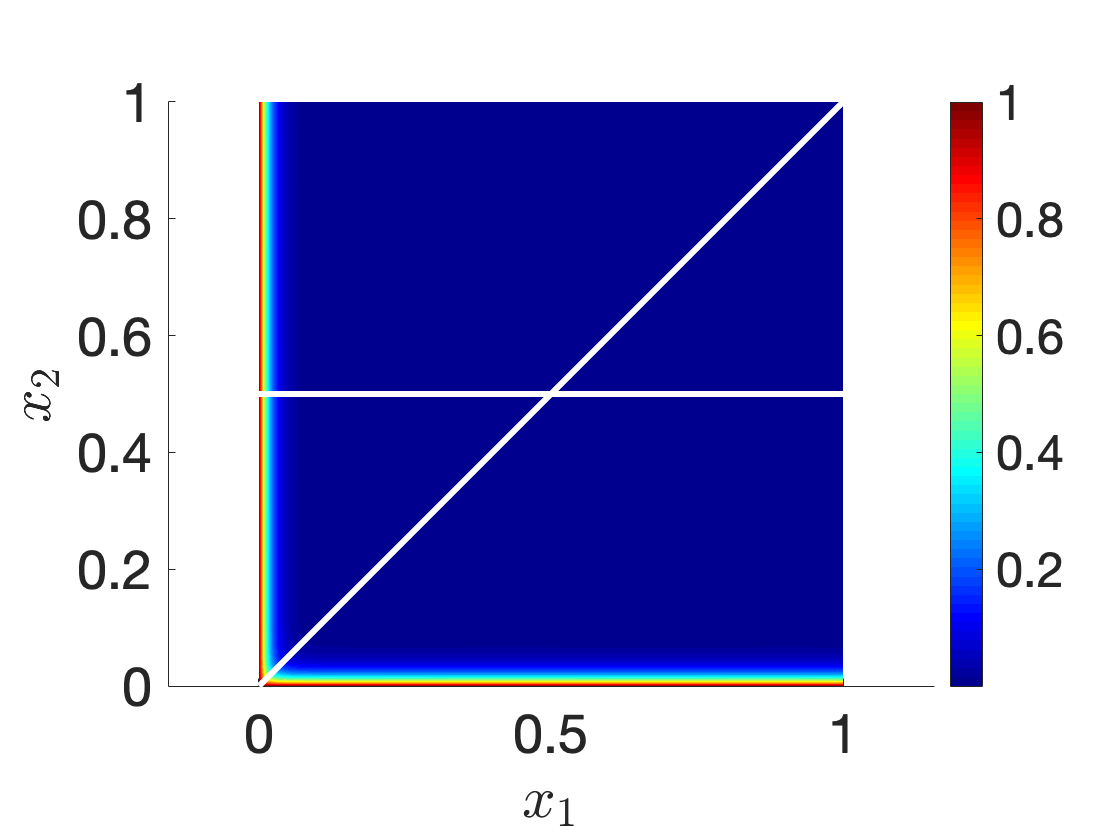}}  
    ~ ~
\subfloat[] {\includegraphics[width=0.3\textwidth]
 {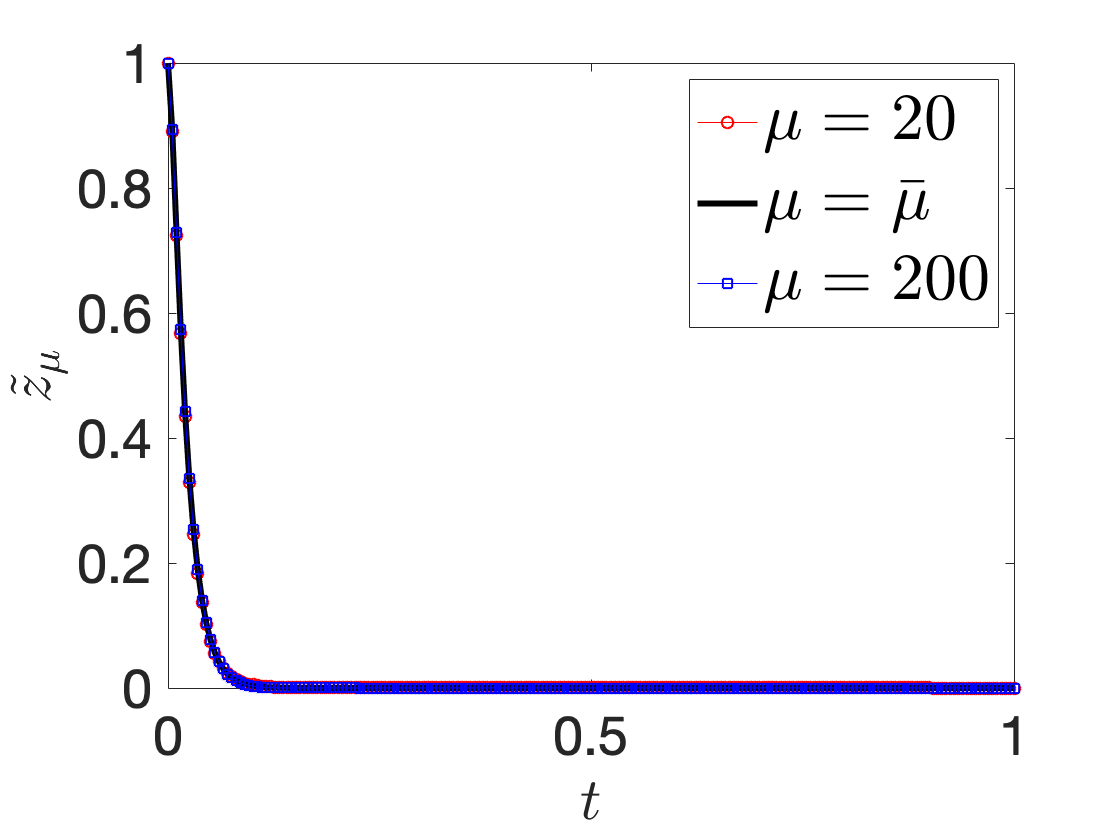}}  
     ~ ~
\subfloat[] {\includegraphics[width=0.3\textwidth]
 {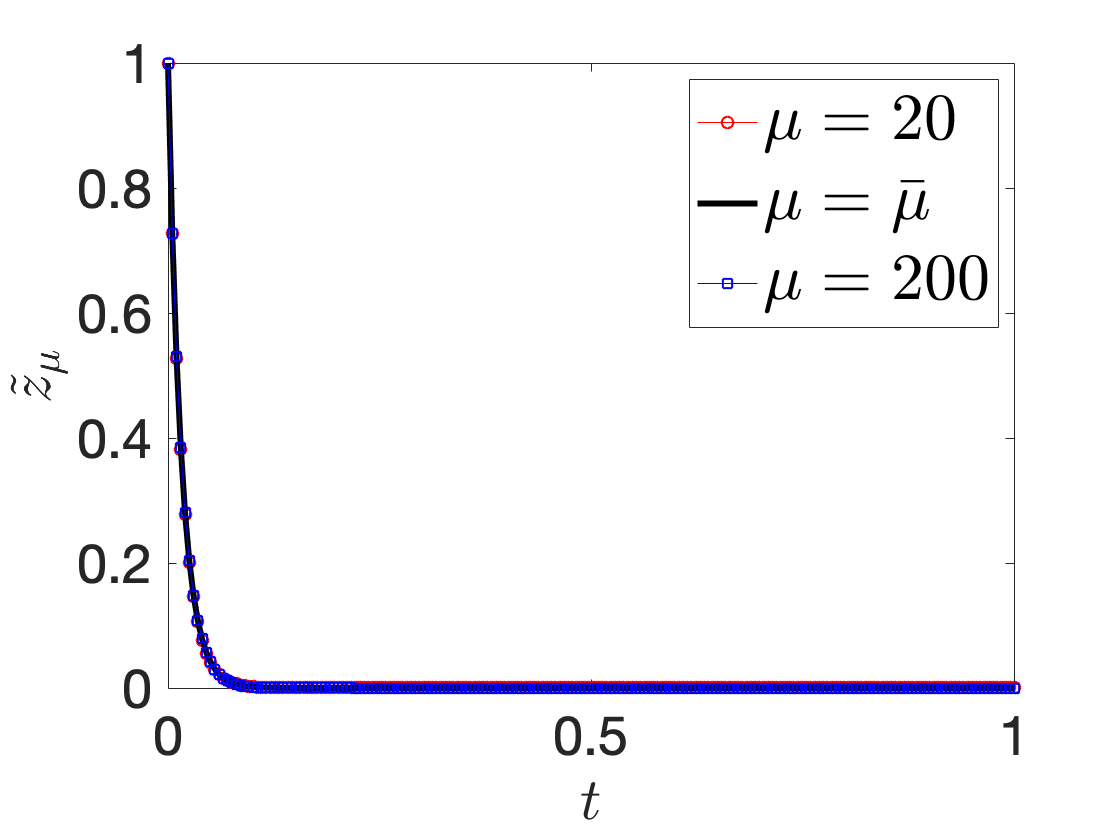}}  
 
\caption{reaction-diffusion problem with boundary layer. 
(a): behavior of $\tilde{z}_{\mu=20}$.
(b) and (c): behavior of $\tilde{z}_{\mu}(\mathbf{X} = (t,t))$ 
and $\tilde{z}_{\mu}(\mathbf{X} = (t,0.5))$
for three values of $\mu \in \mathcal{P}$, $t \in [0,1]$.
}
 \label{fig:BL_vis_mapped}
\end{figure} 

Figure \ref{fig:BL_eigen_2}(a) shows the behavior of the relative $H^1$ projection error associated with the POD space
in the physical (unregistered) and mapped (registered) configurations,
\begin{equation}
\label{eq:Eproj_BL}
E_{\rm proj}=
\max_{\mu \in \{ \mu^j \}_{j=1}^{n_{\rm test}}    }
\frac{\| z_{\mu} - \Pi_{\mathcal{Z}_N} z_{\mu}  \|_{H^1(\Omega)}}{
\| z_{\mu}   \|_{H^1(\Omega)}
},
\quad
\widetilde{E}_{\rm proj}=
\max_{\mu \in \{ \mu^i \}_{i=1}^{n_{\rm test}}    }
\frac{\| \tilde{z}_{\mu} - \Pi_{\widetilde{\mathcal{Z}_N}} \tilde{z}_{\mu}  \|_{H^1(\Omega)}}{
\| \tilde{z}_{\mu}   \|_{H^1(\Omega)}
},
\end{equation}
where $\mu^1,\ldots,\mu^{n_{\rm test}} \overset{\rm iid}{\sim} {\rm Uniform}(\mathcal{P})$ 
($n_{\rm test}=200$) and  the space $\{ \mathcal{Z}_N \}_N$ (resp., 
$\{ \widetilde{\mathcal{Z}}_N \}_N$)  is built applying  $H^1$-POD  to  the snapshot sets 
$\{ z_{\mu^i} \}_{i=1}^{n_{\rm test}}$
(resp. 
$\{ \tilde{z}_{\mu^i} \}_{i=1}^{n_{\rm test}}$).
On the other hand, Figure \ref{fig:BL_eigen}(b)
shows   the $H^1$-POD eigenvalues associated with $\{ z_{\mu^i} \}_{i=1}^{n_{\rm test}}$ (unregistered) and $\{ \tilde{z}_{\mu^i} \}_{i=1}^{n_{\rm test}}$
(registered).
Note that  $\widetilde{E}_{\rm proj} \leq  {E}_{\rm proj}$  
for $N \leq  6$, while $\widetilde{E}_{\rm proj} >  {E}_{\rm proj}$ for  $N>6$; similar behavior can be observed for the POD eigenvalues:
{to reduce the sensitivity of the solution to the value of $\mu$, the mapping process 
introduces some small-amplitude smaller spatial scale
distortions, which then ultimately control the convergence.
} Nevertheless, we do emphasize that for $N=1$  the relative  error satisfies 
$\widetilde{E}_{\rm proj} \lesssim   10^{-4}$ and  is thus comparable with the accuracy of the underlying truth discretization: as a result, the ``complexity overhead" due to mapping-induced distortions has  little practical importance for this model problem.

\begin{figure}[h!]
\centering
 \subfloat[] {\includegraphics[width=0.4\textwidth]
 {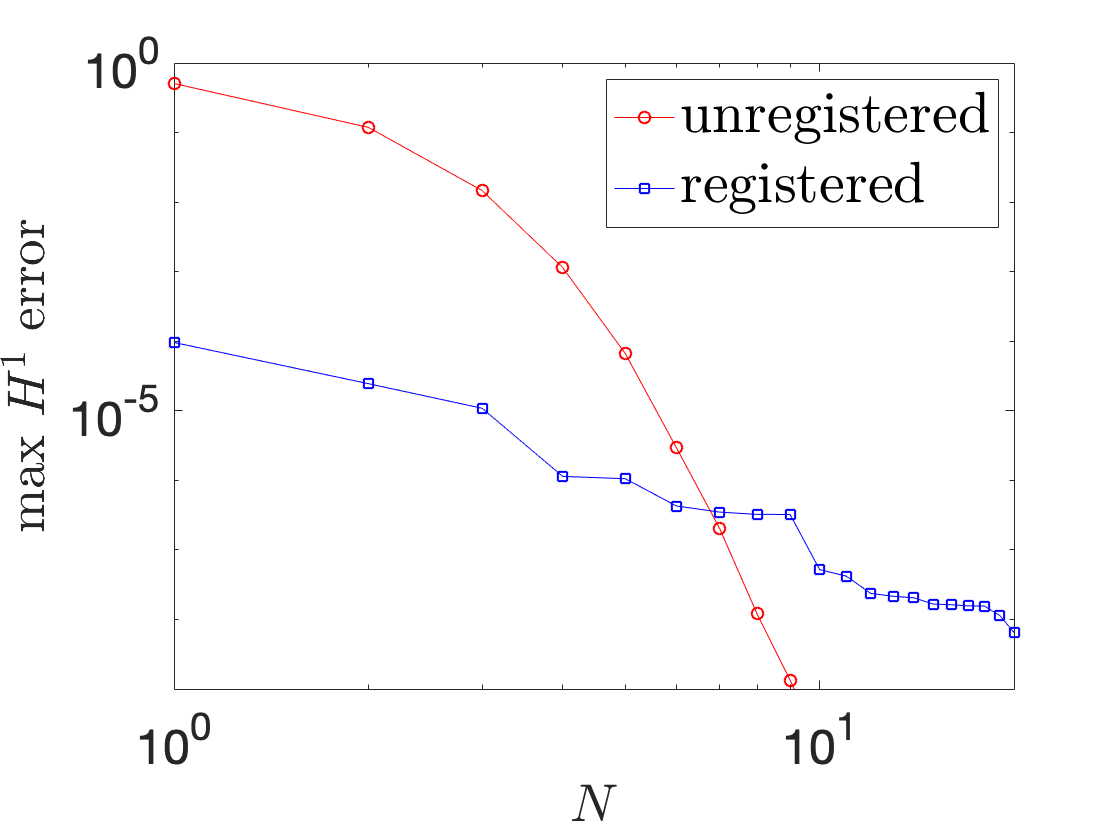}}  
    ~ ~
\subfloat[] {\includegraphics[width=0.4\textwidth]
 {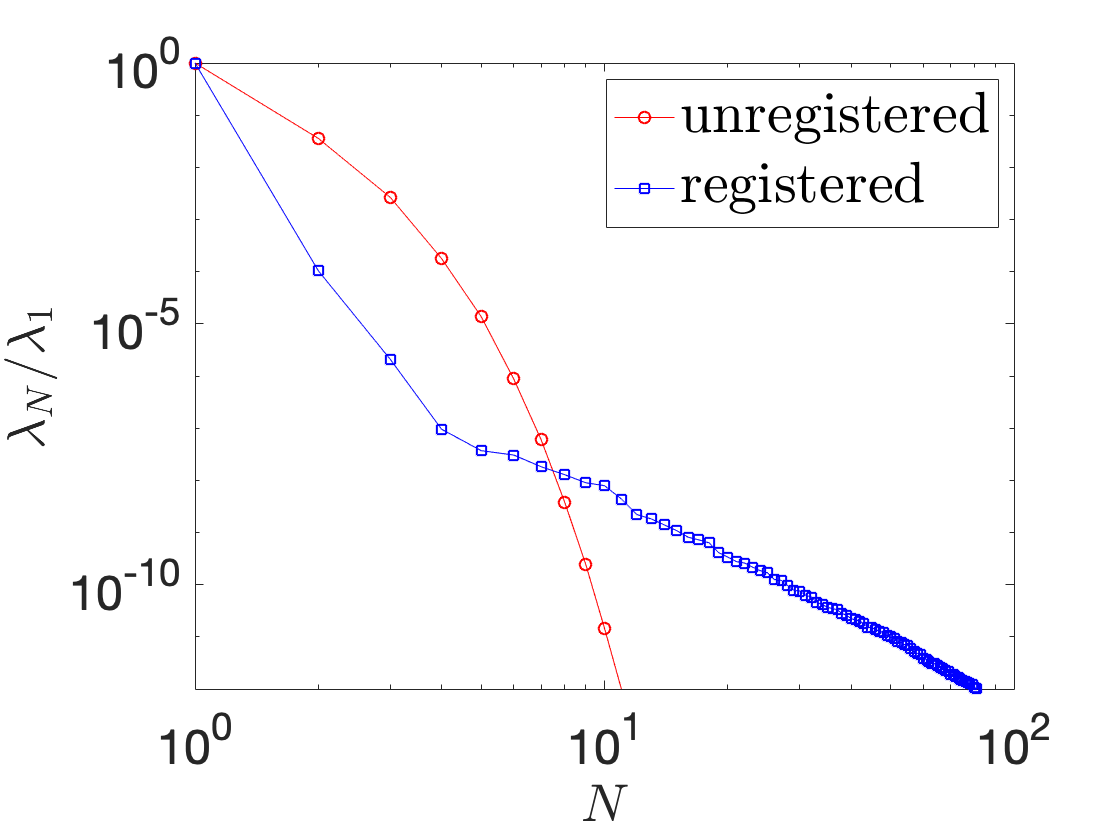}}  
 
\caption{reaction-diffusion problem with boundary layer. 
(a): behavior of $E_{\rm proj}, \widetilde{E}_{\rm proj}$
\eqref{eq:Eproj_BL} with $N$.
(b): behavior of the $H^1$-POD eigenvalues associated with 
$\{ z_{\mu^i} \}_{i=1}^{n_{\rm test}}$ (unregistered) and
$\{ \tilde{z}_{\mu^i} \}_{i=1}^{n_{\rm test}}$ (registered).
}
 \label{fig:BL_eigen_2}
\end{figure} 

{In Figure \ref{fig:BL_xi}, we investigate the impact on performance of the choice of $\xi$ in  
\eqref{eq:optimization_statement}. More in detail, we show the behavior of the proximity measure 
$\mathfrak{f}( \mathbf{a}_{\rm hf}(\mu), \mu, \bar{\mu})$ and of the mapping seminorm
$|  \boldsymbol{\Psi}_{  \mathbf{a}_{\rm hf}(\mu)    }^{\rm hf} |_{H^2(\Omega)}$ with respect to $\xi$ for three values of $\mu$ in $\mathcal{P}$, where 
$\mathbf{a}_{\rm hf}(\mu)$ denotes the solution to 
\eqref{eq:optimization_statement} for a given $\mu \in\mathcal{P}$. Interestingly, the performance is nearly independent of $\xi$ for $\xi \lesssim 10^{-8}$, for all  three choices of $\mu$.
}

\begin{figure}[h!]
\centering
 \subfloat[] {\includegraphics[width=0.4\textwidth]
 {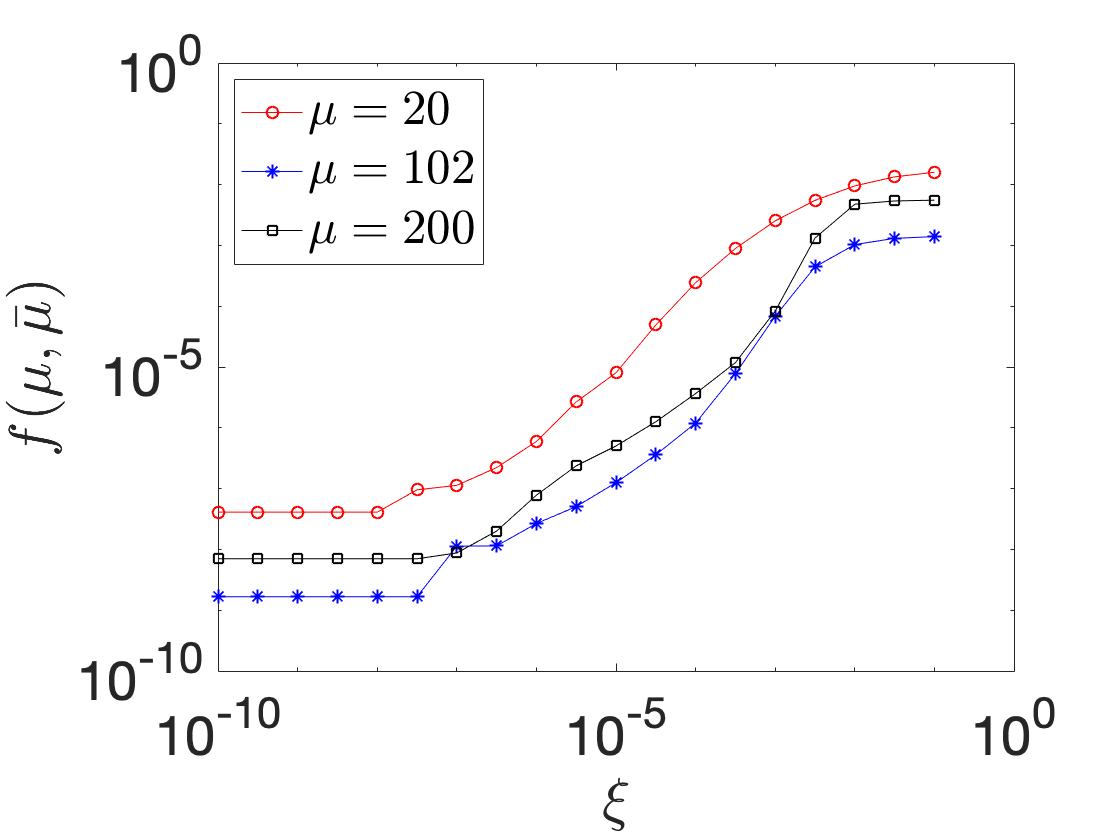}}  
    ~ ~
\subfloat[] {\includegraphics[width=0.4\textwidth]
 {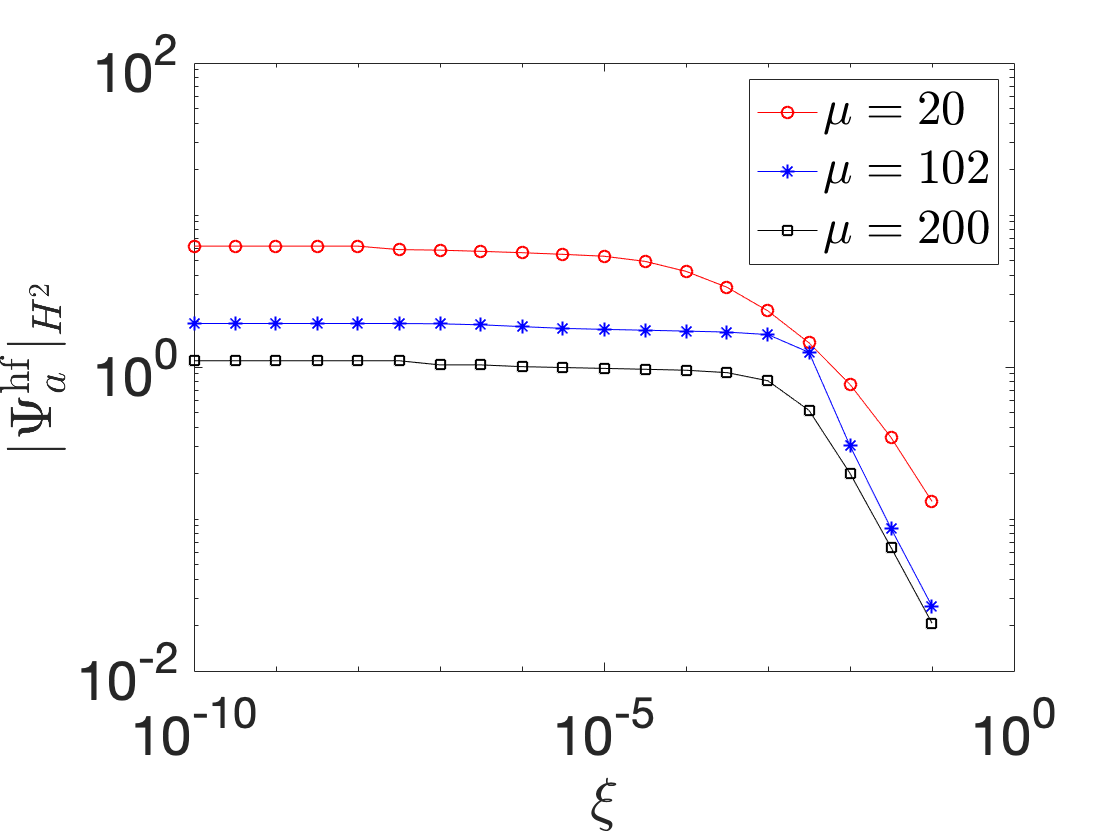}}  
 
\caption{reaction-diffusion problem with boundary layer. Sensitivity analysis with respect to $\xi$ in
\eqref{eq:optimization_statement} for three values of $\mu \in \mathcal{P}$.
(a): behavior of the optimal proximity measure 
$\mathfrak{f}( \mathbf{a}_{\rm hf}(\mu), \mu, \bar{\mu})$ with $\xi$.
(b): behavior of the mapping $H^2$ seminorm with $\xi$.
}
 \label{fig:BL_xi}
\end{figure} 

{In Figure \ref{fig:BL_barmu},  we show  the behavior of the proximity measure 
$\mathfrak{f}( \mathbf{a}_{\rm hf}(\mu), \mu, \bar{\mu})$ and of the mapping seminorm
$|  \boldsymbol{\Psi}_{  \mathbf{a}_{\rm hf}(\mu)    }^{\rm hf} |_{H^2(\Omega)}$ with respect to $\mu$,  for two choices of the reference parameter, $\bar{\mu}= \sqrt{\mu_{\rm min} \mu_{\rm max}}$ and $\bar{\mu}= \frac{\mu_{\rm min} + \mu_{\rm max}}{2}$. 
For the second choice, we chose $\epsilon=0.05$.  We observe that the first choice leads to superior performance in terms of accuracy  and also reduces the maximum (over $\mathcal{P}$) magnitude of the mapping seminorm: similarly to 
\cite{mowlavi2018model}, we empirically observe that the choice of the reference field (\emph{template})  $\bar{u}$ has a significant impact on performance.
}

\begin{figure}[h!]
\centering
 \subfloat[] {\includegraphics[width=0.4\textwidth]
 {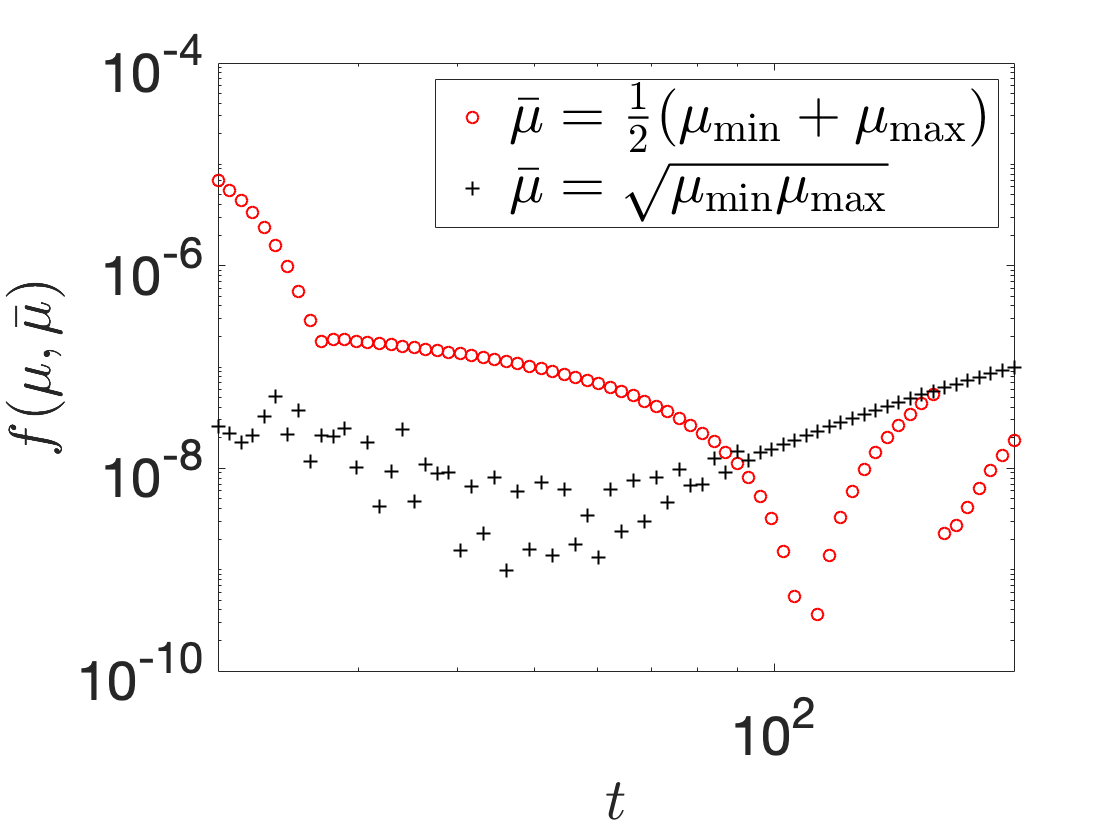}}  
    ~ ~
\subfloat[] {\includegraphics[width=0.4\textwidth]
 {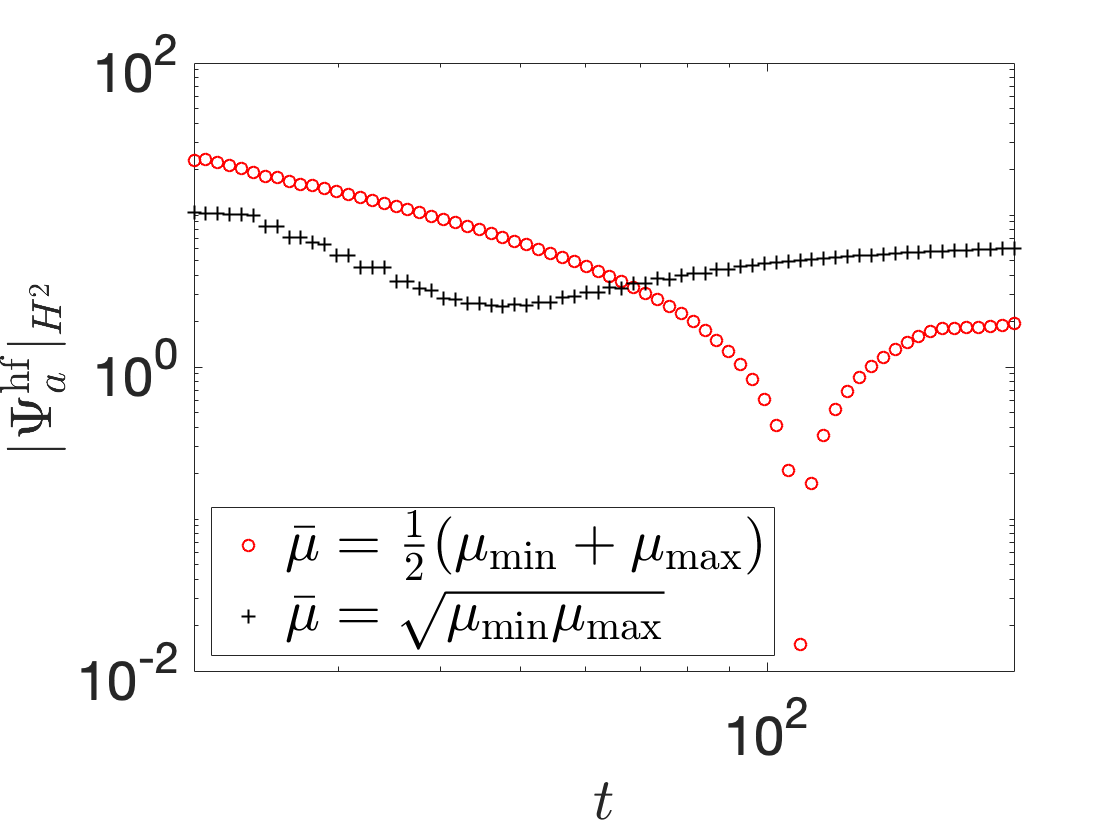}}  
 
\caption{reaction-diffusion problem with boundary layer.
 Sensitivity analysis with respect to the reference field $\bar{u}=u_{\bar{\mu}}$.
(a): behavior of the optimal proximity measure 
$\mathfrak{f}( \mathbf{a}_{\rm hf}(\mu), \mu, \bar{\mu})$ with $\mu$, for two choices of $\bar{\mu}$.
(b): behavior of the mapping $H^2$ seminorm with $\mu$,
for two choices of $\bar{\mu}$.
}
 \label{fig:BL_barmu}
\end{figure}

 \subsubsection{Approximation of advection-dominated problems}
\label{sec:hyperbolic}

\subsubsection*{Problem statement}
We consider the advection-reaction  problem:
\begin{subequations}
\label{eq:model_problem_strong}
\begin{equation}
\left\{
\begin{array}{ll}
 \nabla \cdot ( \mathbf{c}_{\mu}  z_{\mu} ) \, + \,
\sigma_{\mu} z_{\mu}   = f_{\mu}  & {\rm in} \, \Omega = (0,1)^2 \\
z_{\mu} = z_{\rm D,\mu} & {\rm on} \, \Gamma_{\rm in,\mu}:= \{ \mathbf{x} \in \partial \Omega: \, \mathbf{c}_{\mu} \cdot \mathbf{n} < 0  \}
\\
\end{array}
\right.
\end{equation}
where $\mathbf{n}$ denotes the outward normal to $\partial \Omega$, and
\begin{equation}
\begin{array}{l}
\displaystyle{
\mathbf{c}_{\mu} = \left[
\begin{array}{c}
\cos(\mu_1) \\
\sin(\mu_1) \\
\end{array}
\right],
\quad
\sigma_{\mu} = 1 + \mu_2 e^{x_1 + x_2},
\quad
f_{\mu}= 1 + x_1 x_2,
}\\[3mm]
\displaystyle{
z_{\rm D, \mu} = 4 \arctan \left( \mu_3 \left( x_2 - \frac{1}{2} \right) \right) \,
\left(   x_2 -  x_2^2 \right)
}
\\[3mm]
\displaystyle{
\mu
=[\mu_1,\mu_2,\mu_3] \in \mathcal{P}:=
 \left[ -\frac{\pi}{10},  \frac{\pi}{10} \right] \times
 \left[ 0.3,  0.7 \right] \times
  \left[60,100\right].
 }
\end{array}
\end{equation}
\end{subequations}
Problem \eqref{eq:model_problem_strong} is a parametric hyperbolic advection-reaction problem; since $\sigma_{\mu} + \frac{1}{2} \nabla \cdot \mathbf{c}_{\mu} \geq 1 > 0$ for all $\mathbf{x} \in \Omega$, $\mu \in \mathcal{P}$, there exists a unique solution $z_{\mu}$ to \eqref{eq:model_problem_strong} for all $\mu \in \mathcal{P}$.
 We refer to \cite{brunken2019parametrized,dahmen2012adaptive} for a mathematical analysis of the problem, and for the derivation of the infinite-dimensional variational form.
We here resort to a  DG  P3  discretization on a structured triangular mesh 
with $N_{\rm hf}=23120$ degrees of freedom
to approximate \eqref{eq:model_problem_strong}; we denote by $\mathcal{X}$ the broken DG space equipped with the inner product  
$(w,v):= \sum_{k=1}^{n_{\rm el}} \, 
\int_{\texttt{D}^k} \, w \, v \, dx,$  and the induced norm
$\| \cdot \| = \sqrt{(\cdot, \cdot)}$.
Note that changes in $\mu_1$ lead to changes in the large-gradient region associated with the solution $z_{\mu}$ (cf. Figures \eqref{fig:visu_hyp}(a) and (b)), and are thus critical for linear approximation methods.

Exploiting the reasoning in section \ref{sec:offline_online_pMOR}, we can derive the variational formulation for the mapped field $\tilde{z}_{\mu}$. We provide the explicit expressions of the terms in \eqref{eq:ADR_FEM_general} 
  in section \ref{sec:DG_hyp_app}.
Here, we remark that 
 the size of the expansions in \eqref{eq:eim_ADR}  are chosen based on the  criterion in \eqref{eq:POD_cardinality_selection}, with  tolerance ${ tol}_{\rm eim}= 5 \cdot 10^{-7}$.

\subsubsection*{Construction of the mapping}

We set $\bar{\mu}$ equal to the centroid of $\mathcal{P}$, and we consider the proximity measure $\mathfrak{f}$ in \eqref{eq:L2obj} with $u_{\mu}=z_{\mu}$.
We further consider $\xi=10^{-3}$, and we consider a polynomial expansion with   $\overline{M}=6$ in \eqref{eq:tensorized_polynomials} ($M_{\rm hf}=72$).
To build the regressor $\widehat{\mathbf{a}}:\mathcal{P} \to \mathbb{R}^M$, we consider $n_{\rm train}=250$ uniformly-sampled parameters in $\mathcal{P}$, 
 and we set $tol_{\rm pod}=10^{-4}$ in \eqref{eq:POD_cardinality_selection}:
 for this choice of the parameters, the registration procedure  returns an affine expansion with $M = 12$ terms.

As opposed to the  previous problem, we do not expect that the mapping procedure will lead to a nearly one-dimensional mapped manifold; nevertheless, we do expect that the mapping procedure will reduce the sensitivity of the solution to parametric changes, particularly in $\mu_1$ --- which regulates the position of the high-gradient region in $\Omega$. For this reason, we here consider a larger value of the regularization parameter $\xi$ compared to the previous two examples: larger values of $\xi$  lead to smoother mappings and thus reduce the risk of overfitting and   facilitate hyper-reduction. We further consider a lower value of  $M_{\rm hf}$ ($M_{\rm hf}=72$ as opposed to 
$M_{\rm hf}=128$) compared to the previous example: in our numerical experience, 
the algorithm is insensitive to the choice of $M_{\rm hf}$, for $M_{\rm hf} \geq 72$.

\subsubsection*{Results}

Figure \ref{fig:visu_hyp} shows the solution to \eqref{eq:model_problem_strong} for three values of the parameters $\mu^1= [-\pi/10,0.3,60]$,
$\mu^2= [\pi/10,0.7,100]$ and the centroid 
$\bar{\mu}=[0,0.55,80]$; on the other hand, Figure \ref{fig:vistildeu_hyp} shows the solution to the mapped problem $\tilde{z}_{\mu}=z_{\mu} \circ \boldsymbol{\Phi}_{\mu}$ for the same three values of the parameter.
As stated above, changes in $\mu_1$ lead to changes in the large-gradient region which corresponds to the propagation of the inflow boundary condition in the interior of $\Omega$: the mapping $\boldsymbol{\Phi}_{\mu}$ significantly reduces the sensitivity of the solution to changes in the angle $\mu_1$.

\begin{figure}[h!]
\centering
 \subfloat[$z_{\mu^1}$] {\includegraphics[width=0.3\textwidth]
 {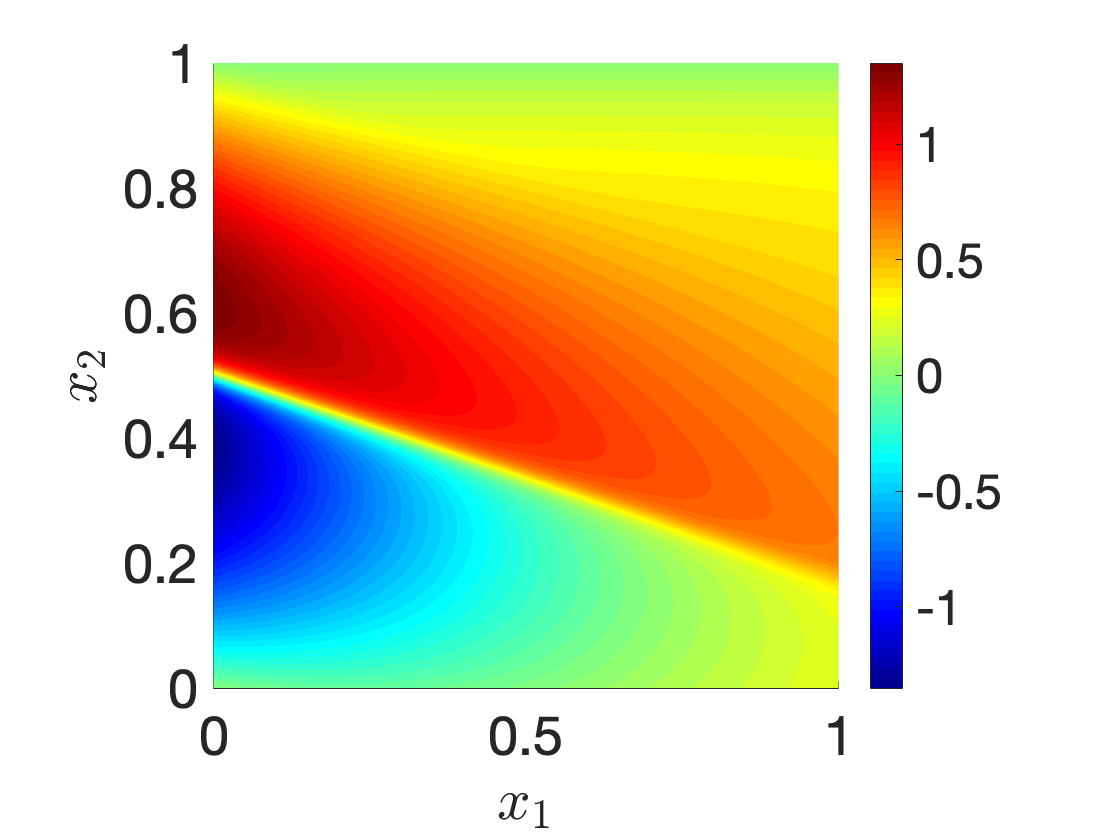}}  
    ~ ~
\subfloat[$z_{\mu^2}$] {\includegraphics[width=0.3\textwidth]
 {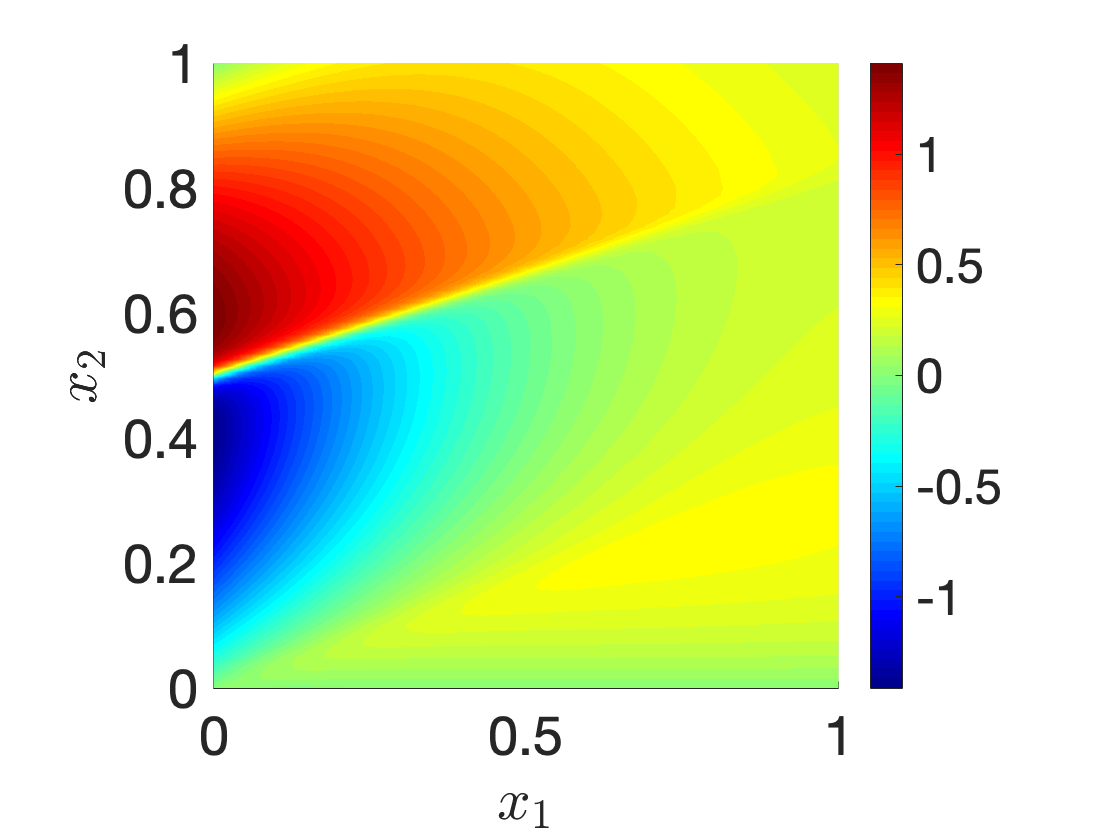}}  
     ~ ~
\subfloat[$z_{\bar{\mu}}$] {\includegraphics[width=0.3\textwidth]
 {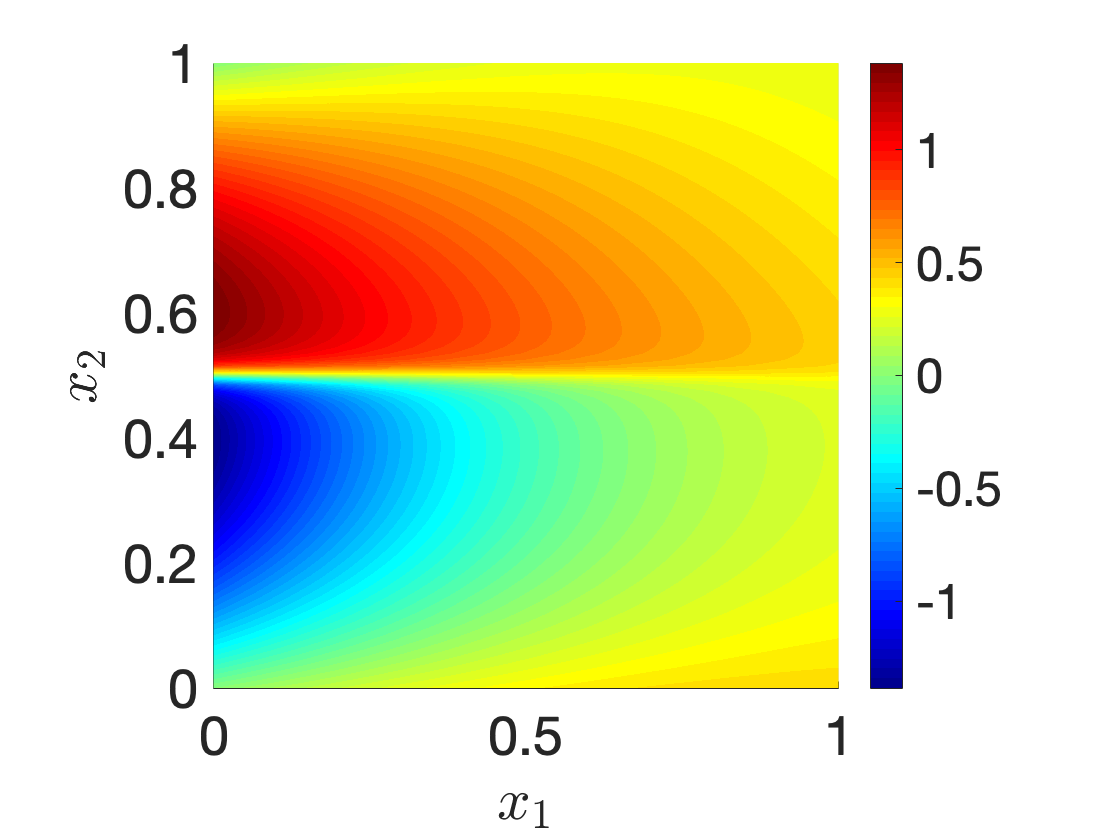}}  
 
\caption{advection-reaction problem. Solution to 
\eqref{eq:model_problem_strong} fo three parameter values.
}
 \label{fig:visu_hyp}
\end{figure} 

\begin{figure}[h!]
\centering
 \subfloat[$\tilde{z}_{\mu^1}$] {\includegraphics[width=0.3\textwidth]
 {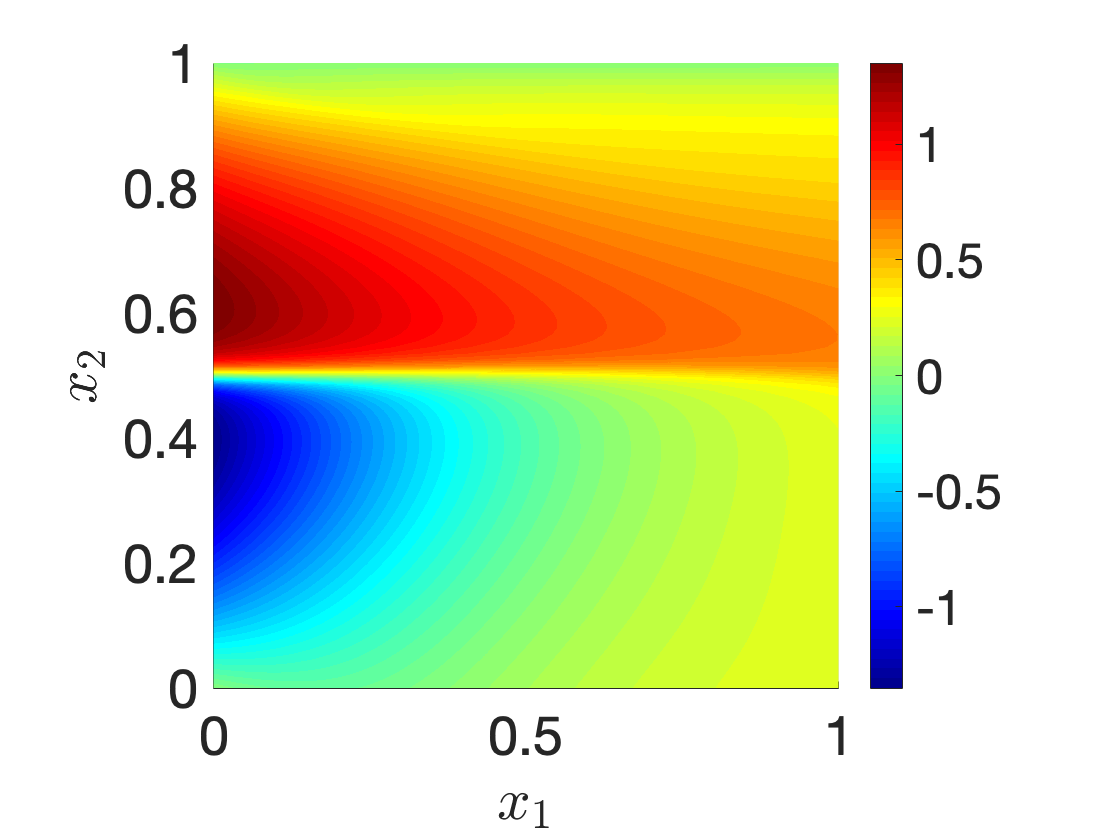}}  
    ~ ~
\subfloat[$\tilde{z}_{\mu^2}$] {\includegraphics[width=0.3\textwidth]
 {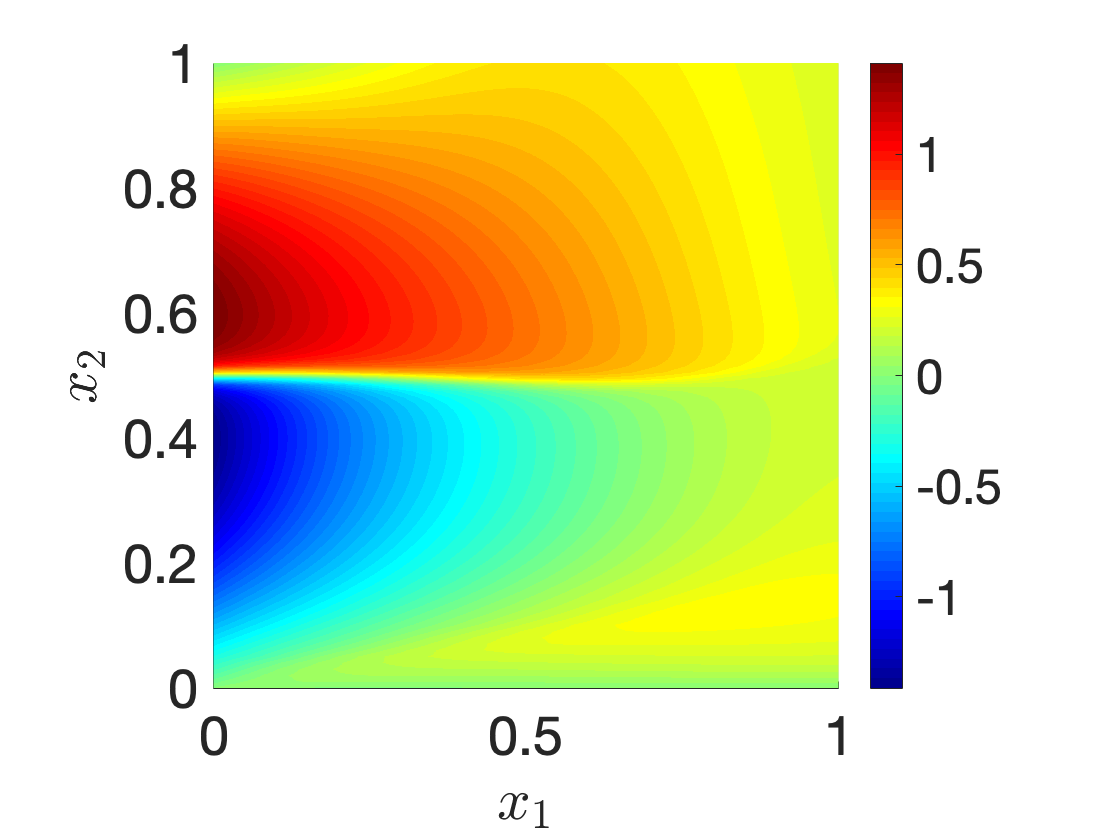}}  
     ~ ~
\subfloat[$z_{\bar{\mu}}$] {\includegraphics[width=0.3\textwidth]
 {imm/hyp/baru.png}}  
 
\caption{advection-reaction problem. Solution to 
\eqref{eq:model_problem_strong} fo three parameter values in the mapped configuration.
}
 \label{fig:vistildeu_hyp}
\end{figure} 

Figure \ref{fig:error_hyp}(a) shows the behavior of the average relative $L^2$ error 
$$
E_{\rm avg} := \frac{1}{n_{\rm test}} \,
\sum_{i=1}^{n_{\rm test}} \; 
\frac{
\| z_{\mu^i}  - \widehat{z}_{\mu^i}   \|_{L^2(\Omega)}
}{ \| z_{\mu^i}    \|_{L^2(\Omega)} },
$$
for the $L^2$-POD-Galerkin ROM associated with the unregistered configuration, and for the $L^2$-POD-Galerkin ROM associated with the  registered configuration. Here, $\mu^1,\ldots$, $\mu^{n_{\rm test}}$ $\overset{\rm iid}{\sim} {\rm Uniform}(\mathcal{P})$, $n_{\rm test}=20$.  
In order to use the same metric for both registered and unregistered ROMs,  for the registered case we compute the error
as
$$
\| z_{\mu}  - \widehat{z}_{\mu} \circ \boldsymbol{\Phi}_{\mu}^{-1}  \|_{L^2(\Omega)}
=
\sqrt{
\int_{\Omega} \, 
\mathfrak{J}_{\mu} \, \left( \widetilde{z}_{\mu} - \widehat{z}_{\mu} \right)^2 \, dX
}.
$$
Figure \ref{fig:error_hyp}(b) shows the $L^2(\Omega)$-POD eigenvalues associated with $\{ z_{\mu^i}  \}_{i=1}^{n_{\rm train}}$ (unregistered) and $\{ \tilde{z}_{\mu^i}  \}_{i=1}^{n_{\rm train}}$  (registered).

\begin{figure}[h!]
\centering
 \subfloat[] {\includegraphics[width=0.4\textwidth]
 {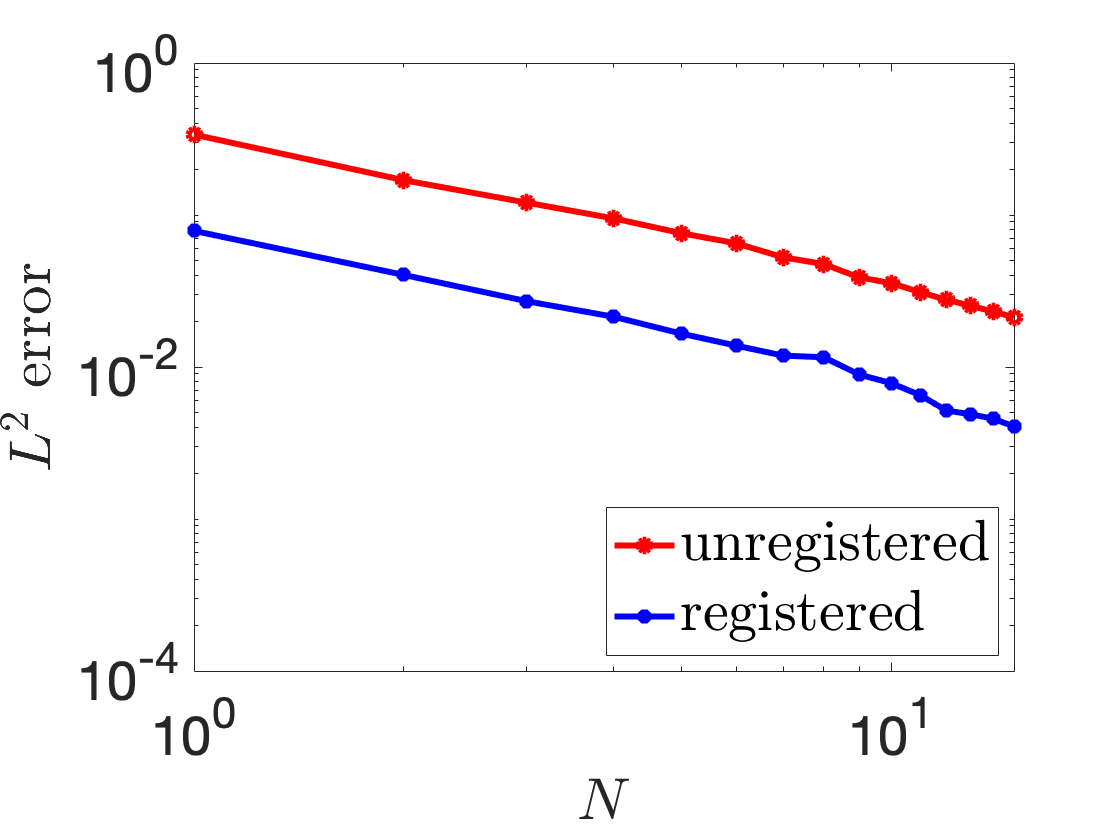}}  
    ~ ~
\subfloat[] {\includegraphics[width=0.4\textwidth]
 {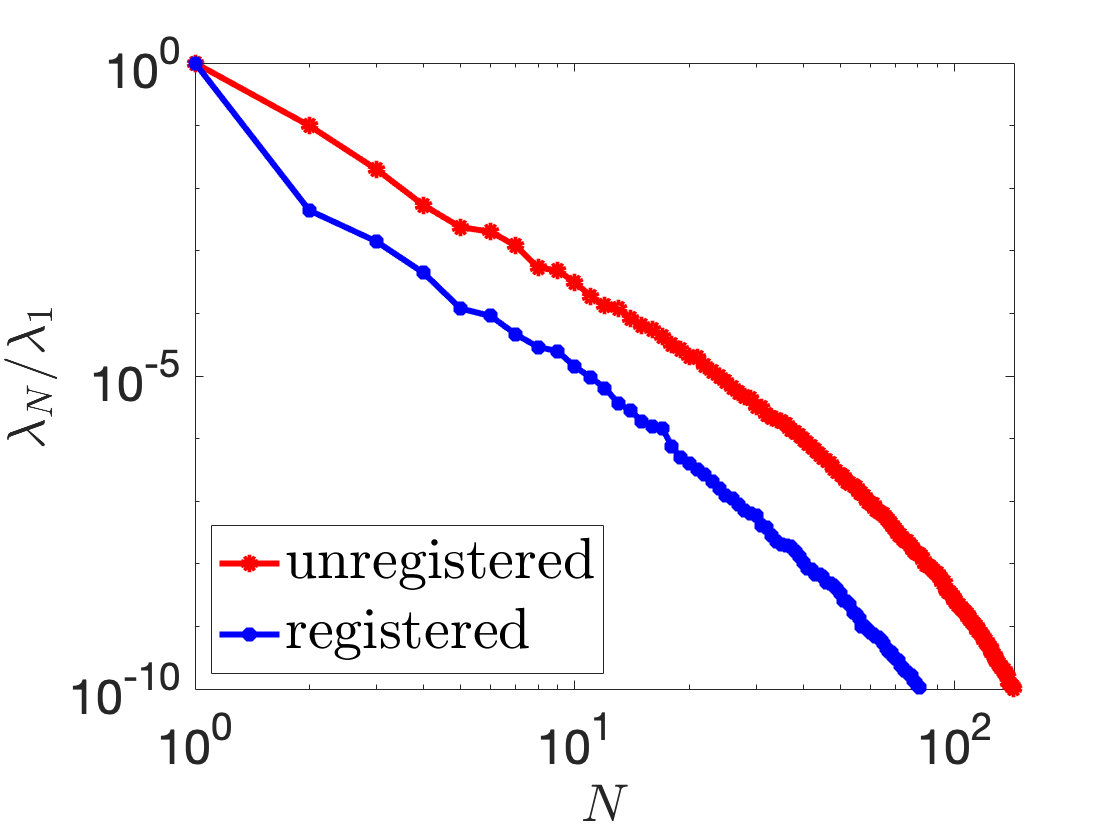}}  
 
\caption{advection-reaction problem.
(a): behavior of $E_{\rm avg}$ with $N$ for the registered and unregistered POD-Galerkin ROMs.
(b): behavior  of $L^2(\Omega)$-POD eigenvalues associated with $\{ z_{\mu^i}  \}_{i=1}^{n_{\rm train}}$ (unregistered) and $\{ \tilde{z}_{\mu^i}  \}_{i=1}^{n_{\rm train}}$  (registered).
}
 \label{fig:error_hyp}
\end{figure} 

We observe that the decay rates of $E_{\rm avg}$ and of the POD eigenvalues are nearly the same for both registered and unregistered configurations; however, the multiplicative constant is significantly different: 
$E_{\rm avg}^{\rm reg.} \approx \frac{1}{4} E_{\rm avg}^{\rm unreg.}$ for all values of $N$ considered, while
$
\left( \frac{\lambda_N}{\lambda_1}  \right)^{\rm reg.} \approx \frac{1}{100} 
\left( \frac{\lambda_N}{\lambda_1}  \right)^{\rm unreg.}
$ for $N \geq 2$. As a result, for any given $N$, the nonlinear mapping procedure leads to a significant improvement.
{This empirical observation suggests a multiplicative effect between $N$ and $M$ approximations and is thus in good agreement with the estimate of the Kolmogorov $N-M$ width  in \eqref{eq:NM_2D}.}
 On the other hand, we remark that  the mapping procedure leads to a significant increase in the number of EIM modes in \eqref{eq:eim_ADR}. More in detail,  for the registered case, we have
$$
Q_{\rm a, el} = 19, \quad
Q_{\rm a, ed} = 26, \quad
Q_{\rm f, el} = 15, \quad
Q_{\rm f, ed} = 6;
$$
for the unregistered case, we have
$$
Q_{\rm a, el} = 4, \quad
Q_{\rm a, ed} = 3, \quad
Q_{\rm f, el} = 1, \quad
Q_{\rm f, ed} = 2.
$$
Therefore, for any given $N$, the registered ROM is more expensive in terms of memory than the unregistered ROM.
{This behavior of EIM can be explained by observing that  most coefficients in \eqref{eq:model_problem_strong} are parametrically affine in the unregistered configuration.}

{In Figure \ref{fig:hyp_xi}, we investigate the effect  of the choice of $\xi$ in  \eqref{eq:optimization_statement}. 
Figures \ref{fig:hyp_xi}(a) and (b) show the 
  behavior of the proximity measure 
$\mathfrak{f}( \mathbf{a}_{\rm hf}(\mu), \mu, \bar{\mu})$ and of the mapping seminorm
$|  \boldsymbol{\Psi}_{  \mathbf{a}_{\rm hf}(\mu)    }^{\rm hf} |_{H^2(\Omega)}$ with respect to $\xi$ for three values of $\mu$ in $\mathcal{P}$, where 
$\mathbf{a}_{\rm hf}(\mu)$ denotes the solution to 
\eqref{eq:optimization_statement} for a given $\mu \in\mathcal{P}$.
Figure \ref{fig:hyp_xi}(c) shows the decay of the mapping
$\| \cdot \|_2$-POD  eigenvalues   associated with 
$\{ \mathbf{a}_{\mu^i}^{\rm hf}   \}_{i=1}^{n_{\rm train}}$ for two choices of $\xi$.
Similarly to the previous test case, 
both addends of the objective function converge to finite values for $\xi \to 0^+$; however, their numerical values and the threshold below which curves exhibit a plateau   differ significantly between the two test cases.
Furthermore, we observe that reducing $\xi$ leads to an increase of the complexity of the mapping manifold and ultimately increases the complexity of the generalization (cf. section \ref{sec:generalization}) step.
}

\begin{figure}[h!]
\centering
 \subfloat[] {\includegraphics[width=0.3\textwidth]
 {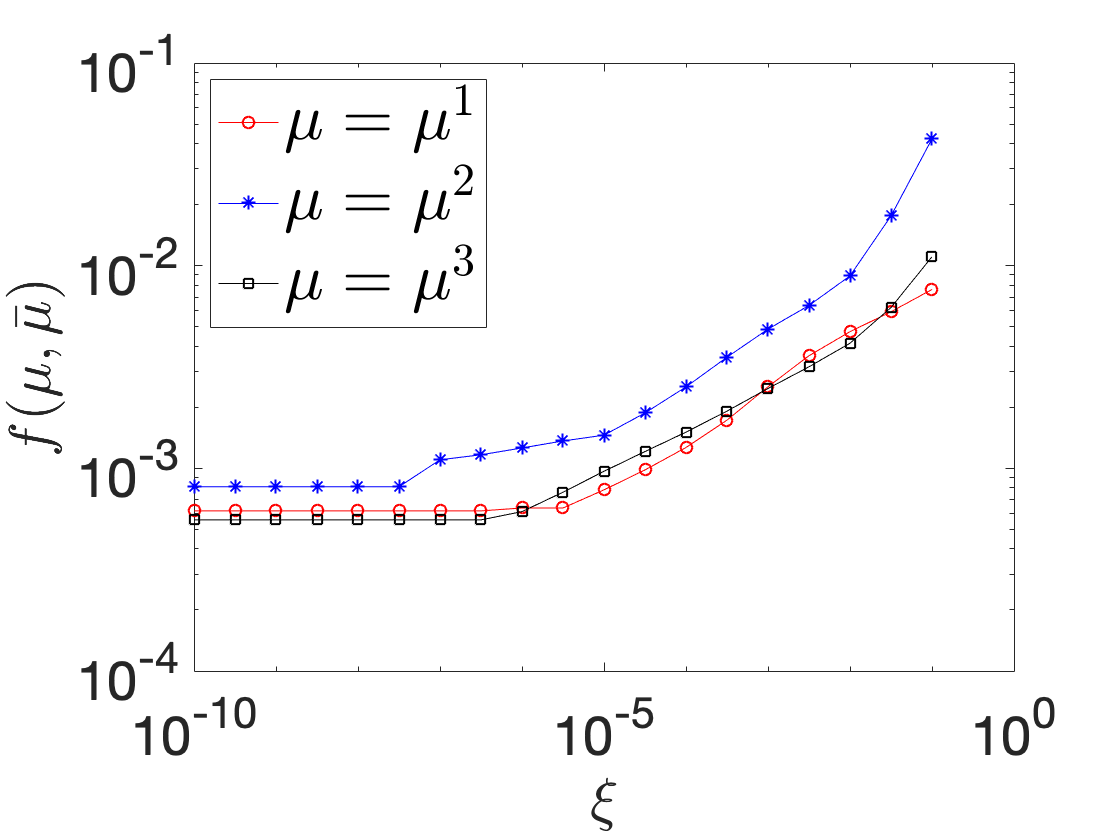}}  
    ~ ~
\subfloat[] {\includegraphics[width=0.3\textwidth]
 {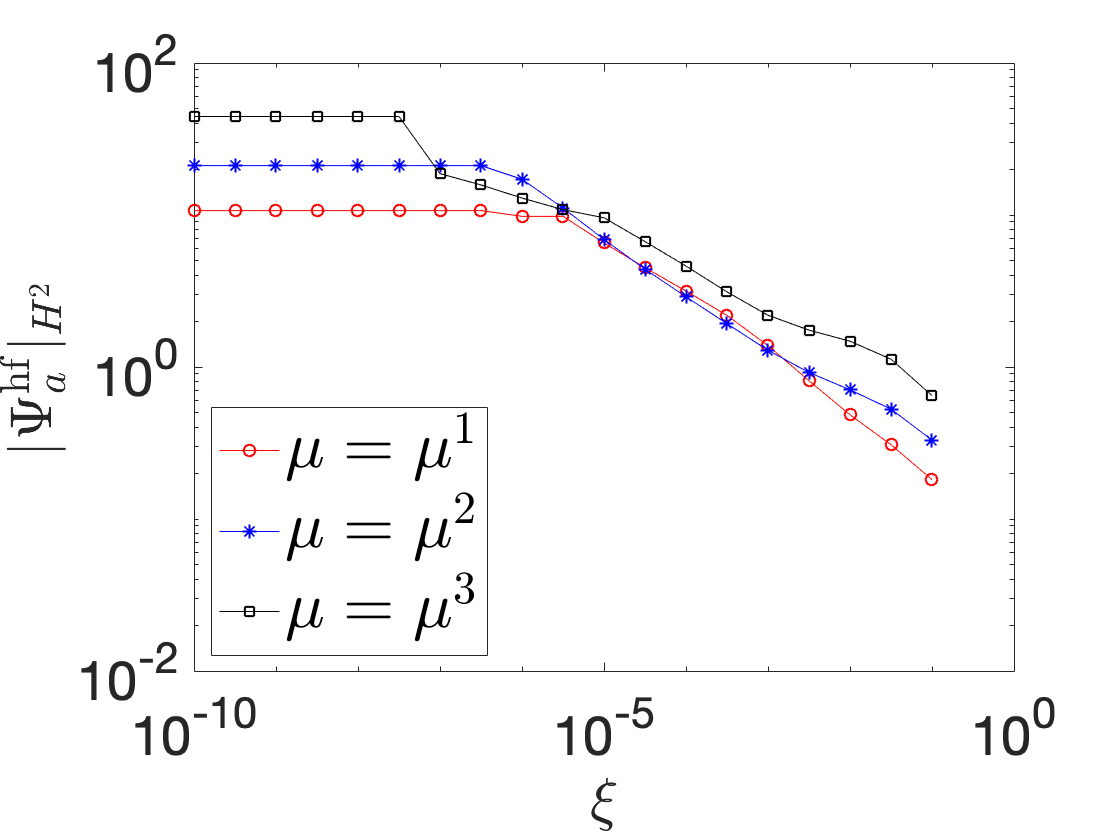}}  
     ~ ~
\subfloat[] {\includegraphics[width=0.3\textwidth]
 {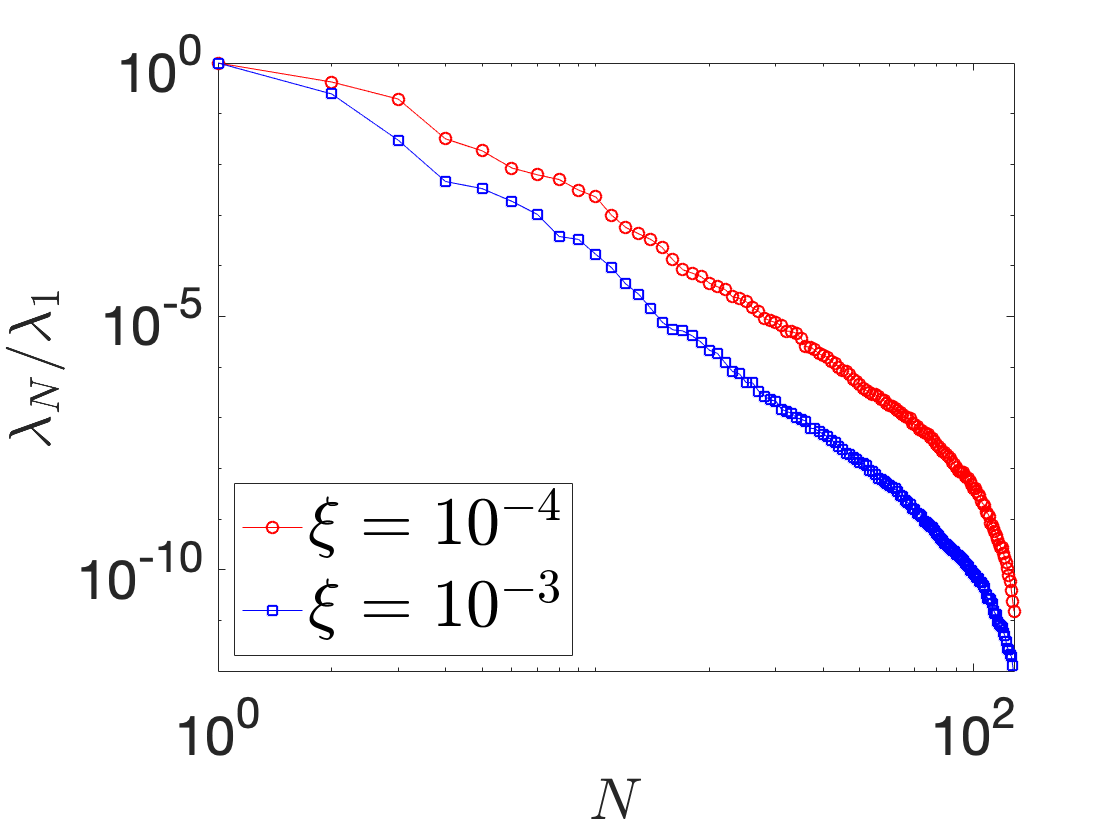}}  
 
\caption{advection-reaction problem.  
Sensitivity analysis with respect to $\xi$ in
\eqref{eq:optimization_statement} for three values of $\mu \in \mathcal{P}$.
(a): behavior of the optimal proximity measure 
$\mathfrak{f}( \mathbf{a}_{\rm hf}(\mu), \mu, \bar{\mu})$ with $\xi$.
(b): behavior of the mapping $H^2$ seminorm with $\xi$.
(c): decay of the $\| \cdot \|_2$-POD eigenvalues associated with 
$\{ \mathbf{a}_{\mu^i}^{\rm hf}   \}_{i=1}^{n_{\rm train}}$ for two choices of $\xi$.
}
 \label{fig:hyp_xi}
\end{figure}

\section{Geometry reduction}
\label{sec:geometry_reduction}
  
 In this section,  we discuss how to adapt the registration procedure introduced in section  \ref{sec:data_compression} to  geometry reduction. 
In this paper, we shall assume that the reference domain $\Omega$ and the parameterized displacement field $\mathbf{d}_{\mu}: \partial \Omega \to \mathbb{R}^d$ are given for all $\mu \in \mathcal{P}$: given $\mathbf{X} \in \partial \Omega$, we can thus compute the corresponding material point in the physical configuration as $\mathbf{x}^{\mu} = \mathbf{X} + \mathbf{d}_{\mu}(\mathbf{X})$. As discussed in the introduction, we remark that the boundary displacement --- or equivalently a parameterization of the boundary $\partial \Omega_{\mu}$ --- might not be available for various classes of problems, including biological systems.

In view of the discussion, we further introduce the rectangle $\Omega_{\rm box} = (a,b) \times (c,d)$ such that  $\Omega, \Omega_{\mu} \subset \Omega_{\rm box}$ for all $\mu \in  \mathcal{P}$. Then, we introduce the 
parameterized function $\boldsymbol{\Psi}^{\rm hf}$ in \eqref{eq:prescribed_mapping_form}, and we define the bases $\{ \boldsymbol{\varphi}_m^{\rm hf} \}_{m=1}^{M_{\rm hf}}$ as in \eqref{eq:tensorized_polynomials} --- after having applied a suitable change of variables.

\subsection{Registration procedure}
\label{sec:registration_geo}

We introduce  (i) the reference points $\{ \mathbf{X}_i  \}_{i=1}^{N_{\rm bnd}}$ $ \subset \partial \Omega$, (ii) the corresponding displaced points 
$\{  \mathbf{x}_i^k = \mathbf{X}_i + \mathbf{d}_{\mu^k}(\mathbf{X}_i) \}_{i,k}$ for some parameters $\mu^1,\ldots,\mu^{n_{\rm train}} \in \mathcal{P}$, and (iii) the parameterized function $\boldsymbol{\Psi}^{\rm hf}$ \eqref{eq:prescribed_mapping_form} - \eqref{eq:tensorized_polynomials}.
Then, for $k=1,\ldots, n_{\rm train}$, we choose $\mathbf{a}_{\rm hf}^k = \mathbf{a}_{\rm hf}(\mu^k)$ as a solution to 
\begin{subequations} 
\label{eq:optimization_statement_geo}
\begin{equation}
\label{eq:optimization_statement_geo_a}
\begin{array}{l}
\displaystyle{
\min_{\mathbf{a} \in \mathbb{R}^{M^{\rm hf}}} \,
\mathfrak{f}^k \left(\mathbf{a}  \right) \,  + \, \xi \big|   
\boldsymbol{\Psi}_{\mathbf{a}}^{\rm hf}  \big|_{H^2(\Omega_{\rm box})}^2, } \\[3mm]
\displaystyle{
{\rm s.t.} \;
\int_{\Omega_{\rm box}} \, 
{\rm exp} \left( \frac{\epsilon  - \mathfrak{J}_{\mathbf{a}}^{\rm hf}(\mathbf{X})}{C_{\rm exp}} \right) \,  + \, 
{\rm exp} \left( \frac{\mathfrak{J}_{\mathbf{a}}^{\rm hf}(\mathbf{X}) - 1/\epsilon }{C_{\rm exp}} \right)
\, dX \leq \delta,} \\
\end{array}
\end{equation}
where  the proximity measure $\mathbf{f}^k(\mathbf{a}) =  \mathfrak{f} \left(\mathbf{a} , \mu^k   \right)  $   is given by
\begin{equation}
\label{eq:mapping_obj}
\mathfrak{f}\left(
\mathbf{a} , \mu  \right)
:= \; 
\frac{1}{N_{\rm bnd}} \, 
\sum_{i=1}^{N_{\rm bnd}} \, \big\|
\boldsymbol{\Psi}_{\mathbf{a}}^{\rm hf} \left(   \mathbf{X}_i  \right)
\, -  \, \mathbf{x}_i^{\mu}  \big\|_2^2.
\end{equation}
\end{subequations}
Given the dataset $\{ (\mu^k, \mathbf{a}_{\rm hf}^k ) \}_{k=1}^{n_{\rm train}}$, we proceed as in section \ref{sec:generalization} to generate the mapping $\boldsymbol{\Phi}: \Omega \times \mathcal{P} \to \mathbb{R}^d$.

We observe that \eqref{eq:optimization_statement_geo} differs from \eqref{eq:optimization_statement} due to the choice of the proximity measure. Here, $\mathfrak{f}$ is an approximation of the $L^2(\partial \Omega)$ error over the boundary
$$
\mathfrak{f}\left(
\mathbf{a} , \mu  \right) \approx
\mathfrak{f}^{\infty}\left(
\mathbf{a} , \mu  \right) :=
\int_{\partial \Omega} \, 
\|  \boldsymbol{\Psi}_{\mathbf{a}}^{\rm hf} \left(   \mathbf{X}   \right)
\, -  \, \mathbf{X} - \mathbf{d}_{\mu}(\mathbf{X} ) \, \|_2^2 \, dX.
$$
If $\mathfrak{f}^{\infty}(\mathbf{a}^{\rm hf}, \mu) = 0$ for some admissible $\mathbf{a}^{\rm hf} \in \mathbb{R}^M$ and $\Omega$ satisfies the hypotheses of Proposition \ref{th:mapping_general},
recalling Proposition \ref{th:extension_sad},
 we find that $ \boldsymbol{\Psi}_{\mathbf{a}}^{\rm hf}$  is a bijection from $\Omega$ to $\Omega_{\mu}$. 
We further remark that, if $\boldsymbol{\Psi}_{\mathbf{a}}^{\rm hf}$ is a bijection from $\Omega_{\rm box}$ into itself, it is injective in $\Omega$.

\begin{remark}
\textbf{Implicit surfaces.}
If $\partial \Omega_{\mu}$ is represented implicitly as
$\partial \Omega_{\mu} = \{ \mathbf{x} \in \mathbb{R}^d: \; \mathit{G}_{\mu}(\mathbf{x})  = 0  \}$ for $\mathit{G}_{\mu}: \mathbb{R}^d \to \mathbb{R}$, 
then we might consider the proximity measure:
$$
\mathfrak{f}(\mathbf{a}, \mu, \bar{\mu}) \, = \, 
\frac{1}{N_{\rm bnd}} \, 
\sum_{i=1}^{N_{\rm bnd}} \, \big|  \mathit{G}_{\mu}\left(
\boldsymbol{\Psi}_{\mathbf{a}}^{\rm hf} (  \mathbf{X}_i   ) \right)   \big|^2
$$ 
where $\{   \mathbf{X}_i  \}_i$ is the  set of control points on $\partial \Omega$. We do not consider this choice  in the numerical experiments.
\end{remark}

\subsection{Model problems}
\label{sec:gr_model_problem}
We present below the two 
 model problems considered in the numerical experiments.

\subsubsection{Diffusion problem with discontinuous coefficients}
\label{sec:model_pb_diff}

We consider the diffusion problem
\begin{subequations}
\label{eq:diffusion_md_pb}
\begin{equation}
-\nabla \cdot \left( \kappa_{\mu} \nabla z_{\mu} \right) \, = \, 1  \; \;   {\rm in} \; \Omega_{\rm box}, \quad
z_{\mu} \big|_{\partial \Omega_{\rm box}} = 0,
\end{equation}
where 
$\Omega_{\rm box} = (0,1)^2$, 
$\mu= [\mu_1,\mu_2] \in \mathcal{P}= [-0.05,0.05]^2$ and the conductivity coefficient is given by 
\begin{equation}
\label{eq:kappa_diffusion_md_pb}
\kappa_{\mu}(\mathbf{x}):= \, 0. 1 \, + \, 0.9 \, \mathbbm{1}_{\Omega_{\rm  \mu}}(\mathbf{x}), \quad
\Omega_{\rm  \mu}:= \{\mathbf{x} \in \Omega: \, \| \mathbf{x} - \bar{\mathbf{x}}_{\mu} \|_{\infty} \leq \frac{1}{4}  \},
\end{equation}
with $\bar{\mathbf{x}}_{\mu}:= [1/2 +\mu_1,1/2 + \mu_2]$. The problem is a variant of the thermal block problem, which has been extensively considered in the reduced basis literature
(see, e.g., \cite[section 6.1.1]{rozza2007reduced}). 
We here approximate the solution to
\eqref{eq:diffusion_md_pb} through a 
continuous Galerkin P3 FE  discretization with $N_{\rm hf}= 21025$ degrees of freedom.
 \end{subequations}

Since $\kappa_{\mu}$ is piecewise constant with parameter-dependent jump discontinuities,  pMOR techniques are not well-suited to directly tackle  \eqref{eq:diffusion_md_pb}.
It is thus necessary to introduce a mapping $\boldsymbol{\Phi}:\Omega_{\rm box} \times \mathcal{P} \to \Omega_{\rm box}$ such that
$\kappa_{\mu} \circ \Phi_{\mu}: \Omega_{\rm box} \to \mathbb{R}$ is parameter-independent. From the geometry reduction viewpoint, 
given $\Omega := \Omega_{\rm  \mu=0}$, 
we seek $\boldsymbol{\Phi}$  such that
(i) $\boldsymbol{\Phi}_{\mu}(\Omega_{\rm box})=\Omega_{\rm box}$  and  (ii) $\boldsymbol{\Phi}_{\mu}(\Omega )=\Omega_{\mu}$ for all $\mu \in \mathcal{P}$. In view of the discussion, we introduce the mapped problem for a generic map $\boldsymbol{\Phi}_{\mu}:\Omega_{\rm box} \to \Omega_{\rm box}$: 
find $\tilde{z}_{\mu} = z_{\mu}\circ \boldsymbol{\Phi}_{\mu}  \in \mathcal{X} := H_0^1(\Omega_{\rm box})$ such that
\begin{subequations}
\label{eq:affine_diff_pb}
\begin{equation}
\label{eq:affine_diff_eq}
\int_{\Omega_{\rm box}} \, \mathbf{K}_{\mu}^{\star}  \widehat{\nabla} \tilde{z}_{\mu} \cdot \widehat{\nabla}  v \, dX \, = \, \int_{\Omega_{\rm box}} \, \mathfrak{J}_{\mu}  \, v dX,
\; \;
\forall \, v \in   \mathcal{X},
\end{equation} 
with 
\begin{equation}
\label{eq:affine_diff_coeffs}
\mathbf{K}_{\mu}^{\star}:=  \mathfrak{J}_{\mu} \, 
\left(  \kappa_{\mu} \, \circ \boldsymbol{\Phi}_{\mu}  \right)  \, 
\left( \widehat{\nabla} \boldsymbol{\Phi}_{\mu}  \right)^{-1}\, 
\left( \widehat{\nabla} \boldsymbol{\Phi}_{\mu}   \right)^{-T},
\quad
 \mathfrak{J}_{\mu}:= {\rm det} \left(   \widehat{\nabla} \boldsymbol{\Phi}_{\mu}  \right).
\end{equation}
\end{subequations}
{Note that \eqref{eq:diffusion_md_pb} is a special case of the general advection-diffusion-reaction problem considered in section \ref{sec:offline_online_pMOR}.
}

\subsubsection*{Automatic piecewise-affine maps}
Since the deformation of $\Omega_{  \mu}$ is rigid, we might   explicitly build a piecewise linear map $\boldsymbol{\Phi}$.
First, we identify a set of \emph{control points} --- the black dots in Figure \ref{fig:geometry_reduction_model}(a) --- and we use them to build a coarse partition of $\Omega_{\rm box}$ --- shown in Figure \ref{fig:geometry_reduction_model}(a) as well. Then, we define a mapping $\boldsymbol{\Phi}=\boldsymbol{\Phi}^{\rm aff}$ such that (i)
$\Phi_{\mu}^{\rm aff}({\Omega}  ) = \Omega_{\mu} $ for all $\mu \in \mathcal{P}$, and  (ii) $\boldsymbol{\Phi}_{\mu}^{\rm aff}$ is piecewise linear in $\mathbf{X}$ in all elements of the partition.
 Then, we define the mapped problem \eqref{eq:affine_diff_pb} with $\boldsymbol{\Phi}_{\mu}=\boldsymbol{\Phi}_{\mu}^{\rm aff}$.

Note that, for this choice of the mapping,  $\mathbf{K}_{\mu}^{\star} $ and $ \mathfrak{J}_{\mu}$
in \eqref{eq:affine_diff_coeffs}
 are piecewise-constant in each element of the partition: this implies that \eqref{eq:affine_diff_pb} is \emph{parametrically-affine} --- that is, 
 $\mathbf{K}_{\mu}^{\star}$ and $ \mathfrak{J}_{\mu} $ are linear combinations of parameter-dependent coefficients and parameter-independent spatial fields. Therefore, the solution $\widetilde{z}_{\mu}$  can be efficiently  approximated using standard pMOR (e.g., Reduced Basis) techniques.
 
Given the user-defined control points,  
 Rozza el al. have developed in \cite{rozza2007reduced}  an automatic procedure to generate the partition of $\Omega_{\rm box}$ and to determine an economic  piecewise-constant approximation of the form in \eqref{eq:affine_diff_pb}. The latter relies on symbolic manipulation techniques to identify redundant terms in the affine expansions.  Furthermore, for the approach to be effective, we should consider FE triangulations that are conforming to the coarse-grained partition in Figure \ref{fig:geometry_reduction_model}(a). Finally, the choice of the control points, which is trivial in this case, might not be straightforward for more challenging problems.
 In conclusion, even if the approach works remarkably well in many situations, practical implementation of the procedure in \cite{rozza2007reduced}   is rather involved and highly problem-dependent.

\subsubsection{Laplace's equation in parameterized domains}
\label{sec:laplace_gr}

Given the parameter domain $\mathcal{P} = [0.1,0.4]^2 \times [0,\pi/4]$, we first introduce the parametric closed curve $t \in [0,2\pi) \mapsto \boldsymbol{\gamma}_{\rm in, \mu}(t) \in \mathbb{R}^2$ such that
\begin{subequations}
\label{eq:laplace_model_problem}
\begin{equation}
\boldsymbol{\gamma}_{\rm in, \mu}(t) = \left[
\begin{array}{l}
\displaystyle{ \cos(t) \left(  1 + \mu_1 \, \left( \cos (t + \mu_3) \right)^2 \, + \, 2 \cdot 10^{-3} \left( (2\pi - t) \, t \right)^2 \right) }
\\[3mm]
\displaystyle{ \sin(t) \left(  1 + \mu_2 \, \left( \sin (t + \mu_3) \right)^2 \, + \, 2 \cdot 10^{-3} \left( (2\pi - t) \, t \right)^2 \right)  }
\\
\end{array}
\right]
\end{equation}
and we denote by $\Omega_{\rm in,  \mu}$ the bounded domain such that $\partial \Omega_{\rm in,  \mu} = \boldsymbol{\gamma}_{\mu}([0,2\pi])$ for all $\mu \in \mathcal{P}$.
We further define $\Omega_{\rm box} = (-2,2)^2$: note that
$\Omega_{\rm in, \mu} \Subset \Omega_{\rm box}$ for all $\mu \in \mathcal{P}$. Then, we introduce the Laplace's problem:
\begin{equation}
\label{eq:laplace_model_problem_b}
-\Delta \, z_{\mu} \, = \, 0  \quad {\rm in} \, \Omega_{\mu}, \qquad
z_{\mu} \big|_{ \partial \Omega_{\rm in, \mu}    } \, = \, 1,
\quad
z_{\mu} \big|_{ \partial \Omega_{\rm box}    } \, = \, 0,
\end{equation}
where $\Omega_{\mu}  := \Omega_{\rm box} \setminus \Omega_{\rm in, \mu}$ is depicted in Figures \ref{fig:geometry_reduction_model}(c) and \ref{fig:geometry_reduction_model}(d) for two values of $\mu \in \mathcal{P}$.
\end{subequations}

We introduce the reference domain $\Omega = \Omega_{\rm box}  \setminus \Omega_{\rm in}$, with 
$ \Omega_{\rm in} = \mathcal{B}_1(\mathbf{0})$: note that $\Omega$ is diffeomorphic  to $\Omega_{\mu}$ for all $\mu \in \mathcal{P}$. Then, given the bijection $\boldsymbol{\Phi}_{\mu}: \Omega \to \Omega_{\mu}$, 
and the lift $R_{\rm D}$ such that 
\begin{subequations}
\label{eq:laplace_weak_mapped}
\begin{equation}
\label{eq:lift_laplace}
-\Delta \, R_{\rm D}  \, = \, 0  \quad {\rm in} \,   \Omega, \qquad
R_{\rm D} \big|_{ \partial \Omega_{\rm in}    } \, = \, 1,
\quad
R_{\rm D}  \big|_{ \partial \Omega_{\rm box}    } \, = \, 0,
\end{equation}
we define the mapped  problem for the lifted solution: find $\tilde{z}_{\mu} \in  \mathcal{X}:= H_0^1(\Omega)$ such that 
\begin{equation}
\int_{\Omega} \, \mathbf{K}_{\mu}^{\star}  \widehat{\nabla} \left( \tilde{z}_{\mu} + R_{\rm D}  \right) \cdot \widehat{\nabla}  v \, dX \, = \, 0,
\; \;
\forall \, v \in   \mathcal{X},
\end{equation} 
with  $\mathbf{K}_{\mu}^{\star} :=  \mathfrak{J}_{\mu} \, 
\left( \widehat{\nabla} \boldsymbol{\Phi}_{\mu}  \right)^{-1}\, 
\left( \widehat{\nabla} \boldsymbol{\Phi}_{\mu}   \right)^{-T}$.
\end{subequations}

\begin{figure}[h!]
\centering

\subfloat[] 
 { 
\begin{tikzpicture}[scale=0.75]
\linethickness{0.3 mm}
\draw[ultra thick]  (0,0)--(4,0)--(4,4)--(0,4)--(0,0);
\draw[ultra thick]  (1,1)--(3,1)--(3,3)--(1,3)--(1,1);
\draw[ultra thick]  (0,0)--(1,1)--(2,0)--(3,1)--(4,0);
\draw[ultra thick]  (0,0)--(1,1)--(0,2)--(1,3)--(0,4);

\draw[ultra thick]  (0,4)--(1,3)--(2,4)--(3,3)--(4,4);

\draw[ultra thick]  (4,0)--(3,1)--(4,2)--(3,3)--(4,4);

  \coordinate   (E) at (0,0) ;  \fill[black] (E) circle[radius=4pt];
\coordinate   (E) at (4,0) ;  \fill[black] (E) circle[radius=4pt];
\coordinate   (E) at (0,4) ;  \fill[black] (E) circle[radius=4pt];
\coordinate   (E) at (4,4) ;  \fill[black] (E) circle[radius=4pt];

  \coordinate   (E) at (1,1) ;  \fill[black] (E) circle[radius=4pt];
\coordinate   (E) at (1,3) ;  \fill[black] (E) circle[radius=4pt];
\coordinate   (E) at (3,3) ;  \fill[black] (E) circle[radius=4pt];
\coordinate   (E) at (3,1) ;  \fill[black] (E) circle[radius=4pt];

  \coordinate   (E) at (0,2) ;  \fill[black] (E) circle[radius=4pt];
\coordinate   (E) at (4,2) ;  \fill[black] (E) circle[radius=4pt];
\coordinate   (E) at (0,2) ;  \fill[black] (E) circle[radius=4pt];
\coordinate   (E) at (4,2) ;  \fill[black] (E) circle[radius=4pt];

\coordinate [label={center:  {\large {$ {\Omega}$}}}] (E) at (2, 2) ;
\coordinate [label={right:  {\large {$ {\Omega}_{\rm box}$}}}] (E) at (4, 4.1) ;

\end{tikzpicture}
}
~~~~~~
\subfloat[] 
 { 
\includegraphics[width=0.34\textwidth]
 {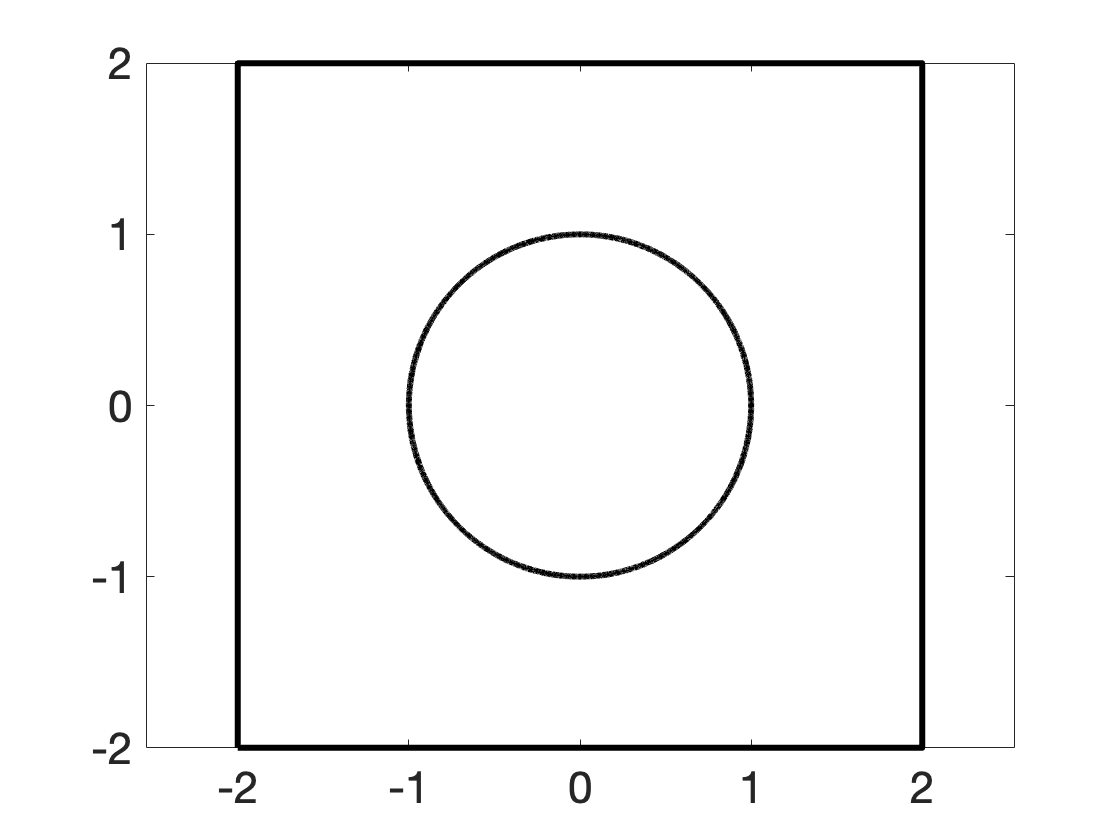}
 }
 
\subfloat[] 
 { \includegraphics[width=0.34\textwidth]
 {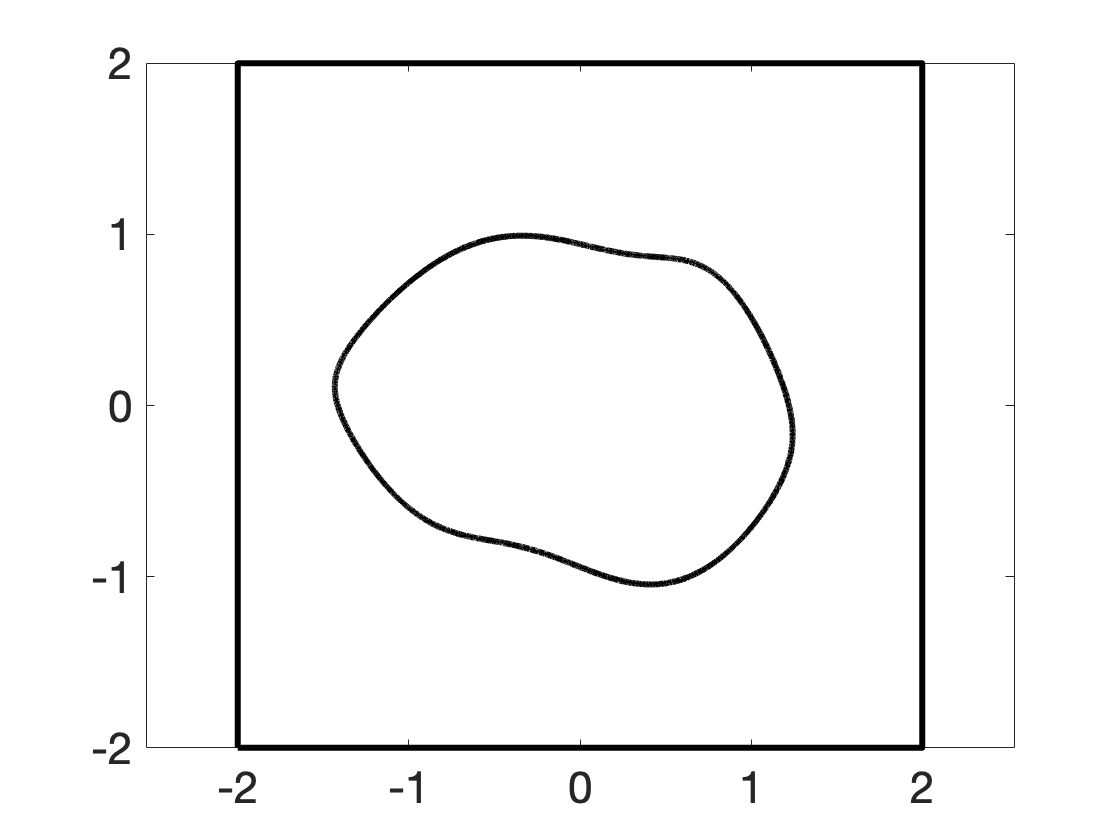} } 
~~~~~
 \subfloat[] 
 { \includegraphics[width=0.34\textwidth]
 {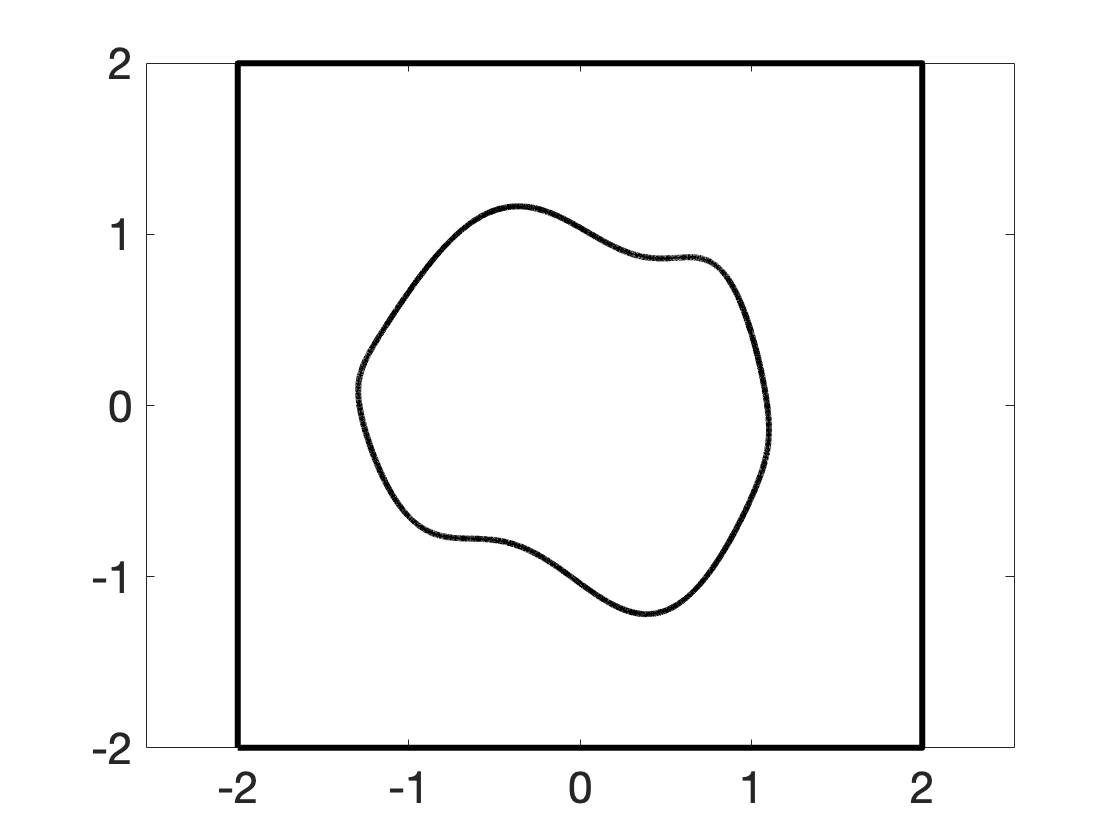} } 
 
\caption{geometry reduction.
(a): affine geometry (section \ref{sec:model_pb_diff}).  Black dots denote the user-defined control points associated with the coarse triangulation.
(b)-(d): non-affine geometry (section \ref{sec:laplace_gr}).
(b):  reference domain $\Omega$;
(c)-(d): domain  $\Omega_{\mu}$ for two values of $\mu\in \mathcal{P}$.
}
\label{fig:geometry_reduction_model}
\end{figure}
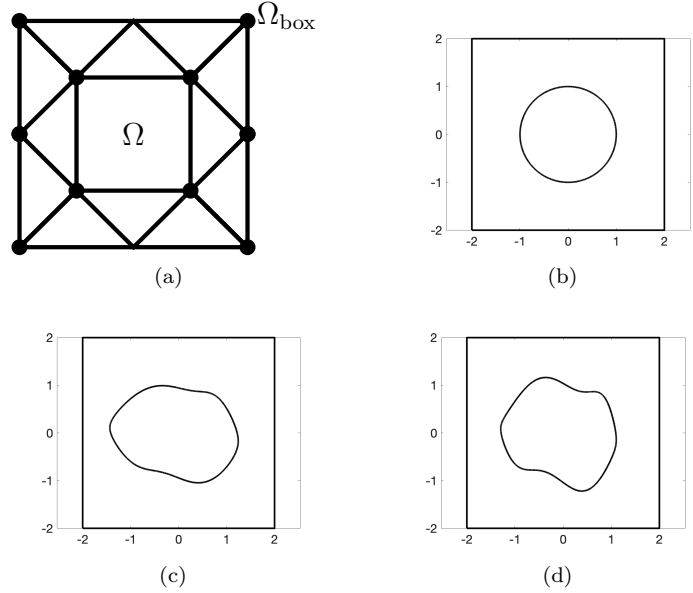 

\subsubsection{Application of the registration procedure}
\label{sec:gr_registration}

We consider the parametric function $\boldsymbol{\Psi}^{\rm hf}$  
\eqref{eq:prescribed_mapping_form} - \eqref{eq:tensorized_polynomials}: thanks to this choice, 
control points on $\partial \Omega_{\rm box}$ are not necessary. Since, as in the examples of section \ref{sec:data_compression}, the conductivity matrix $\mathbf{K}_{\mu}^{\star}$ and the Jacobian determinant $\mathfrak{J}_{\mu}$ are not expected to be affine, we resort to EIM,
\begin{equation}
\label{eq:affine_diff_eq_EIM}
\mathbf{K}_{\mu}^{\star} \approx \mathbf{K}_{\mu}^{\rm eim} :=
\sum_{q=1}^{Q_{\rm a}} \, \Theta_{\mu}^{q, \kappa} \, \mathbf{K}_{q},
\quad
\mathfrak{J}_{\mu}^{\star}  \approx \mathfrak{J}_{\mu}^{\rm eim} :=
\sum_{q=1}^{Q_{\rm f}} \, \Theta_{\mu}^{q, \mathfrak{j}} \, \mathfrak{J}_{q},
\end{equation}
to obtain a parametrically-affine surrogate model. 
As explained in section \ref{sec:eim}, the basis functions are built using POD: the size of the expansions $Q_{\rm a}, Q_{\rm f}$ is chosen based on the criterion in \eqref{eq:POD_cardinality_selection}.

Some comments are in order. Our approach leads to a non-affine formulation ---and thus  requires hyper-reduction to achieve online efficiency --- for both problems; on the other hand, we observe that the approach  
exclusively relies on the parameterization of the boundary $\partial \Omega_{\rm in, \mu}$
 and can be applied   for virtually any choice of the conductivity.

\subsection{Numerical results}
\label{sec:gr_numerics}
We present below results for the two model problems introduced in section \ref{sec:gr_model_problem}. Since the focus of this section is geometry reduction, we do not discuss the construction of the ROM for the mapped problems  \eqref{eq:affine_diff_pb} and \eqref{eq:laplace_weak_mapped}.

\subsubsection{Diffusion problem with discontinuous coefficients}
\label{sec:results_diff}

We choose $\xi=4 \cdot 10^{-4}$,   $tol_{\rm POD}= 10^{-5}$, $\overline{M} = 6$ ($M_{\rm hf}=72$); furthermore, we consider the proximity measure  \eqref{eq:mapping_obj} where $\{ \mathbf{X}_i \}_{i=1}^{N_{\rm bnd}}$ is an uniform grid of $\partial \Omega_{\rm in}$ with $N_{\rm bnd} = 400$ and $\{ \mathbf{x}_i^{\mu} = \mathbf{X}_i \, + \, \mu  \}_{i}$; finally, we consider $n_{\rm train}=16^2$ equispaced  parameters in $\mathcal{P}$.
For this choice of the parameters, our procedure  returns an affine expansion with $M=6$ terms.

In Figure \ref{fig:diff_coeffs}, we investigate the (linear) complexity of the parametric manifolds associated with \eqref{eq:diffusion_md_pb}.  In Figure \ref{fig:diff_coeffs}(a), we compute the $L^2(\Omega_{\rm box})$-POD eigenvalues associated with $\{ \kappa_{\mu^k} \}_{k=1}^{n_{\rm train}}$ (unregistered) and
$\{ \widetilde{\kappa}_{\mu^k}
= \kappa_{\mu^k} \circ \boldsymbol{\Phi}_{\mu^k} \}_{k=1}^{n_{\rm train}}$ (registered). We observe that $\lambda_{N=2} \leq 10^{-15} \lambda_{N=1}$ for the registered case: the mapping $\boldsymbol{\Phi}_{\mu}$ is able to ``fix" the position of the jump discontinuity in the reference configuration\footnote{
More precisely, if we denote by $\{ \mathbf{x}_q^{\rm in}  \}_q$ (resp. $\{ \mathbf{x}_q^{\rm out}  \}_q$)
the FE quadrature points in $\Omega $ (resp. $\Omega_{\rm box} \setminus \Omega$), we have that
$\boldsymbol{\Phi}_{\mu} (   \{ \mathbf{x}_q^{\rm in}  \}_q  ) \subset \Omega_{  \mu}$
(resp. $\boldsymbol{\Phi}_{\mu} (   \{ \mathbf{x}_q^{\rm out}  \}_q  ) \subset 
\Omega_{\rm box} \setminus \Omega_{  \mu}$) for all $\mu \in \mathcal{P}$.}. In Figure \ref{fig:diff_coeffs}(b), we show the $L^2(\Omega_{\rm box})$-POD eigenvalues of $\{ \mathbf{K}_{\mu^k}^{\star}  \}_{k=1}^{n_{\rm train}}$ and $\{ \mathfrak{J}_{\mu^k}  \}_{k=1}^{n_{\rm train}}$. Finally, in Figure \ref{fig:diff_coeffs}(c), we show the behavior of the $H^1(\Omega_{\rm box})$-POD eigenvalues  of 
$\{ z_{\mu^k} \}_{k=1}^{n_{\rm train}}$ (unregistered) and
$\{ \widetilde{z}_{\mu^k} \}_{k=1}^{n_{\rm train}}$ (registered):  even if the mapping is built based on the diffusivity coefficient, it is also effective in improving the linear reducibility of the solution manifold.  

\begin{figure}[h!]
\centering
 \subfloat[] {\includegraphics[width=0.3\textwidth]
 {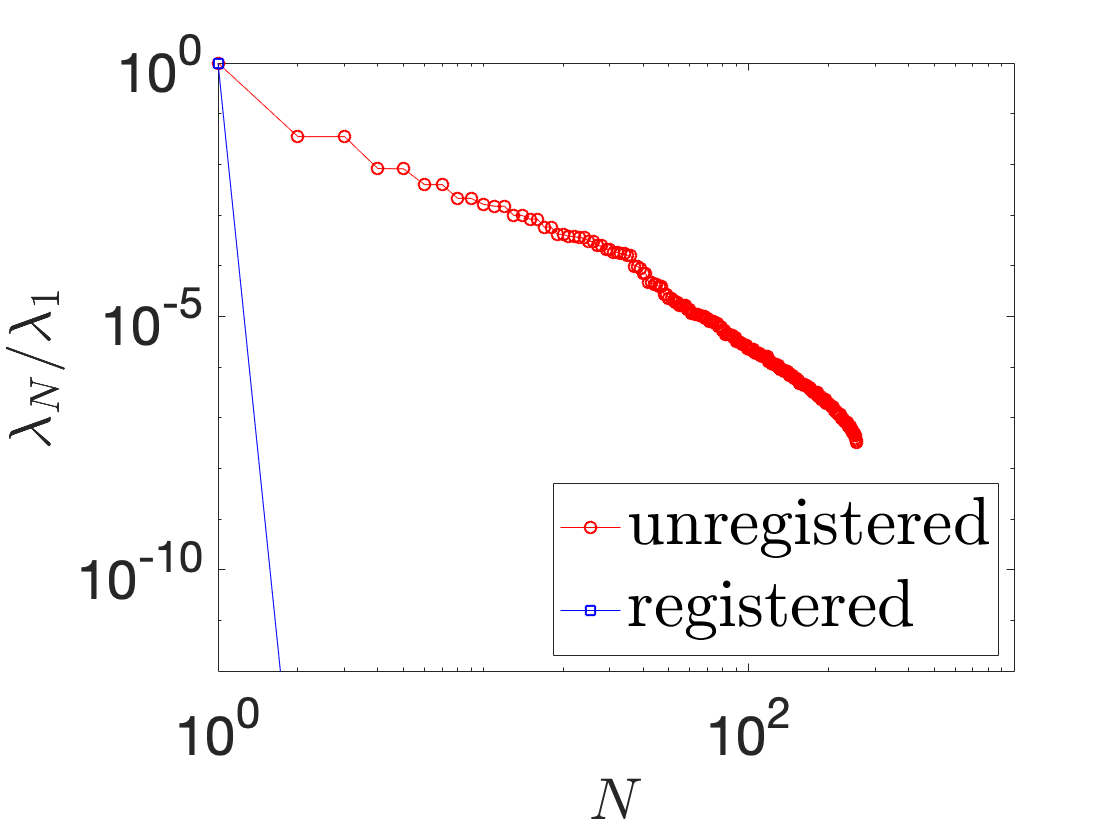}}  
    ~ ~
\subfloat[] {\includegraphics[width=0.3\textwidth]
 {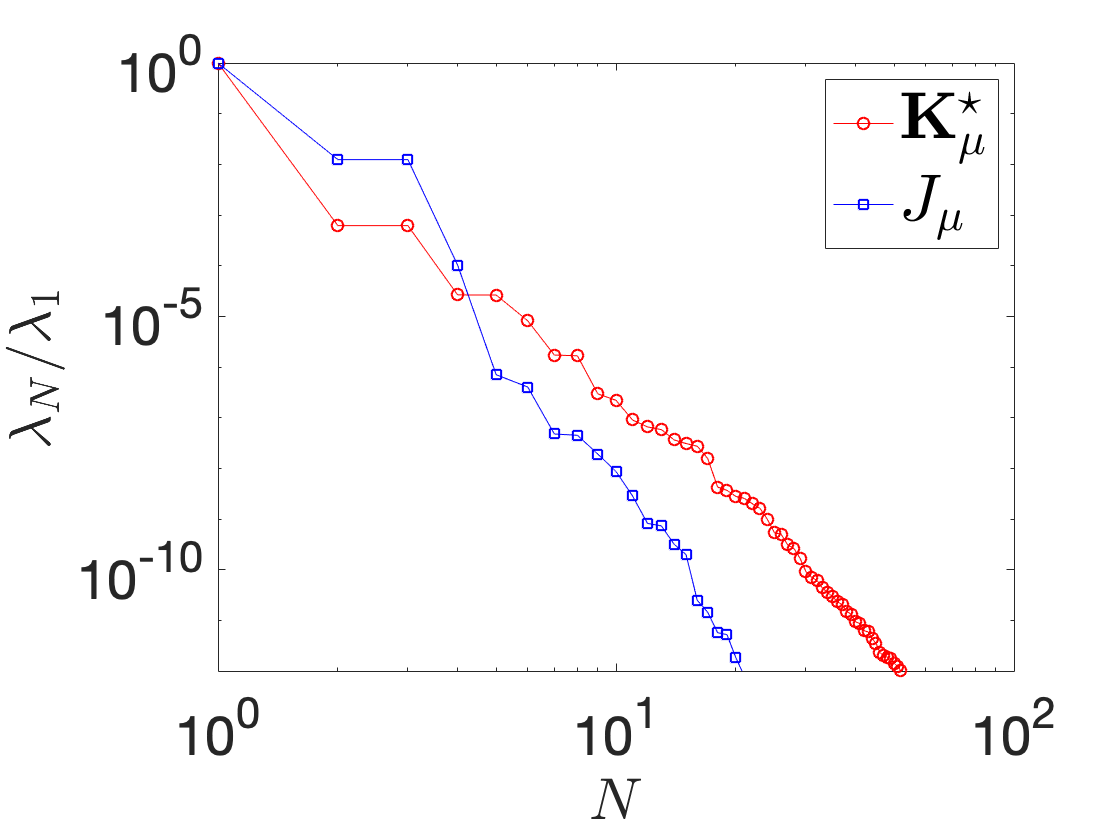}}  
     ~ ~
\subfloat[] {\includegraphics[width=0.3\textwidth]
 {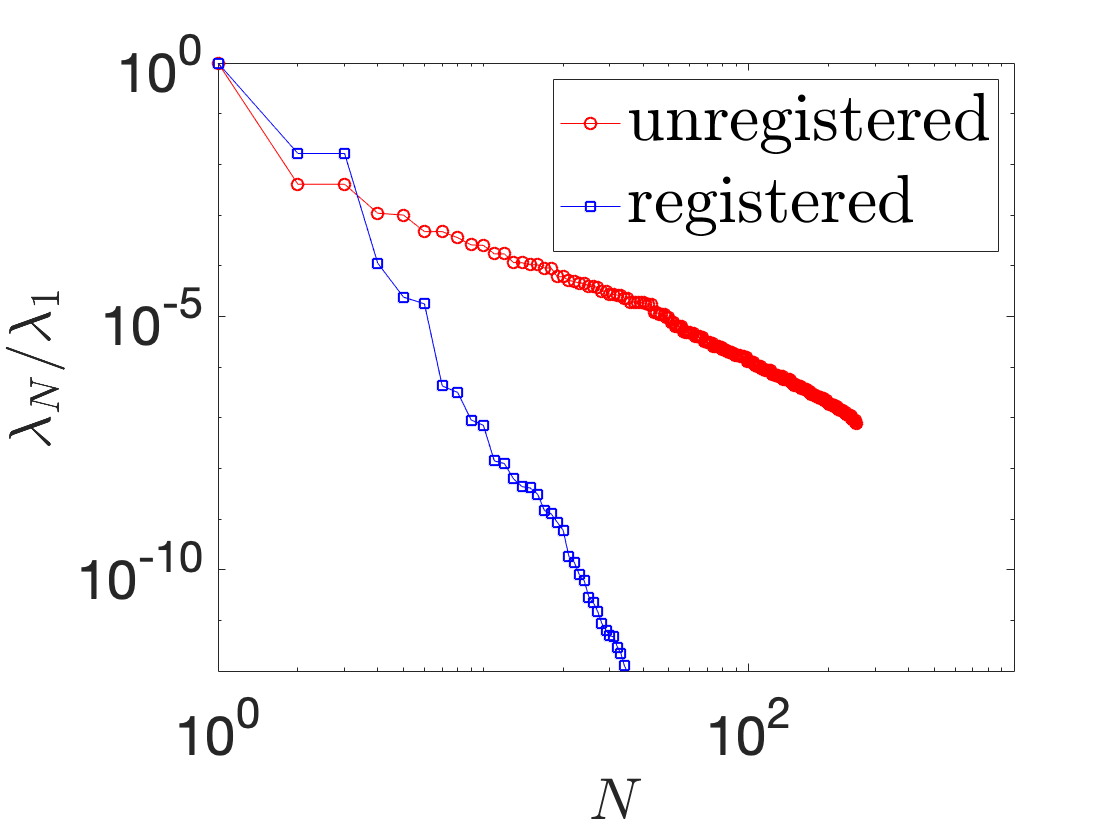}}   
 
\caption{diffusion problem with discontinuous coefficients. Behavior of  POD eigenvalues  
associated with 
$\{ \kappa_{\mu^k} \}_k$ 
and $\{ \tilde{\kappa}_{\mu^k} \}_k$
(Fig.  (a)),
 $\{ \mathbf{K}_{\mu^k}^{\star} \}_k$ and
$\{ \mathfrak{J}_{\mu^k} \}_k$
(Fig. (b)), and
$\{ z_{\mu^k} \}_k$ and
$\{ \tilde{z}_{\mu^k} \}_k$
(Fig. (c)).
}
 \label{fig:diff_coeffs}
\end{figure} 

In Figure \ref{fig:diff_eim}, we show the behavior of the mean  relative error 
\begin{equation}
\label{eq:Eavg_diff}
E_{\rm avg}:= 
\frac{1}{n_{\rm test}}  \, \sum_{i=1}^{n_{\rm test}}  \,
 \frac{\| \tilde{z}_{\mu^i}  - \tilde{z}_{\mu^i}^{\rm eim}   \|_{\star}}{\| \tilde{z}_{\mu^i} \|_{\star}},
\quad
\mu^1,\ldots,\mu^{n_{\rm test}} \overset{\rm iid}{\sim} {\rm Uniform}(\mathcal{P}),
\end{equation}
with respect to the tolerance $tol_{\rm eim}$ associated with  the choice of $Q_{\rm a}, Q_{\rm f}$ in \eqref{eq:affine_diff_eq_EIM}, for $n_{\rm test}=20$. 
Here, $\| \cdot \|_{\star}$ denotes either the $L^2(\Omega)$ norm or the $H^1(\Omega)$ norm:
$$
\| w \|_{L^2(\Omega_{\rm box})}^2
=
\int_{\Omega_{\rm box}} \, |w|^2\, dx,
\quad
\| w \|_{H^1(\Omega_{\rm box})}^2
=
\int_{\Omega_{\rm box}} \, |w|^2\, + \| \nabla w \|_2^2 dx.
$$
In Figure \ref{fig:diff_eim}(b), we show the behavior of 
$Q_{\rm a}, Q_{\rm f}$ with respect to  $tol_{\rm eim}$: we observe that  the relative $H^1$  and $L^2$ errors are less than $10^{-3}$, for $Q_{\rm a}=19, Q_{\rm f}=9$.
{We remark that using piecewise-affine mappings (cf. section \ref{sec:model_pb_diff}) we might obtain an \emph{exact} affine form with $Q_{\rm a}=13, Q_{\rm f}=3$.
Once again, we remark   that  our approach is fully automatic.
}

\begin{figure}[h!]
\centering
 \subfloat[] {\includegraphics[width=0.4\textwidth]
 {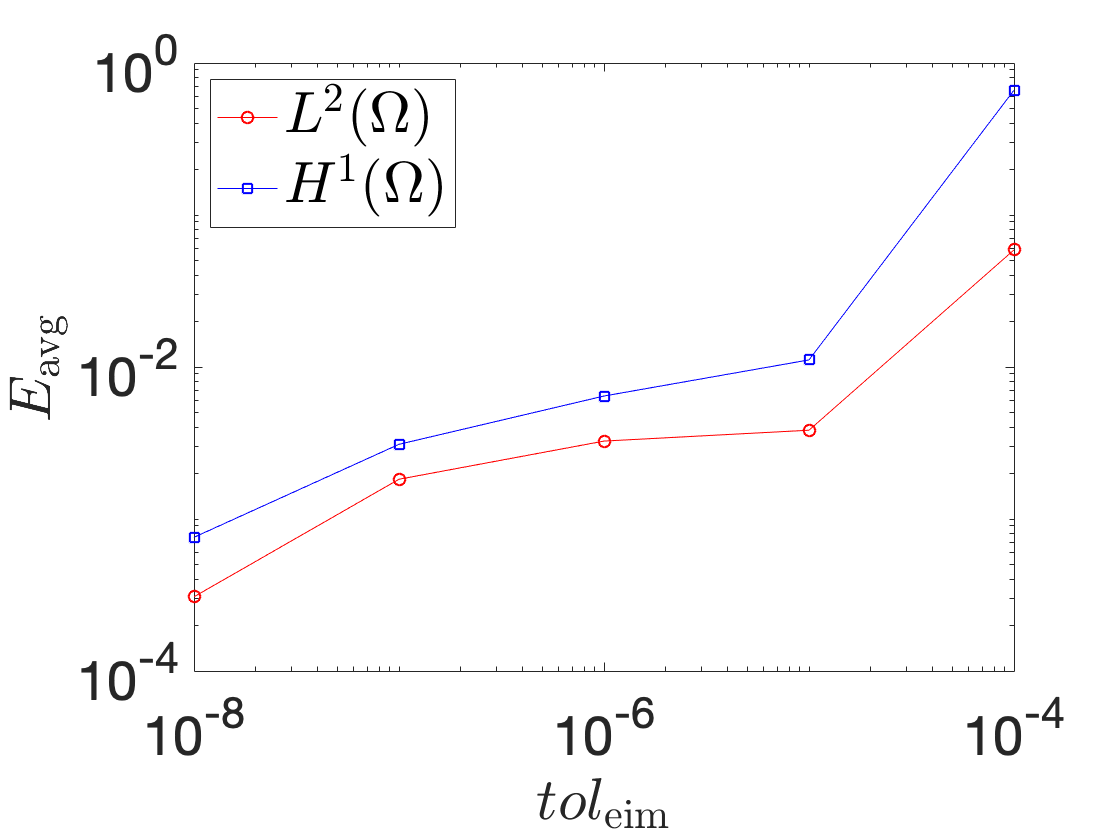}}  
    ~ ~
\subfloat[] {\includegraphics[width=0.4\textwidth]
 {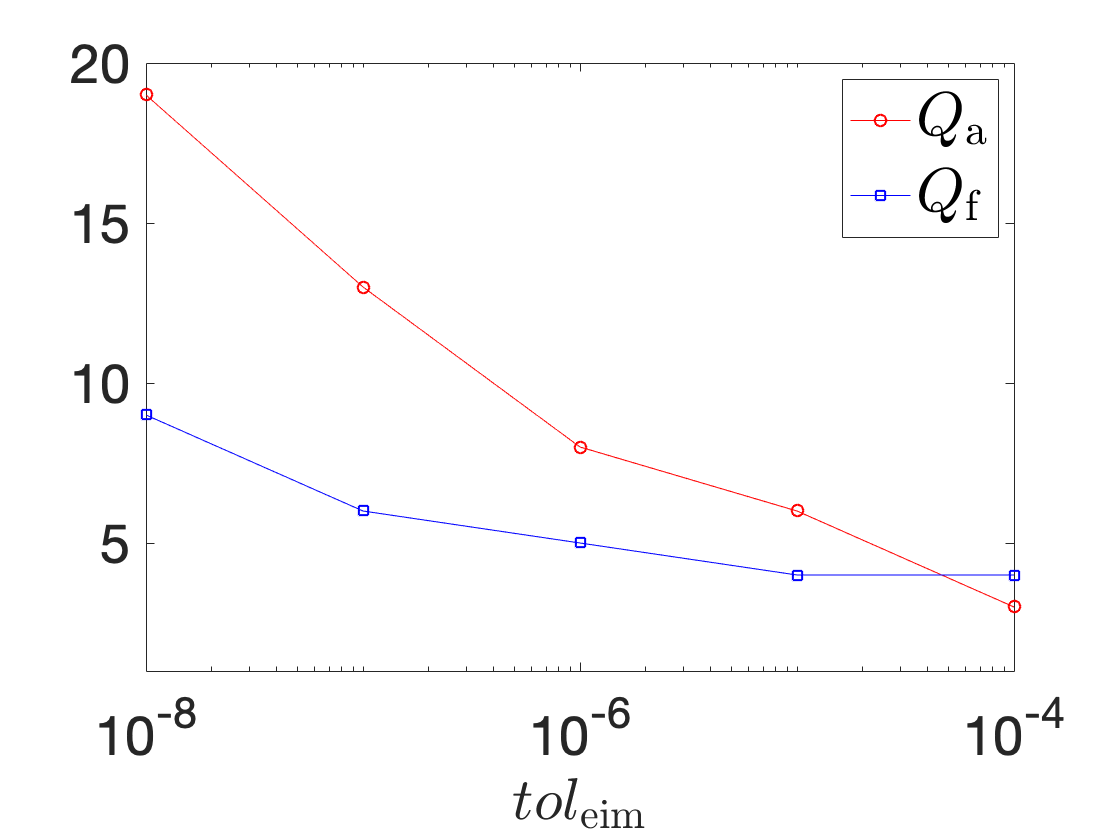}}  
 
\caption{diffusion problem with discontinuous coefficients.
(a): behavior of $E_{\rm avg}$ \eqref{eq:Eavg_diff} with $tol_{\rm eim}$
($n_{\rm test}=20$).
(b): behavior of $Q_{\rm a},Q_{\rm f}$  in \eqref{eq:affine_diff_eq_EIM} with $tol_{\rm eim}$.
}
 \label{fig:diff_eim}
\end{figure}

\subsubsection{Laplace's equation in parameterized domains}

{We choose $\xi=10^{-2}$, $tol_{\rm POD}=10^{-5}$, $M_{\rm hf}=288$,  and we define the proximity measure $\mathfrak{f}$ \eqref{eq:mapping_obj} based on $N_{\rm bnd}=10^{3}$ points
$\{ \mathbf{X}_i= \boldsymbol{\gamma}^{\rm ref}( 2 \pi i /N_{\rm bnd} )  \}_{i=1}^{N_{\rm bnd}}$ with 
$\boldsymbol{\gamma}^{\rm ref}(t)=[\cos(t), \sin(t)]$.
To generate the mapping, we consider uniform grids of $\mathcal{P}$ $\{  \mu^k \}_{k=1}^{n_{\rm train}}$ with $n_{\rm train}=4^3, 6^3, 8^3, 10^3$. }

{In Figure \ref{fig:registration_laplace},  
we show boxplots of the ``in-sample"  and ``out-of-sample"  errors
given by
\begin{equation}
\label{eq:Egeo}
\left\{
\begin{array}{ll}
\displaystyle{
E_k^{\rm geo,in} := \max_{i=1,\ldots,N_{\rm bnd}} \, \| \boldsymbol{\Psi}_{\mathbf{a}_{\rm hf}^k}^{\rm hf}(\mathbf{X}_i) - \mathbf{x}_i^{\mu^k} \|_2} &
\displaystyle{k=1,\ldots,n_{\rm train}} \\[3mm]
\displaystyle{
E_j^{\rm geo,out} := \max_{i=1,\ldots,N_{\rm bnd}^{\rm t}} \, {\rm dist} \left(
 \boldsymbol{\Phi}_{\tilde{\mu}^j}(\mathbf{X}_i^{\rm t}) , \, 
\{   \mathbf{x}_{i',j}^{\rm t}    \}_{i'=1}^{ N_{\rm bnd}^{\rm t}  }
\right)}
&
\displaystyle{j=1,\ldots,n_{\rm test}} \\
\end{array}
\right.
\end{equation}
where $\{\mathbf{X}_i^{\rm t}= \boldsymbol{\gamma}^{\rm ref}(t^i)       \}_i$ and $\{\mathbf{x}_{i,j}^{\rm t}= \boldsymbol{\gamma}_{\tilde{\mu}^j} (t^i)       \}_i$, with $t^i= \frac{2 \pi i}{N_{\rm bnd}^{\rm t}}$, $\tilde{\mu}^1,\ldots,\tilde{\mu}^{n_{\rm test}} \overset{\rm iid}{\sim} {\rm Uniform}(\mathcal{P})$, 
$N_{\rm bnd}^{\rm t}=10^4$, 
$n_{\rm test}=50$.}
 In-sample error depends on the choice of the hyper- parameters in \eqref{eq:optimization_statement} and on the choice of the discretization parameter $M_{\rm hf}$; on the other hand, out-of-sample error depends on the choice of $tol_{\rm POD}$ and on the number of training points. Interestingly, the value of $M$ is the same (and equal to $7$) for all choices of $n_{\rm train}$.

\begin{figure}[h!]
\centering
 \subfloat[$n_{\rm train}=4^3$] {\includegraphics[width=0.4\textwidth]
 {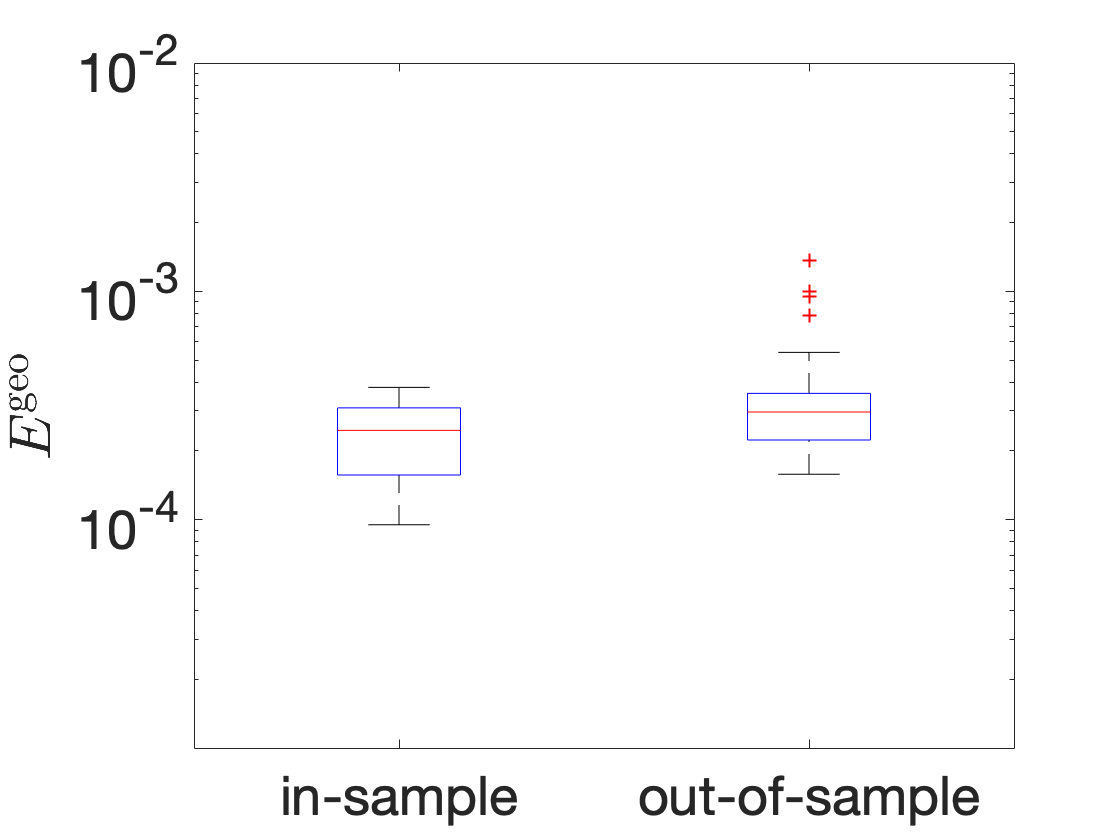}}  
    ~ ~
\subfloat[$n_{\rm train}=6^3$] {\includegraphics[width=0.4\textwidth]
 {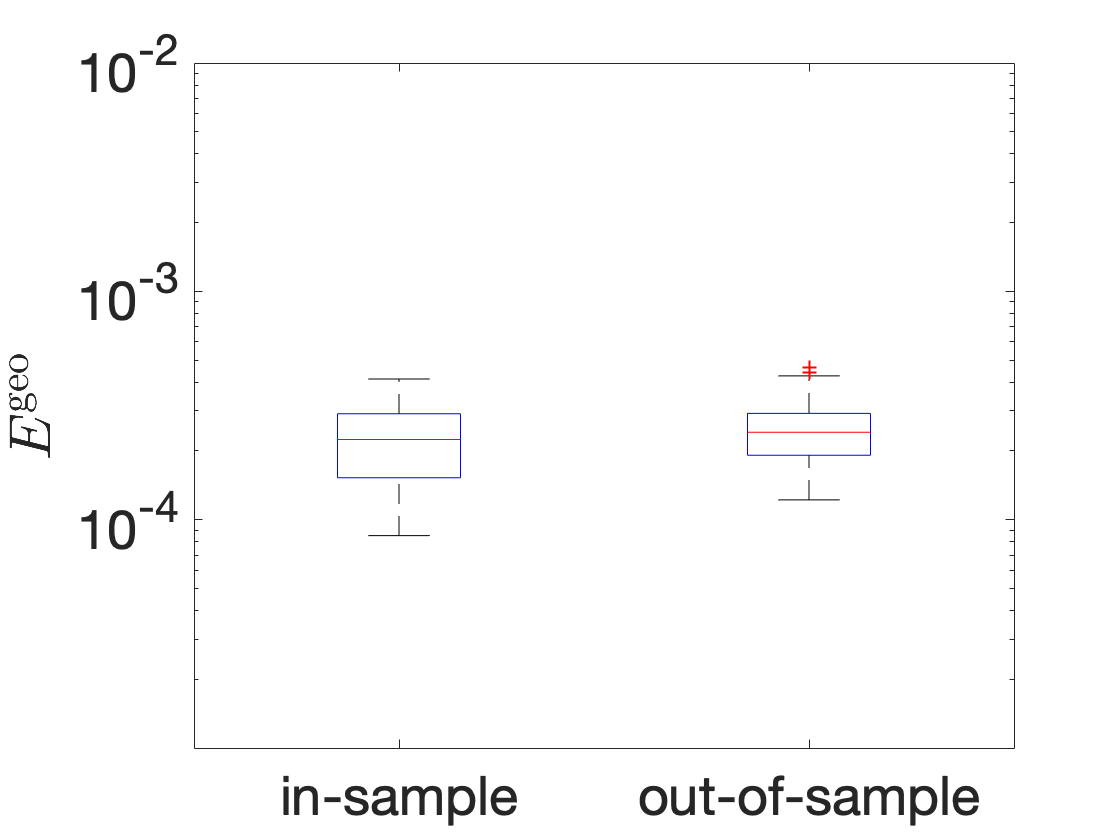}}

\subfloat[$n_{\rm train}=8^3$] {\includegraphics[width=0.4\textwidth]
 {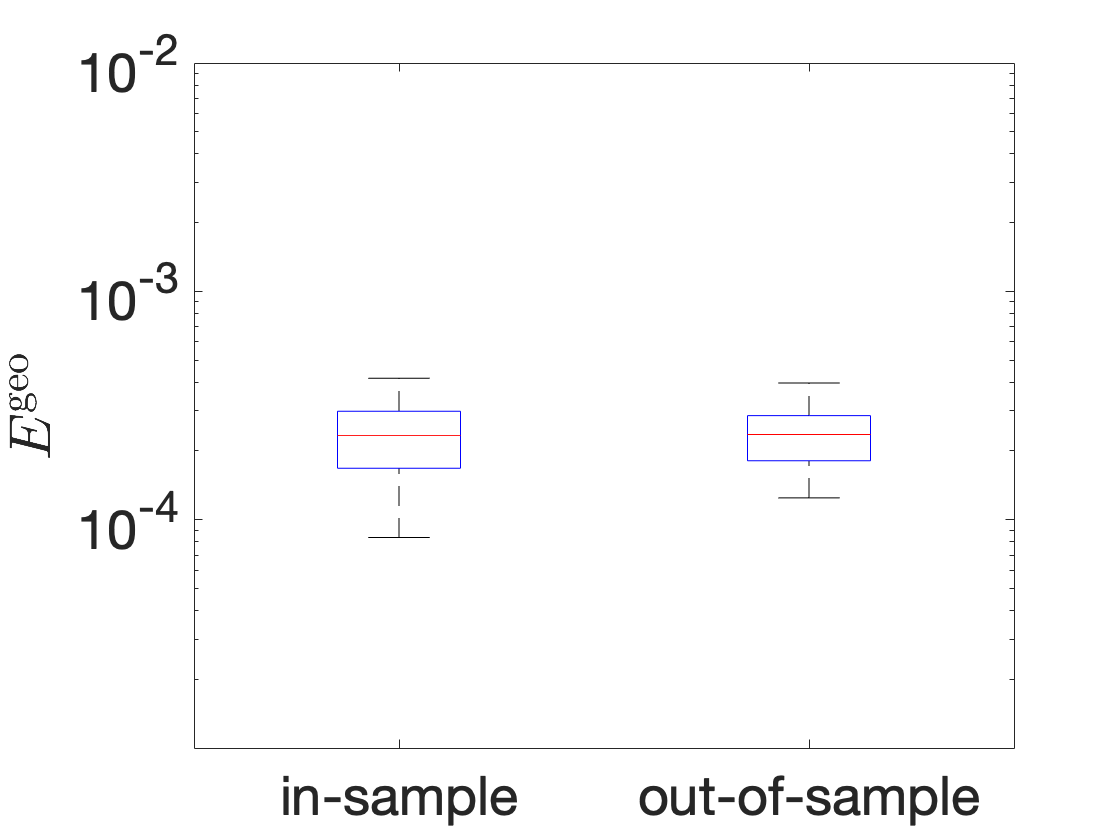}}  
     ~ ~
\subfloat[$n_{\rm train}=10^3$] {\includegraphics[width=0.4\textwidth]
 {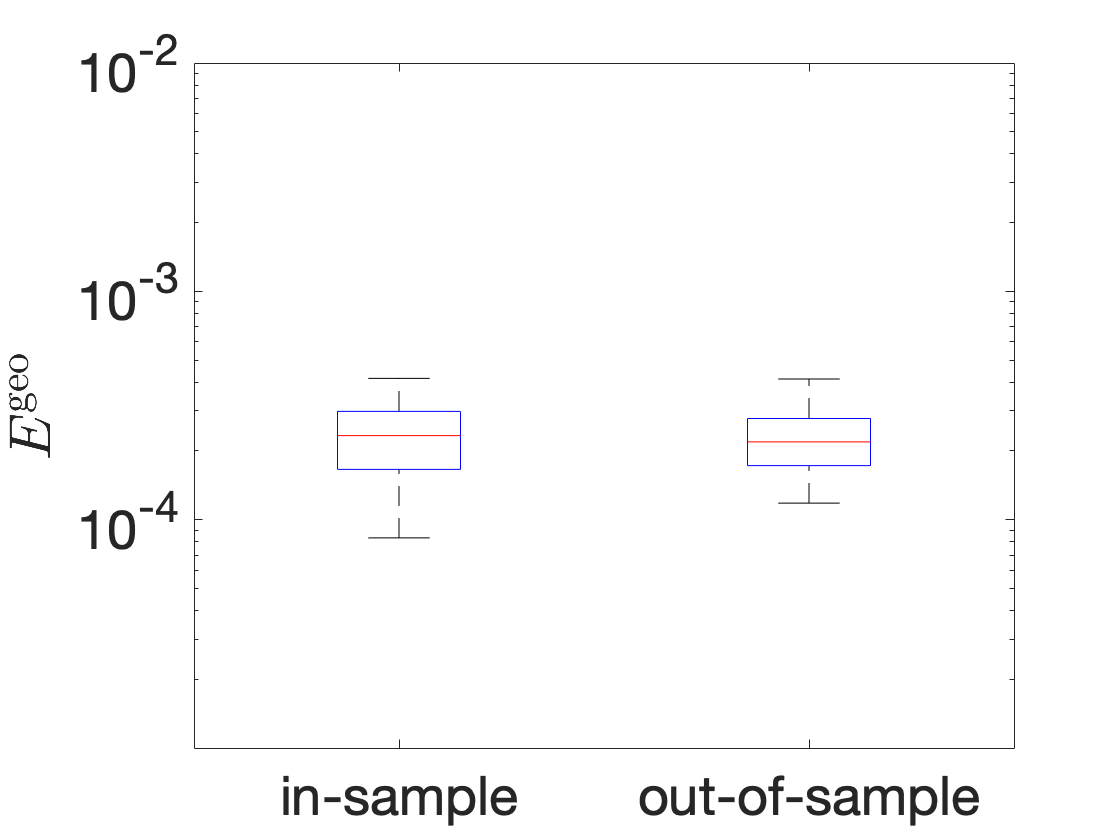}}

\caption{Laplace's equation in parameterized domain; performance of the registration procedure.
Boxplots of in-sample and out-of-sample errors 
 $E_{\rm geo}$ for several training sizes.
}
 \label{fig:registration_laplace}
\end{figure} 

Figure \ref{fig:registration_laplace_2} replicates the results of Figure \ref{fig:diff_eim} for the second model problem.
Here, we consider $n_{\rm train}=10^3$, $\xi=10^{-2}$, $M_{\rm hf}=288$, $tol_{\rm POD}=10^{-5}$,   and we compute 
$E_{\rm avg}$ \eqref{eq:Eavg_diff} based on $n_{\rm test}=50$ parameters.
We find that the relative $L^2$ and $H^1$ errors
are below $10^{-3}$ for $Q_{\rm a}$ equal to $17$.

\begin{figure}[h!]
\centering
 \subfloat[] {\includegraphics[width=0.4\textwidth]
 {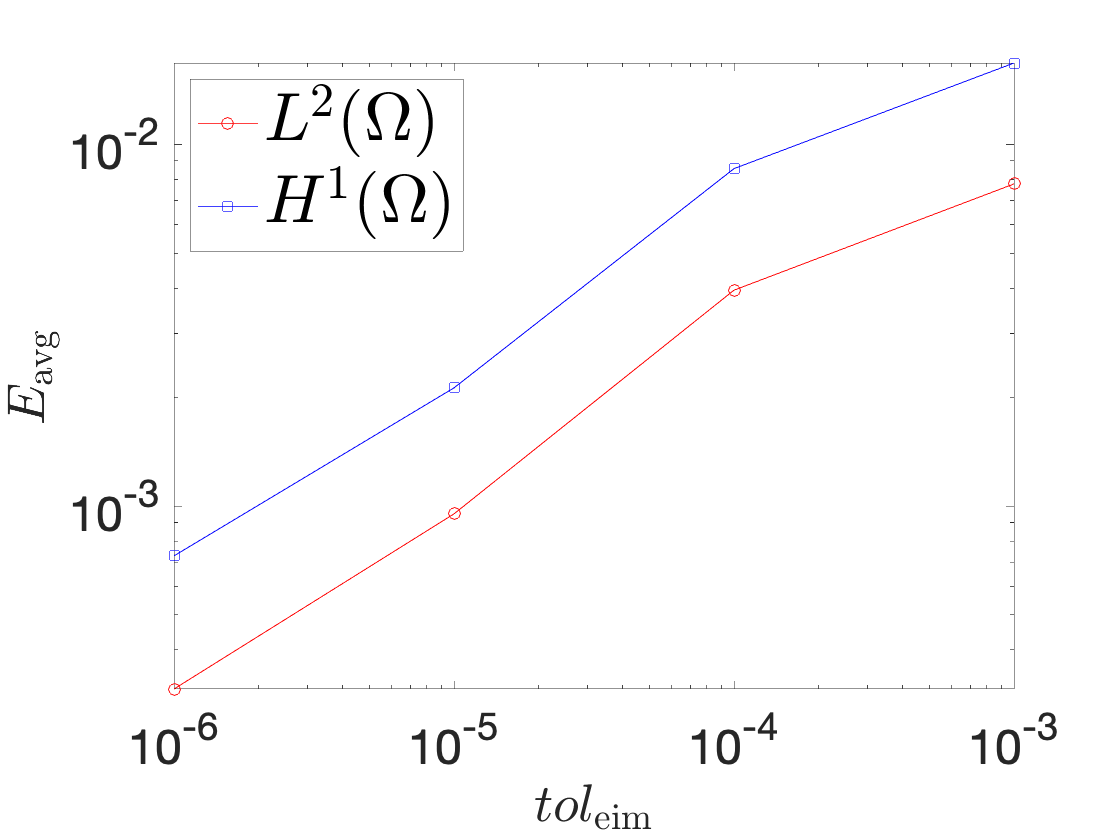}}  
    ~ ~
\subfloat[] {\includegraphics[width=0.4\textwidth]
 {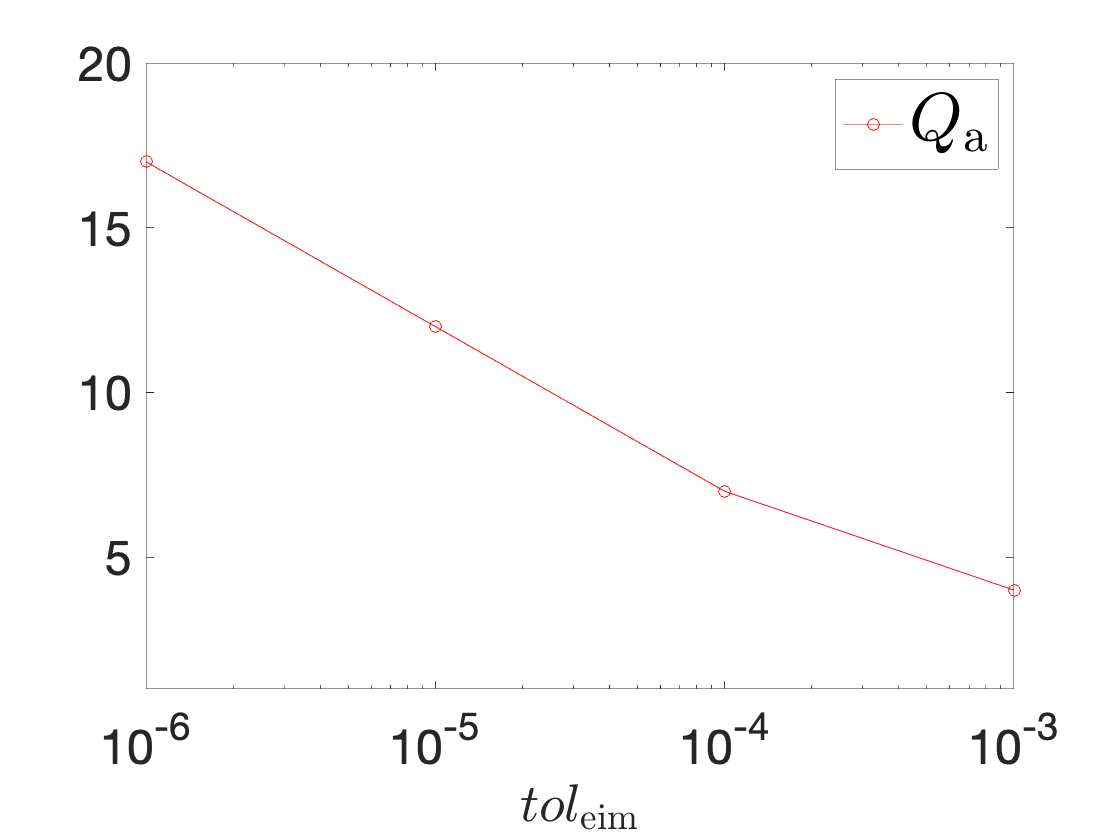}}  
 
\caption{Laplace's equation in parameterized domain; performance of the registration procedure. 
(a):
behavior of $E_{\rm avg}$ \eqref{eq:Eavg_diff} with $tol_{\rm eim}$
($n_{\rm test}=50$).
(b):
behavior of $Q_{\rm a}$  in \eqref{eq:affine_diff_eq_EIM} with $tol_{\rm eim}$.
}
 \label{fig:registration_laplace_2}
\end{figure} 

In the supplementary material (cf. section \ref{sec:RBF_laplace}), we further provide a numerical comparison of our approach with a representative mapping procedure for a given value of $\mu$.

\section{Summary and discussion}
\label{sec:conclusions}
 In this work, we proposed a general (i.e., independent of the underlying PDE model) registration procedure, and we applied it to data compression and geometry reduction:
in data compression, our registration procedure is used in combination with a linear compression method to devise a \emph{Lagrangian nonlinear compression method}; 
in geometry reduction, our registration procedure is used to build a parametric mapping from a reference domain to a family of parameterized domains $\{ \Omega_{\mu} \}_{\mu \in \mathcal{P}}$.
Several numerical results empirically motivate our proposal.
Although the examples considered are rather academic, we believe that our results demonstrate the applicability 
ìof our approach to a broad class of relevant problems in science and engineering. 

In the future, we wish to extend the approach in several directions.
First, in this work, we chose the reference field $\bar{u}$ (cf. section \ref{sec:data_compression}) 
and the reference domain $\Omega$ (cf. section \ref{sec:geometry_reduction})
that enters in the optimization statement \emph{a priori}; in the future, we wish to develop automatic procedures to adaptively choose $\bar{u}$ and $\Omega$.
Second, in order to tackle problems with more complex parametric behaviors, we wish to develop partitioning techniques that leverage the use of multiple references,
{and we also wish to couple the approach with domain decomposition techniques.}
Third, as discussed in section \ref{sec:registration}, a major limitation of our approach is the need for several offline simulations to build the mapping $\boldsymbol{\Phi} $. To address this issue, we wish to exploit recent advances in multi-fidelity approaches to reduce the offline computational burden. Alternatively, we wish to assess performance of projection methods to simultaneously learn mapping and solution coefficients during the online stage.
{Fourth, the registration procedure for data compression has been developed for $\Omega=(0,1)^2$: we aim to combine the proximity measures \eqref{eq:L2obj} and \eqref{eq:mapping_obj} to deal with more complex domains. Furthermore, we aim to extend Proposition \ref{th:extension_sad} to assess the sensitivity of the registration procedure to perturbations.
}
Finally, we wish to study the approximation properties of  Lagrangian  methods based on problem-dependent mappings for a range of parametric problems. Examples in section \ref{sec:NM_widths_examples}  illustrate the effectivity of Lagrangian mappings for elementary  one- and two-dimensional problems; in the future, we aim to investigate whether it is possible to prove approximation results for a broader class of problems.

\appendix

\section*{Acknowledgments}
The author thanks
Prof. Angelo Iollo (IMB, Inria),   Prof. Pierre Mounoud (IMB), and
Prof. Anthony Patera (MIT)  for fruitful discussions.
 
\appendix

\section{Proof of Proposition \ref{th:mapping_general}}
\label{sec:proof_technical}
 
 \subsection{Preliminaries}
\label{sec:preliminary_top}

Given the set $U \subset \mathbb{R}^d$, we denote by $\mathcal{O}_{U}$ the induced (or subspace) topology on $U$,
$
\mathcal{O}_{U} := \{A \cap U: \, A \, {\rm is \, open \, in} \;  \mathbb{R}^d \}.
$
It is possible to show that the ordered pair
$\mathcal{T}_{U} = (U, \mathcal{O}_{U})$ is a topological space. We say that $B$ is open in $U$ if $B \in \mathcal{O}_{U}$; similarly, we say that $B$ is closed in $U$ if the complement of $B$ in $U$,
$B^{\rm c}:= U \setminus B$,
belongs to  $\mathcal{O}_{U}$.
We further say that $\mathcal{T}_{U}$ is connected if it cannot be represented as the union of two more disjoint non-empty open subsets; we say that $\mathcal{T}_{U}$ is path-connected if there exists a path joining any two points in $U$; finally, we say that 
$\mathcal{T}_{U}$ is simply-connected if $\mathcal{T}_{U}$ is path-connected and every path between two points can be continuously transformed into any other such path while preserving the endpoints.  Next three results are key for our discussion.  Theorem \ref{th:connected_sets} is a standard result in Topology that can be found in \cite{mendelson1990introduction}, while  Theorem \ref{th:hadamard} is known as 
Hadamard's implicit function theorem and is proven in  \cite[Chapter 6]{krantz2012implicit}.
On the other hand, we report the proof of  Lemma \ref{th:trivial_lemma}.

\begin{theorem}
\label{th:connected_sets}
A set $A$ in a topological space $\mathcal{T}_{U}$ is open and closed if and only if $\partial A = \emptyset$.
Furthermore, the topological space $\mathcal{T}_{U}$ is connected if and only if the only open and closed sets are the empty set and $U$.  
\end{theorem}

\begin{theorem}
\label{th:hadamard}
Let $M_1,M_2$ be smooth and connected $d$-dimensional manifolds. Suppose $\boldsymbol{\Phi}: M_1 \to M_2$ is a $C^1$ function such that
(i) $\boldsymbol{\Phi}$ is proper (i.e., for any compact set $K \subset M_2$,  $\boldsymbol{\Phi}^{-1}(K)$ is compact in $M_1$),  (ii) the Jacobian matrix of $\boldsymbol{\Phi}$ is everywhere invertible, and (iii) $M_2$ is simply connected. Then, $\boldsymbol{\Phi}$ is a homeomorphism (hence globally bijective).
\end{theorem}

\begin{lemma}
\label{th:trivial_lemma}
Let 
$A,B,U$ be connected,  open and Lipschitz domains
 such that $\partial A = \partial B$, and
 $\overline{A \cup B}$ is strictly contained in $\overline{U}$. Then, $A=B$.
\end{lemma}
\begin{proof}
Let $C= A \cap B$. Clearly, $C$ is open in $A$ (since it is the intersection of open sets); furthermore, $C$ is closed in $A$ since $\partial A  = \partial B$. It thus follows 
from Theorem \ref{th:connected_sets} that $C$ is either the empty set or $A$.

By contradiction, suppose that $C= \emptyset$, and let us define $D=A \cup B$. Since $A,B$ are open, we have
$\partial (\overline{A \cup B}) \subset \partial A \cup \partial B = \partial A$. Let $\mathbf{x} \in \partial A$: since (i) $\partial A = \partial B$, (ii) $A \cap B= \emptyset$, and (iii) the boundary of $A$ is smooth, there exists $\epsilon>0$ such that 
$\mathcal{B}_{\epsilon}(\mathbf{x}) \subset \overline{A \cup B}$; therefore,  $\mathbf{x} \notin \partial (\overline{A \cup B})$, which implies that $\partial (\overline{A \cup B})  = \emptyset$ and thus (exploiting Theorem \ref{th:connected_sets}) $\overline{A \cup B} = U$. The latter contradicts the hypothesis that $\overline{A \cup B}$ is  a proper subset of  $U$: we can thus conclude that $C= A$. 

Exploiting the same argument, we can prove that $C=B$. Thesis follows by applying the transitive property: $A=C, B=C \Rightarrow A=B$. 
\end{proof}

\subsection{Proof}
\label{sec:proof}

We split the proof in four parts.
\medskip

\textbf{1.  $\mathbf{\boldsymbol{\Phi}(\overline{{U}}) \subseteq  \overline{V}}$,  $\mathbf{U}=\mathbf{V}=
\mathbf{\widehat{\Omega}}$.}
Given   $U = 
\widehat{\Omega} = 
\{ \mathbf{x} \in \mathbb{R}^d: \, f(\mathbf{x}) < 0  \} $ where $f$ is convex,
we define $g(\mathbf{X})=f(\boldsymbol{\Phi}(\mathbf{X}))$. 
We denote by $\mathbf{X}^{\star}$ a global maximum of $g$ in $\overline{U}$, and we define $\mathbf{x}^{\star}= \boldsymbol{\Phi}(\mathbf{X}^{\star})$. Since $\boldsymbol{\Phi}$ is locally invertible at $\mathbf{X}^{\star}$, if $\mathbf{X}^{\star}$ belongs to the interior of $U$,
we must have that $\mathbf{x}^{\star}$ is a local maximum of $f$: this is not possible due to the fact that $f$ is convex. As a result, we must have $\mathbf{X}^{\star} \in \partial U$: recalling (iii), we then find
$$
\max_{\mathbf{X} \in \overline{U}} \;
f \left(  \boldsymbol{\Phi}(\mathbf{X}) \right)
\, = \, 
\max_{\mathbf{X} \in \partial  U } \; f \left( \boldsymbol{\Phi}(\mathbf{X})  \right) 
 \, \overset{\rm (iii)}{\leq} \, 
\max_{\mathbf{x} \in \partial  U } \; f \left( \mathbf{x} \right) = 0,
$$
which implies $\boldsymbol{\Phi}(\mathbf{X}) \in \overline{U}$ for all $\mathbf{X} \in \overline{U}$.
\medskip 

  \textbf{2.  $\mathbf{\boldsymbol{\Phi}(\overline{U}) =  \overline{V}}$,  $\mathbf{U=V = \widehat{\Omega}}$.}
Recalling Theorem \ref{th:connected_sets}, we shall simply prove that  $\boldsymbol{\Phi}( 	\overline{U}  )$ is  open and closed in   $\overline{U}$. 
  
  $\boldsymbol{\Phi}(  \overline{U}   )$ is closed in 
$\overline{U} $: since $\boldsymbol{\Phi}$ is continuous, $\boldsymbol{\Phi}(  \overline{U}   )$ is closed in $\mathbb{R}^d$; since $\boldsymbol{\Phi}(\overline{U}) \subseteq   \overline{U}$, we then find that 
 $\boldsymbol{\Phi}(\overline{U})$ is closed with respect to the topology of   $\mathcal{T}_{\overline{U}}$.
 
  $\boldsymbol{\Phi}(  \overline{U}   )$ is  open in   $\overline{U}$.
  To show this, it suffices to prove that, for any $\mathbf{x} \in \boldsymbol{\Phi}(  \overline{U}   )$, there exists $B \in \mathcal{O}_{\overline{U}}$ such that $B \subset  \boldsymbol{\Phi}(  \overline{U})$ and $\mathbf{x} \in B$. 
 Given $\mathbf{x} \in  \boldsymbol{\Phi}(\overline{U})$, we denote by $\mathbf{X} \in \overline{U}$ a point such that $\mathbf{x} = \boldsymbol{\Phi}(\mathbf{X})$. 
  If $\mathbf{X}$ belongs to the interior of $U$, the proof is trivial; for this reason, we focus below on the case in which $\mathbf{X} \in \partial U$ (and thus $\mathbf{x} \in \partial U$). Exploiting the local inverse-function theorem, there exists $\bar{\eta}>0$ such that for all $\eta \leq \bar{\eta}$ $\boldsymbol{\Phi}: \mathcal{B}_{\eta}(\mathbf{X}) \to A_{\mathbf{x},\eta}$ is an homeomorphism and $A_{\mathbf{x},\eta}$ is an open set of $\mathbb{R}^d$ containing $\{ \mathbf{x} \}$.   Provided that
\begin{equation}
\label{eq:tricky_condition}
 \boldsymbol{\Phi} \left( \mathcal{B}_{\eta}(\mathbf{X}) \cap \partial U \right)
=
A_{\mathbf{x}, \eta} \cap \partial U,
\end{equation} 
 we find that
 $ \boldsymbol{\Phi} \left( \mathcal{B}_{\eta}(\mathbf{X}) \cap \overline{U}    \right)
= A_{\mathbf{x}, \eta} \cap \overline{U}$, 
  $ \boldsymbol{\Phi} \left( \mathcal{B}_{\eta}(\mathbf{X}) \setminus  \overline{U}    \right)
= A_{\mathbf{x}, \eta}   \setminus \overline{U}$.
 This implies that the set 
   $B:= A_{\mathbf{x},\eta} \cap \overline{U}$
is contained in $\boldsymbol{\Phi} (  \overline{U}  )$.
Since $B$ is open in $\overline{U}$, we obtain    
the desired result.
 
 It remains to prove \eqref{eq:tricky_condition}.  Since  $U$ is Lipschitz,  for sufficiently small values of $\eta$, we have that
$$
\partial \left( \mathcal{B}_{\eta}(\mathbf{X}) \cap \partial U  \right)
= \partial  \mathcal{B}_{\eta}(\mathbf{X}) \cap \partial U,
\quad
 \partial \left( A_{\mathbf{x},\eta} \cap \partial U  \right)
= \partial  A_{\mathbf{x},\eta} \cap \partial U.
$$
Furthermore, these sets 
 are the collection of two distinct points for $U \subset \mathbb{R}^2$, and
closed one-dimensional curves for $U \subset \mathbb{R}^3$.
Since $\boldsymbol{\Phi}(\partial U) \subset \partial U$,
we must have   $\boldsymbol{\Phi} \left(  \partial   \mathcal{B}_{\eta}(\mathbf{X}) \cap \partial U \right) \subset   \partial  A_{\mathbf{x}, \eta} \cap \partial U$; then, since 
$\boldsymbol{\Phi}$ is a local homeomorphism, we conclude  that 
 $  
 \boldsymbol{\Phi} \left(  \partial   \mathcal{B}_{\eta}(\mathbf{X}) \cap \partial U \right)
=
 \partial  A_{\mathbf{x}, \eta} \cap \partial U$.

We define $A_1=  \boldsymbol{\Phi} \left(     \mathcal{B}_{\eta}(\mathbf{X}) \cap  \partial U \right)$  and $A_2 = A_{\mathbf{x}, \eta} \cap \partial U$.
Exploiting once again the fact that $\boldsymbol{\Phi}$ is a local homeomorphism and the fact that $  
 \boldsymbol{\Phi} \left(  \partial   \mathcal{B}_{\eta}(\mathbf{X}) \cap \partial U \right)
=
 \partial  A_{\mathbf{x}, \eta} \cap \partial U$, we find
$$
\partial A_1 = \boldsymbol{\Phi} \left(
\partial   \mathcal{B}_{\eta}(\mathbf{X}) \cap \partial U
\right) = 
 \partial  A_{\mathbf{x}, \eta} \cap \partial U = \partial A_2.
$$
 Since $A_1,A_2$ are connected, open in $\partial U$, Lipschitz, bounded sets with the same boundary, exploiting Lemma \ref{th:trivial_lemma}, we find $A_1=A_2$, which is \eqref{eq:tricky_condition}.
  \medskip

 \textbf{3.  $\mathbf{\boldsymbol{\Phi}(\overline{U}) =  \overline{V}}$.}
 Recalling the definitions of $\boldsymbol{\Upsilon}$ and $\boldsymbol{\Lambda}$, we find that
 $\widehat{\boldsymbol{\Phi}}  = \boldsymbol{\Lambda}^{-1} \circ \left( \boldsymbol{\Phi} \circ \boldsymbol{\Upsilon}  \right)$
  satisfies the hypotheses (i)-(ii)-(iii) of Proposition \ref{th:mapping_general} and thus 
  $\widehat{\boldsymbol{\Phi}}(\overline{ \widehat{\Omega} })
=\overline{ \widehat{\Omega} }$. Since $\boldsymbol{\Lambda}$ and $\boldsymbol{\Upsilon}^{-1}$ are diffeomorphisms, we obtain
$$
\boldsymbol{\Lambda}^{-1} \left(
\boldsymbol{\Phi}
 \left(
\boldsymbol{\Upsilon} (  \overline{ \widehat{\Omega} }   )
\right)
\right)
=
 \overline{ \widehat{\Omega} }  
\, \Rightarrow \,
\boldsymbol{\Phi}
 \left(
\boldsymbol{\Upsilon} (  \overline{ \widehat{\Omega} }   )
\right)
=
\overline{V} \,
\Rightarrow \,
\boldsymbol{\Phi}( \overline{U}   )
=
\overline{V},
$$
 which is the thesis.
\medskip

 \textbf{4.  $\mathbf{\boldsymbol{\Phi}: \overline{U} \to   \overline{V}}$ is bijective.}
From the third  part of the proof, we have that $ \boldsymbol{\Phi} (\overline{U}) = \overline{V}$. Since $U$ is simply connected, proof follows from Theorem \ref{th:hadamard}.   Note that the condition that $\boldsymbol{\Phi}$ is proper follows directly from the fact that $\boldsymbol{\Phi}$ is continuous and $U$ is bounded.

\section{Estimates of Kolmorogov $N$-$M$ widths}
\label{sec:NM_widths_examples}

\subsection{Boundary layers}
\label{remark:BL_example}

 
We consider the one-dimensional problem:
\begin{equation}
\label{eq:1D_BL}
\left\{
\begin{array}{ll}
-\partial_{x x} \, u_{\mu} \, + \, \mu^2 \, u_{\mu} = 0 & {\rm in} \,  \Omega_{\rm 1D}=(0,1) \\
u_{\mu}(0) =1, \quad \partial_x u_{\mu}(1) = 0 , &\\
\end{array}
\right.
\end{equation}
where $\mu \in \mathcal{P}:= [ \mu_{\rm min}, \mu_{\rm max} =  \epsilon^{-2} \mu_{\rm min}  ]$. 
The solution to \eqref{eq:1D_BL} is given by
$$
u_{\mu}(x) =   u_{\mu}^{(1)}(x) + u_{\mu}^{(2)}(x),
\quad
{\rm with} \; \; 
u_{\mu}^{(1)}(x) = \frac{e^{-\mu x}}{1 + e^{-2 \mu}}, \; \; \;
u_{\mu}^{(2)}(x) = \frac{e^{\mu(x-2)}}{1 + e^{-2\mu}}.
$$
We introduce the parametrically-affine ($M=1$) mapping
$$
{\Phi}_{\mu}( {X}) \, = \,
 X \, + \, c_{\mu}   \, \left(
{X} \mathbbm{1}_{[0,\delta)}( {X}) 
\, + \,
\frac{\delta}{\delta-1} \, (X - 1)  \mathbf{1}_{[\delta,1]}(X)
\right),
\qquad
c_{\mu}:= \frac{\bar{\mu} - \mu}{\mu},
$$
which  is a bijection in $[0,1]$ for   $0 < \delta  < \min \left\{  1,  \frac{\mu}{\bar{\mu}}    \right\}$,
 and we set
\begin{equation}
\label{eq:choice_mubar_BL}
\bar{\mu} = \sqrt{\mu_{\rm min} \, \mu_{\rm max}},
\end{equation} 
Note that the choice of $\bar{\mu}$ in \eqref{eq:choice_mubar_BL} 
minimizes the maximum value attained by either the Jacobian of the mapping or by the Jacobian of the  inverse $ {\Phi}_{\mu}^{-1}$ over $\mathcal{P}$ in $X=0$,
$\max_{\mu \in \mathcal{P}} \, 
\max \left\{
\partial_X  {\Phi}_{\mu}(0), \; 
\partial_x  {\Phi}_{\mu}^{-1}(0)
\right\}.$
Then, we find
$$
\begin{array}{l}
\displaystyle{
\|  u_{\bar{\mu}} \, - \, u_{\mu} \circ  {\Phi}_{\mu}
\|_{L^2(\Omega)}
 \leq \; \; 
 \|  u_{\bar{\mu}}^{(1)} \, - \, u_{\mu}^{(1)}  \circ  {\Phi}_{\mu}
\|_{L^2(\Omega)}  \,  +  
\|  u_{\bar{\mu}}^{(2)}  \, - \, u_{\mu}^{(2)}  \circ  {\Phi}_{\mu}
\|_{L^2(\Omega)}  }
\\[3mm]
\leq  
 \displaystyle{
\sqrt{ \int_{\delta}^1 \, \left(
e^{-\bar{\mu} X } \, - \,  e^{- \mu \Phi_{\mu}(X)} \right)^2 \, dX }   +
 \frac{e^{- \mu_{\rm min}}}{1+e^{-2\mu_{\rm min}}} \leq
 \sqrt{1-\delta} \, e^{-\bar{\mu} \delta } 
 +
 \frac{e^{- \mu_{\rm min}}}{1+e^{-2\mu_{\rm min}}}.
}
\\
\end{array}
$$
Note that in the second inequality we used the fact that $e^{-\bar{\mu} X}$ and $e^{- \mu \Phi_{\mu}(X)}$ are monotonic decreasing in $(\delta, 1)$ and 
$e^{-\bar{\mu} \delta} = e^{- \mu \Phi_{\mu}(\delta)} = e^{-\bar{\mu} \delta}$.
We observe that the right-hand side is monotonic decreasing in $\delta$. Given $\epsilon>0$, we thus choose $\delta = {\rm max} \left\{
\delta': \, \partial_X \Phi_{\mu}(1) \in [\epsilon, 1/\epsilon], \; \forall \, \mu \in \mathcal{P} \right\}$: by tedious but straightforward calculations, we obtain $\delta=  \frac{\epsilon}{1+\epsilon}$. In conclusion, we obtain
$$
\begin{array}{rl}
\displaystyle{
\|  u_{\bar{\mu}} \, - \, u_{\mu} \circ  {\Phi}_{\mu}
\|_{L^2(\Omega)} \leq }
&
 \displaystyle{
 \sqrt{ \int_{\delta}^1 \, \left(
e^{-\bar{\mu} X } \, - \,  e^{- \mu \Phi_{\mu}(X)} \right)^2 \, dX } 
+
 \frac{e^{-\mu_{\rm min}}}{1+e^{-2\mu_{\rm min}}}
 } \\[3mm]
  \leq &
 \displaystyle{
  \sqrt{1-\delta} \, e^{-\bar{\mu} \delta } +
 \frac{e^{-\mu_{\rm min}}}{1+e^{-2\mu_{\rm min}}}
 }
\\
\end{array}
$$
and then, recalling the definition of $\delta$ and $\bar{\mu}$,
$$
d_{N=1,M=1,\epsilon}(\mathcal{M}_{\rm u}; L^2(\Omega)) \leq
 \frac{1}{\sqrt{
1 + \epsilon}}
\, {\rm exp} \left( 
- \frac{\mu_{\rm min}}{1+\epsilon}
 \right) \, + \,
  \frac{e^{-\mu_{\rm min}}}{1+e^{-2\mu_{\rm min}}},
$$
which proves \eqref{eq:NM_boundary_layer}.

\subsection{Shock waves}
\label{remark:hyp_example}
We consider the manifold
(\cite[section 5.1]{ohlberger2015reduced}, \cite[Example 2.5]{taddei2015reduced})
$$
\mathcal{M}:=  \{
u(\cdot ; t) = {\rm sign}(\cdot- t): \,
\; \;
t \in  [1/3, 2/3] \} \subset L^2(\Omega_{\rm 1D}=(0,1)),
$$
which is associated with the  transport problem
$$
\left\{
\begin{array}{ll}
\partial_t u  \,  + \, \partial_x \, u \, = \, 0 & x \in \Omega_{\rm 1 D}, \; t \in (1/3,2/3), \\
z(x)|_{t=1/3}= {\rm sign}\left( x - \frac{1}{3} \right), \quad
z(0,t) = -1& x \in \Omega_{\rm 1 D}, \; t \in (1/3,2/3), \\
\end{array}
\right.
$$
with time interpreted as parameter. It is possible to show (see \cite[section 5.1]{ohlberger2015reduced}) that the Kolmogorov $N$-width associated with $\mathcal{M}$ satisfies
$d_N(\mathcal{M}; L^2(\Omega))= \mathcal{O}\left(  \frac{1}{\sqrt{N}} \right)$. On the other hand, if we consider the parametrically-affine ($M=1$) mapping 
\begin{equation}
\label{eq:mapping_hyp}
{\Phi}(X; t) \, := \,
X \, + \,
\left(  t  -  \frac{1}{2}  \right) \left(
1  \, - \, |2 X  - 1| \right),
\end{equation}
we find $\tilde{u}(X,t) = {\rm sign}(2X - 1 )$, which is parameter-independent. This implies that 
$$
d_{N,M,\epsilon}(\mathcal{M}_{\rm u}) = 0,
\quad
\forall \, N,M \geq 1,
\quad
\epsilon \leq \frac{2}{3}.
$$ 

This example suggests that mapping procedures  are well-suited to tackle problems with travelling fronts.
On the other hand, we observe that ${\Phi}$ in \eqref{eq:mapping_hyp} is not well-defined  for $t \geq 1$: the reason is that the  jump discontinuity exits the domain. This shows that effective applications of our approach to general parametric problems might require the 
partition of the   time/parameter domain in several subdomains.

\subsection{A two-dimensional problem}
\label{remark:2D_problem_NM}

{We consider the parametric field
\begin{equation}
\label{eq:boring_example}
u_{\mu}(\mathbf{x}) \, = \, 
\left\{
\begin{array}{ll}
0 & {\rm if} \; x_2 < f_{\mu}(x_1) \\[3mm] 
1 & {\rm if} \; x_2 \geq  f_{\mu}(x_1) \\
\end{array}
\right.
\end{equation}
where $\mathbf{x} = [x_1,x_2] \in \Omega=(0,1)^2$,
 $f_{\mu} \in {\rm Lip}([0,1])$ and 
$f_{\mu}([0,1]) \subset [\delta, 1 - \delta]$, for some $\delta>0$ and for all $\mu \in \mathcal{P} \subset \mathbb{R}^P$. We define the manifold $\mathcal{M}_{\rm f} = \{ f_{\mu}: \mu \in \mathcal{P} \} \subset {\rm Lip}([0,1])$,  and we define  the  $M$-term approximation
  $\widehat{f}$  of $f$,
$\widehat{f}_{\mu}(x)=\sum_{m=1}^M \, (\widehat{\mathbf{a}}_{\mu})_m \, \varphi_m(x)$; then, we define the mapping
$$
\widehat{\boldsymbol{\Phi}}_{\mu}(\mathbf{X}) = \mathbf{X} \, + \, \left(\widehat{f}_{\mu}(X_1)  - \frac{1}{2} \right)
\left[
\begin{array}{l}
0 \\
1 - |2 X_2 - 1| \\
\end{array}
\right]
$$
Exploiting the definition of $\widehat{f}$, we find that, for all $\mu \in \mathcal{P}$,
\begin{equation}
\label{eq:special_mappings}
\widehat{\boldsymbol{\Phi}}_{\mu}(\mathbf{X}) = \mathbf{X} + \boldsymbol{\varphi}(\mathbf{X}), \qquad
\boldsymbol{\varphi} \in 
\mathcal{Y}_M: =
{\rm span} \left\{
\boldsymbol{\varphi}_m 
\right\}_{m=1}^M,
\end{equation}
where $\boldsymbol{\varphi}_m(\mathbf{X}) \, := \, 
(\varphi_m(X_1) - 1/2) \, (1  - |2 X_2  - 1|) \mathbf{e}_2
$ for $m=1,\ldots,M$.
Furthermore, we find
$$
{\rm det}\left(
\widehat{\nabla} \boldsymbol{\Phi} (\mathbf{X})
\right) \, = \, \left\{
\begin{array}{ll}
2 \widehat{f}_{\mu}(X_1) & {\rm if} \, X_2 < \frac{1}{2} \\[2mm]
2 (1 - \widehat{f}_{\mu}(X_1) ) & {\rm if} \, X_2 \geq \frac{1}{2} \\ 
\end{array}
\right.
$$
Recalling Corollary \ref{th:application_unit_square}, provided that
$\widehat{f}_{\mu}([0,1]) \subset [1/2 \epsilon, 1 - 1/2 \epsilon]$ for some $0 <\epsilon < 2 \delta$ and all $\mu \in \mathcal{P}$,
we obtain that 
 $\widehat{\boldsymbol{\Phi}}_{\mu} \in \mathcal{Y}_M^{\rm bis, \epsilon}$ (cf.  \eqref{eq:kolmogorov_NM_width_a})
for all $\mu \in \mathcal{P}$.

By tedious but straightforward calculations, we find that
$\widetilde{u}_{\mu} = u_{\mu} \circ \boldsymbol{\Phi}_{\mu}$ satisfies
$$
\widetilde{u}_{\mu}(\mathbf{X}) \, = \, 
\left\{
\begin{array}{ll}
0 & {\rm if} \, X_2 < \Lambda_{\mu}(X_1), \\[2mm]
1 & {\rm if} \, X_2 \geq \Lambda_{\mu}(X_1); \\
\end{array}
\right.
\; \; 
\Lambda_{\mu}(x) \, = \, 
\left\{
\begin{array}{ll}
\frac{1}{2} \, + \, \frac{f_{\mu}(x) \, - \, \widehat{f}_{\mu}(x)}{2(1 - \widehat{f}_{\mu}(x))} &
{\rm if} \, \widehat{f}_{\mu}(x) < f_{\mu}(x) \\[2mm]
\frac{1}{2} \, + \, \frac{f_{\mu}(x) \, - \, \widehat{f}_{\mu}(x)}{2  \widehat{f}_{\mu}(x) } &
{\rm if} \, \widehat{f}_{\mu}(x) >  f_{\mu}(x) \\
\end{array}
\right.
$$
Note that, recalling that
$f_{\mu}([0,1]) \subset [\delta, 1 - \delta]$ and
$\widehat{f}_{\mu}([0,1]) \subset [1/2 \epsilon, 1 - 1/2 \epsilon]$ uniformly in $\mu$,  
\begin{equation}
\label{eq:estimate1_NM_width}
\Lambda_{\rm max, \mu} - \Lambda_{\rm min, \mu} \lesssim
\| \widehat{f}_{\mu}  - f_{\mu}  \|_{L^{\infty}(0,1)},
\quad
\| \Lambda_{\mu}'   \|_{L^{\infty}(0,1)}  \lesssim
\| \widehat{f}_{\mu}  - f_{\mu}  \|_{{\rm Lip}([0,1])}, 
\end{equation}
where
$\Lambda_{\rm max, \mu}= \max_{ X_1 \in [0,1]} \, \Lambda_{\mu}(X_1)$ and
$ \Lambda_{\rm min, \mu}=\min_{ X_1 \in [0,1]} \, \Lambda_{\mu}(X_1)$.

To estimate the $N$-width of $\widetilde{\mathcal{M}}_{\rm u}=
\{  \widetilde{u}_{\mu} : \, \mu \in \mathcal{P} \}$, we partition
$[0,1]$   and
$[\min_{\mu} \Lambda_{\mu,\rm min}, \, 
\max_{\mu} \Lambda_{\mu,\rm max}   ]$  in $n-1$ uniform intervals,
$I_i^x= (\bar{x}_i, \bar{x}_{i+1})$
$I_j^y= (\bar{y}_j, \bar{y}_{j+1})$ $i,j=1,\ldots,n$.
Then, we introduce the $N=n^2$ functions
$$
\widetilde{\zeta}_{i,j}(\mathbf{X}) \, = \, \left\{
\begin{array}{ll}
1 & {\rm if} \, X_1 \in I_i^x, \; X_2 > \bar{y}_j \\
0 &{\rm otherwise} \\
\end{array}
\right.
\; \; i,j = 1, \ldots, n.
$$
We observe that
\begin{equation}
\label{eq:estimate2_NM_2D}
\begin{array}{l}
\displaystyle{
\min_{\boldsymbol{\alpha} \in \mathbb{R}^{n,n}} \, \| \tilde{u}_{\mu} \,  - \sum_{i,j} \alpha_{i,j} \, \widetilde{\zeta}_{i,j}  \|_{L^1(\Omega)}  \leq 
\sum_{i=1}^n \, \min_{j=1,\ldots, n} \, 
\int_{I_i^x} \, |y_j - \Lambda_{\mu}(x) | \, dx
 }
\\[2mm]
\displaystyle{ 
\leq 
\left(
\sum_{i=1}^n \, 
\frac{1}{n} \, 
\min_{j=1,\ldots, n} \,  |y_j - \Lambda_{\mu}(\bar{x}_{i+1/2}) | 
\, + \, 
\int_{I_i^x} \, 
\left|
\Lambda_{\mu}(x) \, - \, \Lambda_{\mu}(\bar{x}_{i+1/2}
\right| \, dx
\right) 
 }
\\[2mm]
\displaystyle{ 
\leq 
\frac{ \Lambda_{\rm max, \mu} - \Lambda_{\rm min, \mu} + 
\| \Lambda_{\mu}' \|_{L^{\infty}(0,1)}   }{\sqrt{N}}.
 }
\\
\end{array}
\end{equation}
Since \eqref{eq:estimate2_NM_2D} holds for any choice of $\boldsymbol{\Phi}$ satisfying \eqref{eq:special_mappings}, provided that $M$ is large enough so that there exists an $M$-term approximation satisfying $\sup_{\mu \in \mathcal{P}} \, \|   \widehat{f}_{\mu} - f_{\mu} \|_{{\rm Lip}([0,1])} $
$\leq \delta - 1/2 \epsilon$, we find that
$$
\begin{array}{rl}
\displaystyle{
d_{N,M, \epsilon}\left(   
\mathcal{M}_{\rm u} , L^1(\Omega)
\right)
\lesssim
}
&
\displaystyle{
\frac{1}{\sqrt{N}} \,  
\inf_{
\substack{\mathcal{F}_M \subset {\rm Lip}([0,1]), \;  \\ 
{\rm dim}(\mathcal{F}_M) = M \\
}}
\sup_{\mu \in \mathcal{P}} \, \inf_{\widehat{f} \in \mathcal{F}_M} \, \| f_{\mu} - \widehat{f}_{\mu}  \|_{{\rm Lip}([0,1])}
}
\\[3mm]
=
&
\displaystyle{
\frac{1}{\sqrt{N}}  \, d_N \left( 
\mathcal{M}_{\rm f},   \, {\rm Lip}([0,1]) \right),
}
\\
\end{array}
$$
which is \eqref{eq:NM_2D}.
}

\section{DG discretization of  \eqref{eq:model_problem_strong}}
\label{sec:DG_hyp_app}
 We denote by 
$\mathbb{P}^{\kappa}(\widehat{\texttt{D}})$ the space of polynomials of degree at most $\kappa$ on the triangle
$\widehat{\texttt{D}}$ with vertices $(0,0)-(1,0)-(0,1)$; then,  we define the broken DG space 
$$
\mathcal{X}:=  \left\{
v \in L^2(\Omega): \,
v|_{  {\texttt{D}}^k   } \, = \,
\widehat{v} \circ \boldsymbol{\Psi}_k^{\rm fe}, \; \;
\widehat{v} \in   \mathbb{P}^{\kappa}(\widehat{\texttt{D}}),
 \; k=1,\ldots, n_{\rm el}
\right\},
$$
where  $\{ \texttt{D}^k  \}_{k=1}^{n_{\rm el}}$ are the elements of the mesh and 
$\boldsymbol{\Psi}_k^{\rm fe}: \texttt{D}^k \to \texttt{D}$ are the local FE mappings. 
For each edge $e$ of the mesh, we define the positive (resp. negative) normal $\mathbf{n}^+$
(resp. $\mathbf{n}^-$); given $w \in \mathcal{X}$ and the mesh edge $e$, we define the positive and negative limits $w^+, w^-$ and the edge average and jump
$$
w^{\pm}(\mathbf{x}) =
\lim_{\epsilon \to 0^+ } \, w(\mathbf{x}  - \epsilon \mathbf{n}^{\pm}(\mathbf{x})),
\quad
\{  w \}:= \frac{w^+ + w^-}{2},
\quad
\mathbf{J} w:= \mathbf{n}^+ w^+ \, + \, \mathbf{n}^- w^-,
$$
for all $\mathbf{x} \in e$
If $\mathbf{x} \in e \subset \partial \Omega$, we set $\{  w \}:=  w$ and $\mathbf{J} w:= \mathbf{n} w$. 

Then, we can introduce the high-fidelity DG discretization of \eqref{eq:model_problem_strong}: find $z_{\mu} \in 
\mathcal{X}$ such that
\begin{subequations}
\label{eq:weak_formulation_hyp}
\begin{equation}
\label{eq:weak_formulation_hyp_a}
\mathcal{G}_{\mu}(z_{\mu}, v):= \mathcal{A}_{\mu}(z_{\mu}, v) - \mathcal{F}_{\mu}( v) = 0,
\quad
 \forall \, v \in \mathcal{X},
\end{equation}
where
\begin{equation}
\label{eq:weak_formulation_hyp_b}
\left\{
\begin{array}{l}
\displaystyle{
\mathcal{A}_{\mu}(w,v) \, = \, 
\sum_{k=1}^{n_{\rm el}} \, 
\int_{\texttt{D}^k} \, w \, \left( \sigma_{\mu} v - \mathbf{c}_{\mu} \cdot \nabla v \right) \, dx \, + \, 
\int_{\partial \texttt{D}^k} \, 
\mathcal{H}(w, \mathbf{n}) \, v  \, dx 
},
\\[3mm]
\displaystyle{
\mathcal{F}_{\mu}(w,v) \, = \, 
\sum_{k=1}^{n_{\rm el}} \, 
\int_{\texttt{D}^k} \, f_{\mu} \,  v   \, dx \, - \, 
\int_{\partial \texttt{D}^k} \, 
f_{\mu}^{\rm ed}\, v  \, dx 
},
\\
\end{array}
\right.
\end{equation}
Here,  $f_{\mu}^{\rm ed}(\mathbf{x}) = \delta_{\rm in,\mu} \mathbf{c}_{\mu} \cdot \mathbf{n} \, z_{\rm D, \mu}$,
where $\delta_{\rm in,\mu} (\mathbf{x})=1$ if $\mathbf{x} \in \Gamma_{\rm in, \mu}$ and 
$\delta_{\rm in,\mu} (\mathbf{x})=0$ otherwise, while the flux $\mathcal{H}$ is given by
\begin{equation}
\mathcal{H}(w, \mathbf{n}) :=
\left\{
\begin{array}{ll}
\mathbf{c}_{\mu}\cdot \mathbf{n} \{ w  \} \,  + \, \frac{1}{2} \tau_{\mu} \mathbf{n} \cdot (\mathbf{J} w)
& {\rm on} \, \partial \texttt{D}^k \setminus \partial \Omega \\[3mm]
\mathbf{c}_{\mu}\cdot \mathbf{n}   w   \, \delta_{\rm in,\mu} 
& {\rm on} \, \partial \texttt{D}^k \cap  \partial \Omega \\
\end{array}
\right.
\end{equation}
with $\tau_{\mu}= |\mathbf{c}_{\mu} \cdot \mathbf{n} |$. 
\end{subequations}

Then, we introduce
\begin{equation}
\label{eq:another_formulation_a}
\begin{array}{ll}
\displaystyle{
\boldsymbol{\Upsilon}_{\mu}^{\rm el} :=
\left[
\begin{array}{c}
\sigma_{\mu} \\
\mathbf{c}_{\mu}\\
\end{array}
\right],
}
&
\displaystyle{
\boldsymbol{\Upsilon}_{\mu}^{\rm ed} :=
(1 - \delta_{\rm in, \mu})
\left[
\begin{array}{c}
\mathbf{c}_{\mu} \cdot \mathbf{n} \\
0\\
\end{array}
\right] \, + \,
\frac{1 - \delta_{\partial \Omega}}{2}
\left[
\begin{array}{c}
0 \\
\tau_{\mu}\\
\end{array}
\right],
}
\\[3mm]
\displaystyle{
A^{\rm el}(w,v) \, = \,
\left[
\begin{array}{c}
 w \, v  \\
w \, \widehat{\nabla} v\\
\end{array}
\right],
}
&
\displaystyle{
A^{\rm ed}(w,v) \, = \,
\left[
\begin{array}{c}
 \{  w \}  \, v  \\
\left( \mathbf{n} \cdot \mathbf{J} w \right) \, v \\
\end{array}
\right],
}
\\
\end{array}
\end{equation}
where $\delta_{\partial \Omega}(\mathbf{x})=1 $ if $\mathbf{x} \in \partial \Omega$ and 
$\delta_{\partial \Omega}(\mathbf{x})=0 $ otherwise.
Exploiting this notation, we can rewrite \eqref{eq:weak_formulation_hyp} as
\eqref{eq:ADR_FEM_general}.
 Note that the associated  mapped problem is also of the form \eqref{eq:ADR_FEM_general}, provided that we substitute $\mathbf{c}_{\mu}, \sigma_{\mu}, f_{\mu}, z_{\rm D,\mu}$ with the corresponding definitions in \eqref{eq:mapped_coefficients_general}.

\section{Empirical Interpolation Method}
\label{sec:eim}
 
 \subsection{Review of the interpolation procedure for scalar fields}

We review the Empirical Interpolation Method  (EIM, \cite{barrault2004empirical}),  and we discuss  its extension to the approximation of vector-valued fields. Given 
the Hilbert space $\mathcal{W}$ defined over $\Omega$, 
the $Q$-dimensional linear space $\mathcal{W}_Q = {\rm span} \{ \psi_q \}_{q=1}^Q \subset \mathcal{W}$ and the points $\{  \mathbf{x}_q^{\rm i} \}_{m=1}^Q \subset \overline{\Omega}$, we define the interpolation operator  $\mathcal{I}_Q: \mathcal{W} \to \mathcal{W}_Q$ such that $\mathcal{I}_Q(v)(\mathbf{x}_q^{\rm i}) = v(\mathbf{x}_q^{\rm i})$ for $q=1,\ldots,Q$ for all $v \in \mathcal{W}$. Given the manifold $\mathcal{F} \subset \mathcal{W}$ and an integer $Q>0$, the objective of EIM is to determine an approximation space  $\mathcal{Z}_Q$ and $Q$
points $\{  \mathbf{x}_q^{\rm i} \}_{q=1}^Q$ such that $\mathcal{I}_Q(f)$ accurately approximates $f$ for all $f \in \mathcal{F}$.

Algorithm \ref{EIM} summarizes the EIM procedure as implemented in our code.
The algorithm takes as input snapshots of the manifold
$\{ f^k  \}_{k=1}^{n_{\rm train}} \subset \mathcal{F}$
and a tolerance $tol_{\rm eim}>0$, 
 and returns the functions $\{ \psi_q  \}_{q=1}^Q$,  the interpolation points $\{  \mathbf{x}_q^{\rm i} \}_{q=1}^Q$ and the matrix $\mathbf{B} \in \mathbb{R}^{Q, Q}$ such that
$\mathbf{B}_{q,q'} = \psi_q(\mathbf{x}_{q'}^{\rm i})$.
It is possible to show that the matrix $\mathbf{B}$ is lower-triangular: for this reason, online computations  can be performed in  $\mathcal{O}(Q^2)$ flops.
Note that in  \cite{barrault2004empirical} the authors resort to a strong Greedy procedure to generate $\mathcal{Z}_Q$, while here (as in several other works including \cite{chaturantabut2010nonlinear}) we resort to POD.
A thorough comparison between the two compression strategies is beyond the scope of the present work.

 \begin{algorithm}[H]                      
\caption{
Empirical Interpolation Method.}     
\label{EIM}                           


 \small
\begin{flushleft}
\begin{tabular}{| l | l |  }
\hline
\textbf{Inputs:} & $\{f^k \}_{k=1}^{n_{\rm train}}$, $tol_{\rm eim}$ \\[2mm]
\hline
\textbf{Outputs:}
&
$\{ \psi_q \}_{q=1}^Q, \mathbf{B} \in \mathbb{R}^{Q, Q}, \{ \mathbf{x}_m^{\rm i}  \}_{m=1}^M$
\\[2mm]
  \hline
\end{tabular}
\end{flushleft}  
 \normalsize 

\begin{algorithmic}[1]
 \State
 Build the POD space $\omega_1,\ldots,\omega_Q$ based on the snapshot set $\{f^k \}_{k=1}^{n_{\rm train}}$; $Q$ is chosen using \eqref{eq:POD_cardinality_selection} ($tol_{\rm pod}
 = tol_{\rm eim}$)
 \smallskip
  
 \State
 $\mathbf{x}_1^{\rm i} := {\rm arg} \max_{\mathbf{x} \in \overline{\Omega}} \, |\psi_1(\mathbf{x})|$,
 $\psi_1 :=   \frac{1}{\omega_1(\mathbf{x}_1^{\rm i})} \, \omega_1$, $\left(\mathbf{B} \right)_{1,1}= 1$
  \smallskip
 
\For{$q=2,\ldots,Q$}

 \State
$r_q = \omega_q- \mathcal{I}_{q-1} \omega_q$
  \smallskip

 \State
$\mathbf{x}_q^{\rm i} := {\rm arg} \max_{\mathbf{x} \in \overline{\Omega}} \, |r_q(\mathbf{x})|$,
$\psi_q = \frac{1}{r_q(\mathbf{x}_q^{\rm i}) } \, r_q$,
$\left(\mathbf{B} \right)_{q,q'}= \psi_q(\mathbf{x}_{q'}^{\rm i})$.
 
 \EndFor
   \smallskip

\end{algorithmic}
\end{algorithm}

\begin{remark}
\textbf{Oversampling.}
Several authors have proposed to consider non-interpolatory extensions of Algorithm \ref{EIM}: these algorithms generate a set of $Q_{\rm s}$ points $\{  \mathbf{x}_q^{\rm i} \}_{q=1}^{Q_{\rm s}}$, a basis 
$\{ \psi_q \}_{q=1}^Q$,
and the associated system $\mathbf{B}$,
 with $Q_{\rm s} = \mathfrak{o} Q$,
where $\mathfrak{o}>1$ is the oversampling ratio.
We refer to \cite{peherstorfer2018stabilizing} and the references therein for further details; see also  \cite[Algorithm 2]{maday2015parameterized} for a generalization to a broader class of "measurement" functionals .
\end{remark}

\subsection{Extension to vector-valued fields}
The EIM procedure can be extended to vector-valued fields.
We present below the non-interpolatory extension of EIM employed in this paper; the same approach has also been employed in \cite{taddei2018offline}.
We refer to \cite{tonn2011reduced,negri2015efficient} for two alternatives applicable to vector-valued fields.
Given the space $\mathcal{W}_Q = {\rm span} \{ \boldsymbol{\omega}_q \}_{q=1}^Q  \subset \mathcal{W}$ and 
the points $\{  \mathbf{x}_q^{\rm i} \}_{q=1}^Q \subset \overline{\Omega}$, 
we define the least-squares approximation operator $\boldsymbol{\mathcal{I}}_Q: \mathcal{W} \to \mathcal{W}_Q$ such that for all $\boldsymbol{v} \in \mathcal{W}$
$$
\boldsymbol{\mathcal{I}}_Q(\boldsymbol{v}) := {\rm arg} \min_{\boldsymbol{\omega} \in \mathcal{W}_Q} \, \sum_{q=1}^Q \, \| \boldsymbol{v}(\mathbf{x}_q^{\rm i}) - \boldsymbol{\omega}(\mathbf{x}_q^{\rm i})   \|_2^2.
$$
It is possible to show that $\boldsymbol{\mathcal{I}}_M$ is well-defined if and only if the matrix $\mathbf{B} \in \mathbb{R}^{QD, Q}$,
\begin{subequations}
\label{eq:EIM_vec}
\begin{equation}
\mathbf{B} = 
\left[
\begin{array}{ccc}
\boldsymbol{\omega}_1(\mathbf{x}_1^{\rm i}), & \ldots, & \boldsymbol{\omega}_Q(\mathbf{x}_1^{\rm i}) \\
& \vdots & \\
\boldsymbol{\omega}_1(\mathbf{x}_Q^{\rm i}), & \ldots, & \boldsymbol{\omega}_Q(\mathbf{x}_Q^{\rm i}) \\
\end{array}
\right]
\end{equation}
is full-rank. In this case, we find that   $\boldsymbol{\mathcal{I}}_Q$ can be efficiently computed as
\begin{equation}
\boldsymbol{\mathcal{I}}_Q(\boldsymbol{v}) = \sum_{q=1}^Q \,
\left(  \boldsymbol{\alpha}(\boldsymbol{v}) \right)_q \, \boldsymbol{\omega}_q,
\quad
\boldsymbol{\alpha}(\boldsymbol{v}) = 
\mathbf{B}^{\dagger}
\left[
\begin{array}{c}
\boldsymbol{v}(\mathbf{x}_1^{\rm i}) \\
\vdots \\
\boldsymbol{v}(\mathbf{x}_Q^{\rm i}) \\
 \end{array}
\right]
\end{equation}
for any $\boldsymbol{v} \in \mathcal{W}$, 
where $\mathbf{B}^{\dagger} = (\mathbf{B}^T \mathbf{B})^{-1} \mathbf{B}^T$ denotes the Moore-Penrose pseudo-inverse of $\mathbf{B}$. 
\end{subequations} 

Algorithm \ref{EIM_vec} summarizes the procedure employed to compute $\mathcal{W}_Q$, $\{ \mathbf{x}_q^{\rm i}  \}_{q=1}^Q$ and the matrix $\mathbf{B}$.
We observe that for scalar fields the procedure reduces to the one outlined in Algorithm \ref{EIM}. 
We further observe that online computational cost scales with $\mathcal{O}(D Q^2)$, provided that 
$\mathbf{B}^{\dagger}$ is computed offline.

 \begin{algorithm}[H]                      
\caption{
Empirical Interpolation Method for vector-valued fields.}     
\label{EIM_vec}                           


 \small
\begin{flushleft}
\begin{tabular}{| l | l |  }
\hline
\textbf{Inputs:} & $\{\boldsymbol{f}^k \}_{k=1}^{n_{\rm train}}$, $tol_{\rm eim}$ \\[2mm]
\hline
\textbf{Outputs:}
&
$\{ \boldsymbol{\omega}_q \}_{q=1}^Q, \mathbf{B}^{\dagger} \in \mathbb{R}^{Q, Q}, \{ \mathbf{x}_q^{\rm i}  \}_{q=1}^Q$
\\[2mm]
  \hline
\end{tabular}
\end{flushleft}  
 \normalsize 

\begin{algorithmic}[1]
 \State
 Build the POD space $\boldsymbol{\omega}_1,\ldots,\boldsymbol{\omega}_Q$ based on the snapshot set $\{ \boldsymbol{f}^k \}_{k=1}^{n_{\rm train}}$;
$Q$ is chosen using \eqref{eq:POD_cardinality_selection} ($tol_{\rm pod}  = tol_{\rm eim}$)
 \smallskip
  
 \State
 Set 
 $\mathbf{x}_1^{\rm i} := {\rm arg} \max_{\mathbf{x} \in \overline{\Omega}} \, \|\omega_1(\mathbf{x}) \|_2$,
and  $\mathbf{B}_{Q=1}$ using \eqref{eq:EIM_vec}.
  \smallskip
 
\For{$q=2,\ldots,Q$}

\State
$\boldsymbol{r}_q = \boldsymbol{\omega}_q - \boldsymbol{\mathcal{I}}_{q-1} \boldsymbol{\omega}_q$
  \smallskip

\State
Set 
$\mathbf{x}_q^{\rm i} := {\rm arg} \max_{\mathbf{x} \in \overline{\Omega}} \, \|  \boldsymbol{r}_q(\mathbf{x}) \|_2$, and update
 $\mathbf{B}_{Q=q}$ using \eqref{eq:EIM_vec}. 
 \EndFor
    \smallskip

\State
Compute $\mathbf{B}^{\dagger} = (\mathbf{B}^T \mathbf{B})^{-1} \mathbf{B}$.
\end{algorithmic}

\end{algorithm}

\section{Radial Basis Function maps for geometry reduction}
\label{sec:RBF_laplace}

\subsection{RBF formulation}

We illustrate how to apply RBF approximations to build the mapping $\boldsymbol{\Phi}_{\mu}$ for a given $\mu \in \mathcal{P}$. Given an even integer $N_{\rm bnd}>0$, the parameterization of $\partial \Omega_{\rm in}$ 
 $\boldsymbol{\gamma}_{\rm in}:[0,1) \to \partial \Omega_{\rm in}$,
$\boldsymbol{\gamma}_{\rm in}(t) = [\cos(2 \pi t), \sin(2 \pi t) ]^T$, and the parameterization  of $\partial \Omega_{\rm box}$
$\boldsymbol{\gamma}_{\rm box}:[0,1) \to \partial \Omega_{\rm box}$, we define the control points
$$
\{ \mathbf{X}_i  \}_{i=1}^{N_{\rm bnd}} \, := \, 
\left\{
\boldsymbol{\gamma}_{\rm in}(t^1),\ldots,
\boldsymbol{\gamma}_{\rm in}(t^{N_{\rm bnd}/2}),
\boldsymbol{\gamma}_{\rm box}(t^1),\ldots,
\boldsymbol{\gamma}_{\rm box}(t^{N_{\rm bnd}/2})
\right\},
$$
and the displaced control points 
$$
\{ \mathbf{x}_i^{\mu}  \}_{i=1}^{N_{\rm bnd}} \, := \, 
\left\{
\boldsymbol{\gamma}_{\rm in,\mu}(2 \pi t^1),\ldots,
\boldsymbol{\gamma}_{\rm in,\mu}(2 \pi t^{N_{\rm bnd}/2}),
\boldsymbol{\gamma}_{\rm box}(t^1),\ldots,
\boldsymbol{\gamma}_{\rm box}(t^{N_{\rm bnd}/2})
\right\},
$$
where $0 \leq t^1,\ldots \leq t^{N_{\rm bnd}/2} < 1$. Then, we introduce the  kernel $\phi: \mathbb{R}^+ \to \mathbb{R}$, where $\lambda>0$ is the Kernel width, and the space $\mathbb{P}^1$ of linear functions from $\mathbb{R}^2$ to $\mathbb{R}^2$. Finally, we define the mapping 
$\boldsymbol{\Phi}_{\mu} := \mathbf{p}_{\mu} + \boldsymbol{\omega}_{\mu}$ 
where $(\mathbf{p}_{\mu} , \boldsymbol{\omega}_{\mu})$ is the solution to  the optimization problem
\begin{equation}
\label{eq:RBF_mapping}
\min_{( \mathbf{p}, \boldsymbol{\omega}) \in \mathbb{P}^1 \times \mathcal{N}_{\phi}  }  \, \xi \|  \boldsymbol{\omega}  \|_{ \mathcal{N}_{\phi}   }^2 \, + \,
\sum_{i=1}^{N_{\rm bnd}} \, \big\|  
\mathbf{p}(\mathbf{X}_i) \,  + \,     
\boldsymbol{\omega}(\mathbf{X}_i)
\,  - \, \mathbf{x}_i^{\mu}  \big\|_2^2.
\end{equation}
Here, $\xi>0$ is a regularization parameter, while $\mathcal{N}_{\phi}$  is the native Hilbert space associated with the kernel $\phi$.
We anticipate that in the numerical results we consider the 
Gaussian kernel   with kernel width $\lambda>0$, 
$\phi(r) = e^{-\lambda r^2}$ and we resort to hold-out ($80\% - 20\%$) validation to choose $\lambda$ and $\xi$. 
It is possible to show that the optimal $\boldsymbol{\omega}_{\mu}$ is of the form 
$\boldsymbol{\omega}_{\mu}(\mathbf{X}) = \sum_{i=1}^{N_{\rm bnd}} \mathbf{a}_i \, \phi \left(  \| \mathbf{X} - \mathbf{X}_i  \|_2 \right)$; as a result, solutions to \eqref{eq:RBF_mapping} involve the solution to a  linear system  of size $d (N_{\rm bnd} + 1) + d^2$.
Note that in \cite{manzoni2012model} the authors consider the pure interpolation problem, which corresponds to taking the limit $\xi \to 0^+$ in \eqref{eq:RBF_mapping} (see \cite[Proposition 2.10]{taddei2017adaptive}).

We observe that $\boldsymbol{\Phi}_{\mu}$ defined in \eqref{eq:RBF_mapping} is not guaranteed to be bijective for large deformations: this issue is shared by many approaches referenced in the introduction\footnote{
To address this issue, in the related framework of mesh deformation, several authors (see, e.g., \cite{froehle2015nonlinear}) have proposed to resort to nonlinear elasticity extensions: clearly, resorting to a nonlinear extension increases the overall computational cost.
}. Furthermore, computational cost scales with $\mathcal{O}(N_{\rm bnd}^3)$: as $N_{\rm bnd}$ increases, the computational overhead associated with the mapping process might be the dominant online cost.
On the other hand,
for the approach presented in this paper 
(i) the size of the expansion does not depend on the number of control points, and 
(ii) the mapping is guaranteed to be globally invertible for all $\mu$ in the training set $\{ \mu^k \}_{k=1}^{n_{\rm train}} \subset \mathcal{P}$.

\subsection{Numerical results}

In Figure \ref{fig:RBF}, we compare performance of the RBF map obtained solving \eqref{eq:RBF_mapping} with performance of the map obtained solving \eqref{eq:optimization_statement} for several values of $M_{\rm hf}$, and for a given $\mu \in \mathcal{P}$. 
Figure \ref{fig:RBF}(a) shows  the behavior of the out-of-sample error $E^{\rm geo,out}$ (cf. \eqref{eq:Egeo}) with respect to the number of control points $N_{\rm bnd}$; Figure \ref{fig:RBF}(b) shows the behavior of the minimum Jacobian determinant over $\Omega$ with $N_{\rm bnd}$.
Control points for RBF are chosen on both boundaries $\partial \Omega_{\rm box}, \partial \Omega_{\rm in}$ as described above; on the other hand, in our procedure,
in  \eqref{eq:optimization_statement}, 
 we consider  $N_{\rm bnd}=10^3$ points on $\partial \Omega_{\rm in}$. We choose $\epsilon=0.1$ for all tests, while we observe that higher values of $\xi$ should be considered for small $M_{\rm hf}$ to avoid overfitting: here, we set $\xi=10^{-1}$ for $M_{\rm hf} \geq 2 \cdot 4^2$ and 
$\xi=10$ for $M_{\rm hf} <  2 \cdot 4^2$.
We observe that convergence of the  RBF mapping is relatively slow: this is due to the lack of regularity of $\partial \Omega_{\rm box}$. We further observe that for small values of $M_{\rm hf}=N_{\rm bnd}$ the RBF mapping might be singular. On the other hand, our approach is guaranteed to lead to bijective maps for all choices of $M_{\rm hf}$.

\begin{figure}[h!]
\centering
 \subfloat[] {\includegraphics[width=0.45\textwidth]
 {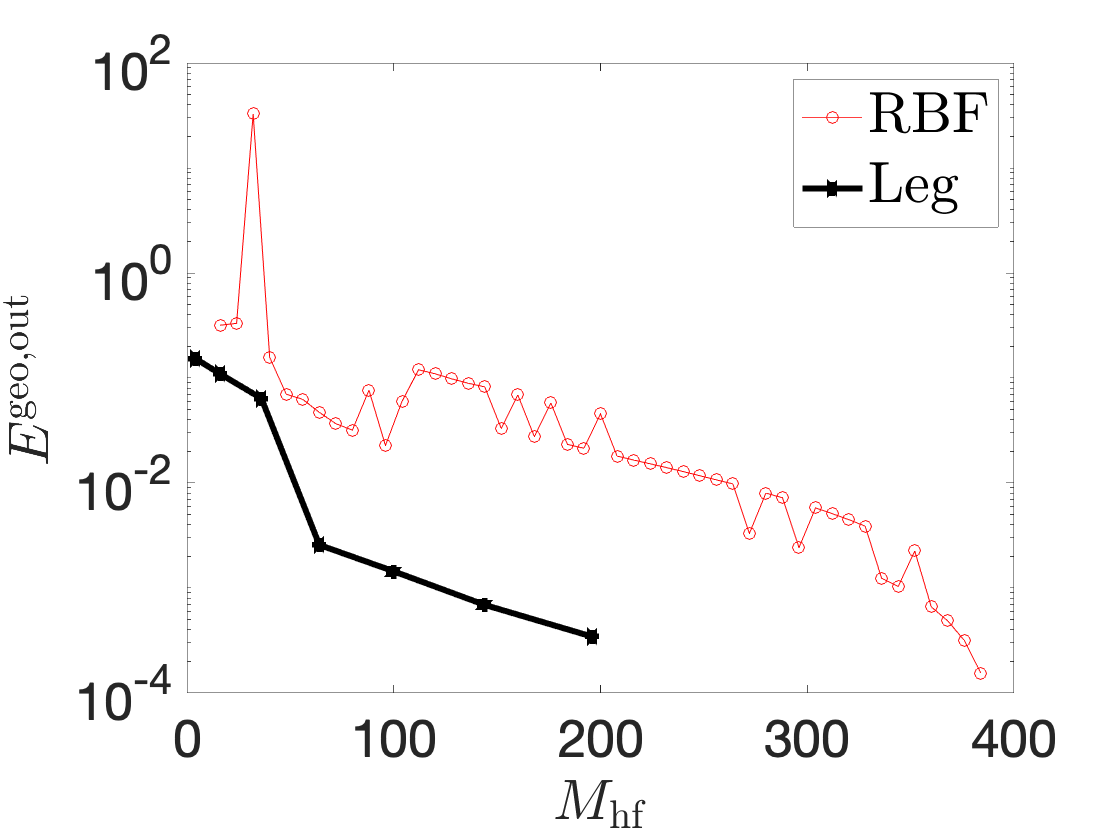}}  
    ~ ~
\subfloat[] {\includegraphics[width=0.45\textwidth]
 {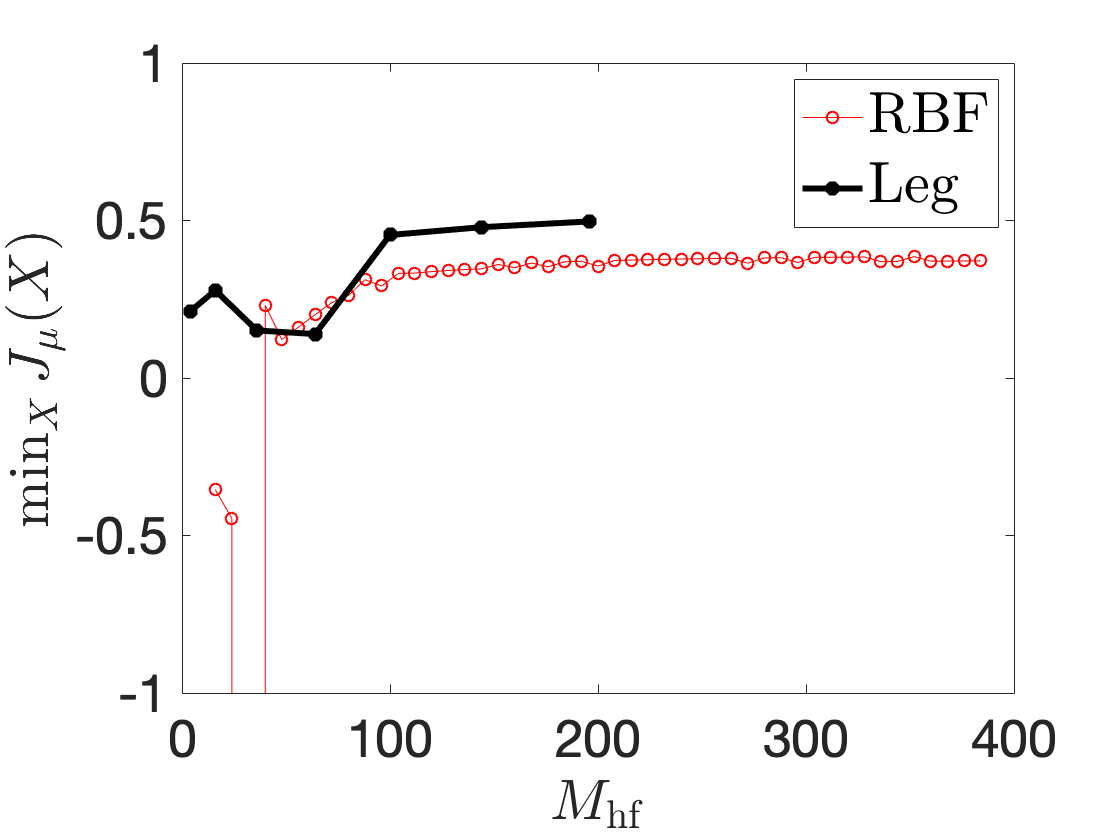}}  
 
\caption{Laplace's equation in parameterized domain; performance of RBF  and Legendre-based  mappings for $\mu = [0.1826, 0.2918, 0.4900] \in \mathcal{P}$.  
(a): behavior of $E^{\rm geo,out}$ \eqref{eq:Egeo} with   $M_{\rm hf}$. 
(b): behavior of $\min_{\mathbf{X}} \mathfrak{J}_{\mu}(\mathbf{X})$ with   $M_{\rm hf}$.}
 \label{fig:RBF}
\end{figure}

 \bibliographystyle{plain}
 \small

 \bibliography{all_refs}
 
\end{document}